\documentclass{amsart}
\usepackage{amssymb} 
\usepackage[cmtip,all]{xy}
\newdir{ >}{{}*!/-5pt/@{>}} 
\usepackage{graphicx} 
\usepackage[alphabetic,nobysame]{amsrefs}

\title[Algebraic $K$-theory of elliptic cohomology]
	{Algebraic $K$-theory \\ of elliptic cohomology}
\author[G. Angelini-Knoll, Ch. Ausoni, D.L. Culver,
	E. H{\"o}ning and J. Rognes]
	{Gabriel Angelini-Knoll, Christian Ausoni, Dominic Leon Culver, \\
	Eva H{\"o}ning and John Rognes}
\address{Universit{\'e} Sorbonne Paris Nord, LAGA, CNRS, UMR 7539, F-93430, Villetaneuse, France}
\email{angelini-knoll@math.univ-paris13.fr}
\address{Universit{\'e} Sorbonne Paris Nord, LAGA, CNRS, UMR 7539, F-93430, Villetaneuse, France}
\email{ausoni@math.univ-paris13.fr}
\address{Max Planck Institute for Mathematics, Bonn, Germany}
\email{dominic.culver@gmail.com}
\address{Department of Mathematics, Radboud University, Nijmegen, The Netherlands}
\email{hoening@science.ru.nl}
\address{Department of Mathematics, University of Oslo, P.O.~Box 1053, Blindern, NO-0316 Oslo, Norway}
\email{rognes@math.uio.no}

\date{October 24th 2023}

\newtheorem{theorem}{Theorem}[section]
\newtheorem{proposition}[theorem]{Proposition}
\newtheorem{lemma}[theorem]{Lemma}
\newtheorem{corollary}[theorem]{Corollary}
\theoremstyle{definition}
\newtheorem{definition}[theorem]{Definition}
\newtheorem{notation}[theorem]{Notation}
\theoremstyle{remark}
\newtheorem{remark}[theorem]{Remark}

\numberwithin{equation}{section}
\numberwithin{figure}{section}

\DeclareMathOperator{\can}{can}
\DeclareMathOperator{\cok}{cok}
\DeclareMathOperator{\Ext}{Ext}
\DeclareMathOperator{\Gal}{Gal}
\DeclareMathOperator{\im}{im}
\newcommand{\bF}{\mathbb{F}}
\newcommand{\bQ}{\mathbb{Q}}
\newcommand{\bT}{\mathbb{T}}
\newcommand{\bZ}{\mathbb{Z}}
\newcommand{\cA}{\mathcal{A}}
\newcommand{\cC}{\mathcal{C}}
\newcommand{\cP}{\mathcal{P}}
\newcommand{\cQ}{\mathcal{Q}}
\newcommand{\ds}{\displaystyle}
\newcommand{\<}{\langle}
\newcommand{\longfrom}{\longleftarrow}
\newcommand{\longto}{\longrightarrow}
\newcommand{\wA}{\widetilde{A}}
\newcommand{\wB}{\widetilde{B}}
\newcommand{\wC}{\widetilde{C}}
\newcommand{\wD}{\widetilde{D}}
\newcommand{\wET}{\widetilde{E\bT}}
\renewcommand{\:}{\colon}
\renewcommand{\>}{\rangle}

\begin{document}

\begin{abstract}
We calculate the mod~$(p, v_1, v_2)$ homotopy $V(2)_* TC(BP\<2\>)$
of the topological cyclic homology of the truncated Brown--Peterson
spectrum $BP\<2\>$, at all primes $p\ge7$, and show that it is a finitely
generated and free $\bF_p[v_3]$-module on $12p+4$ generators in explicit
degrees within the range $-1 \le * \le 2p^3+2p^2+2p-3$.  At these
primes $BP\<2\>$ is a form of elliptic cohomology, and our result also
determines the mod~$(p, v_1, v_2)$ homotopy of its algebraic $K$-theory.
Our computation is the first that exhibits chromatic redshift from pure
$v_2$-periodicity to pure $v_3$-periodicity in a precise quantitative
manner.
\end{abstract}


\maketitle

\tableofcontents

\section{Introduction}

Let $p$ be a prime, let $V(n)$ denote a Smith--Toda complex with $BP_*
V(n) = BP_*/(p, \dots, v_n)$, and let $BP\<n\>$ with $\pi_* BP\<n\>
= \bZ_{(p)}[v_1, \dots, v_n]$ denote a truncated Brown--Peterson
spectrum equipped with the $E_3$ $BP$-algebra structure of
Hahn--Wilson~\cite{HW22}*{Thm.~A}.  Let $P(x) = \bF_p[x]$ and $E(x)$
denote the polynomial and exterior $\bF_p$-algebras on a generator~$x$,
and let $\bF_p\{x\}$ denote the $\bF_p$-module generated by~$x$.

In this paper we confirm the quantitative form of the chromatic
redshift conjecture of~\cite{Rog00}*{p.~8} in the case of $BP\<2\>$
at~$p\ge7$, showing that $V(2)_* TC(BP\<2\>)$ is finitely generated
and free as a $P(v_3)$-module.  Hence the topological cyclic homology
functor takes the ``pure fp-type~$2$'' ring spectrum $BP\<2\>$ with
$V(1)_* BP\<2\>$ finitely generated and free as a $P(v_2)$-module,
to a ``pure fp-type~$3$'' ring spectrum $TC(BP\<2\>)$ with $V(2)_*
TC(BP\<2\>)$ finitely generated and free as a $P(v_3)$-module, dilating
the wavelength of periodicity from $|v_2| = 2p^2-2$ to $|v_3| = 2p^3-2$.
\footnote{See also Remark~\ref{rem:recent} regarding the recent resolution
by Burklund, Schlank and~Yuan~\cite{BSY} of the (weaker) qualitative
form of the redshift conjecture, in the case of $E_\infty$ ring spectra.}

The precise statement follows.

\begin{theorem} \label{thm:mainTCBP2}
Let $p\ge7$.  There is a preferred isomorphism
\begin{align*}
V(2)_* TC(BP\<2\>) &\cong P(v_3)
        \otimes E(\partial, \lambda_1, \lambda_2, \lambda_3) \\
        &\qquad\oplus P(v_3) \otimes E(\lambda_2, \lambda_3)
        \otimes \bF_p\{ \Xi_{1,d} \mid 0 < d < p \} \\
        &\qquad\oplus P(v_3) \otimes E(\lambda_1, \lambda_3)
        \otimes \bF_p\{ \Xi_{2,d} \mid 0 < d < p \} \\
        &\qquad\oplus P(v_3) \otimes E(\lambda_1, \lambda_2)
        \otimes \bF_p\{ \Xi_{3,d} \mid 0 < d < p \}
\end{align*}
of $P(v_3) \otimes E(\lambda_1, \lambda_2, \lambda_3)$-modules.  This is a
finitely generated and free $P(v_3)$-module on $12p+4$ explicit generators
in degrees $-1 \le * \le 2p^3+2p^2+2p-3$.
\end{theorem}

The close relation between algebraic $K$-theory and
topological cyclic homology for $p$-complete ring spectra leads to the
following application, cf.~Theorem~\ref{thm:KBP2p}.

\begin{theorem} \label{thm:mainKBP2p}
Let $p\ge7$.  There is an exact sequence of $P(v_3) \otimes E(\lambda_1,
\lambda_2, \lambda_3)$-modules
\begin{multline*}
0 \to \Sigma^{-2} \bF_p\{\bar\tau_1, \bar\tau_2, \bar\tau_1\bar\tau_2\}
        \longto V(2)_* K(BP\<2\>_p) \\
        \overset{trc_*}\longto V(2)_* TC(BP\<2\>)
        \longto \Sigma^{-1} \bF_p\{1\} \to 0
\end{multline*}
with $|\bar\tau_i| = 2p^i-1$.  The localization homomorphism
$$
V(2)_* K(BP\<2\>_p) \longto v_3^{-1} V(2)_* K(BP\<2\>_p)
$$
is an isomorphism in degrees $* \ge 2p^2+2p$, and the target is a
finitely generated and free $P(v_3^{\pm1})$-module on $12p+4$ generators.
\end{theorem}

The proven Lichtenbaum--Quillen conjecture for $K(\bZ_{(p)})$ and
$K(\bZ_p)$ also lets us pass from the $p$-complete version to the
$p$-local version of~$BP\<2\>$, cf.~Theorem~\ref{thm:KBP2}.

\begin{theorem} \label{thm:mainKBP2}
Let $p\ge7$.  The $p$-completion map induces a $(2p^2+2p-2)$-coconnected
homomorphism
$$
V(2)_* K(BP\<2\>) \overset{\kappa_*}\longto V(2)_* K(BP\<2\>_p) \,.
$$
The localization homomorphism
$$
V(2)_* K(BP\<2\>) \longto v_3^{-1} V(2)_* K(BP\<2\>)
$$
is an isomorphism in degrees $* \ge 2p^2+2p$, and the target is a
finitely generated and free $P(v_3^{\pm1})$-module on $12p+4$ generators.
\end{theorem}

\begin{remark}
An alternative title for this paper could be ``Topological cyclic homology
modulo~$p$, $v_1$ and~$v_2$ of the second truncated Brown--Peterson
spectrum''.  In earlier work~\cite{AR02} we referred to the calculation
of $V(1)_* TC(BP\<1\>)$ as (an essential step toward) a calculation of
the ``algebraic $K$-theory of topological $K$-theory''.  The relation
between $BP\<1\>$ and topological $K$-theory is analogous to that between
$BP\<2\>$ and elliptic cohomology, so we hope the reader grants us the
poetic license presumed by our choice of title.

The $v_1$- and $v_2$-periodic families in $\pi_* V(0)$ and $\pi_* V(1)$,
respectively, are related to the well-known $\alpha$-family visible to
topological $K$-theory and the fairly well understood $\beta$-family
visible to elliptic cohomology.  The $v_3$-periodic families emerging
from our calculation are related to the third family of Greek letter
elements, the $\gamma$-family, which is less well understood, and for
which there is currently no known detecting cohomology theory with a
geometric interpretation of the cohomology classes.  Our result suggests
that algebraic $K$-theory of elliptic cohomology may be such a detecting
cohomology theory.
\end{remark}

We now explain Theorem~\ref{thm:mainTCBP2} in more detail.  For each $E_3$
ring spectrum~$B$ we have maps of $E_2$ ring spectra
$$
S \longto K(B) \overset{trc}\longto TC(B)
	\overset{\pi}\longto THH(B)^{h\bT}
	\longto THH(B)
$$
from the sphere spectrum to the topological Hochschild homology~$THH(B)$
of~$B$, via its algebraic $K$-theory~$K(B)$, topological cyclic
homology~$TC(B)$ and the~$\bT$-homotopy fixed points of~$THH(B)$.
For $p\ge7$ the Smith--Toda spectrum~$V(2)$ exists as a homotopy
commutative and associative ring spectrum, with a periodic class
$v_3 \in \pi_{2p^3-2} V(2)$.
In Section~\ref{sec:THH} we recall that
$$
V(2)_* THH(BP\<2\>) = E(\lambda_1, \lambda_2, \lambda_3) \otimes P(\mu) \,,
$$
with $|\lambda_i| = 2p^i-1$ for $i \in \{1,2,3\}$ and $|\mu| = 2p^3$.
In Sections~\ref{sec:power} and~\ref{sec:V0V1classes} we use $E_2$ ring
spectrum power operations to show that the $THH$-classes~$\lambda_i$
lift to $K$-theory classes $\lambda_i^K \in V(2)_* K(BP\<2\>)$, with
$tr(\lambda_i^K) = \lambda_i$.  We also write $\lambda_i$ for their
images in $V(2)_* TC(BP\<2\>)$ and $V(2)_* THH(BP\<2\>)^{h\bT}$.  In
Sections~\ref{sec:CpTatespseq} through~\ref{sec:TTatespseq} we determine
the structure of the $\bT$-homotopy fixed point spectral sequence
\begin{align*}
E^2(\bT) &= H^{-*}(\bT, V(2)_* THH(BP\<2\>)) \\
	&= P(t) \otimes E(\lambda_1, \lambda_2, \lambda_3) \otimes P(\mu) \\
	&\Longrightarrow V(2)_* THH(BP\<2\>)^{h\bT} \,.
\end{align*}
The image of $v_3$ in $V(2)_* THH(BP\<2\>)^{h\bT}$ is detected by~$t\mu$.
The homotopy classes $\Xi_{i,d} \in V(2)_* TC(BP\<2\>)$ for $i \in
\{1,2,3\}$ and $0<d<p$ are constructed in Section~\ref{sec:TC} so that
$$
\pi(\Xi_{i,d}) = \sum_{n=0}^\infty \xi_{i+3n,d}
$$
in $V(2)_* THH(BP\<2\>)^{h\bT}$.  In this convergent series, each
$\xi_{k,d}$ is a specific $V(2)$-homotopy element detected by a class
$$
x_{k,d} = t^{\frac{d}{p} r(k)} \lambda_{[k]} \mu^{\frac{d}{p} r(k-3)}
	\in E^\infty(\bT) \,.
$$
Here $[k] \in \{1,2,3\}$ satisfies $k \equiv [k] \mod 3$, and $r(k) =
p^k + p^{k-3} + \dots + p^{[k]}$ for $k\ge1$.  In particular
$$
\pi(\Xi_{i,d}) \ ,\ \xi_{i,d} \in \{t^{dp^{i-1}} \lambda_i\}
$$
for $i \in \{1,2,3\}$ are both detected by $t^{dp^{i-1}} \lambda_i$
in $E^\infty(\bT)$.
Letting $\partial$ denote the generator of
$V(2)_{-1} TC(BP\<2\>)$, and noting that $\lambda_i \cdot \Xi_{i,d} = 0$
for each $i$ and~$d$, this concludes our specification of the notation
in Theorem~\ref{thm:mainTCBP2}, which appears as
Theorem~\ref{thm:TCBP2} in the body of the text.
One way to summarize the grading of the module generators is
to say that the Poincar{\'e} series of $V(3)_* TC(BP\<2\>)$ is
\begin{align*}
(1+x^{-1}) &(1+x^{2p-1}) (1+x^{2p^2-1}) (1+x^{2p^3-1}) \\
&+ (1+x^{2p^2-1}) (1+x^{2p^3-1}) (x+x^3+\dots+x^{2p-3}) \\
&+ (1+x^{2p-1}) (1+x^{2p^3-1}) (x^{2p-1}+x^{4p-1}+\dots+x^{2p^2-2p-1}) \\
&+ (1+x^{2p-1}) (1+x^{2p^2-1}) (x^{2p^2-1}+x^{4p^2-1}+\dots+x^{2p^3-2p^2-1})
	\,.
\end{align*}

\begin{remark} \label{rem:descent}
The seminal calculation in this field was made by B{\"o}kstedt and
Madsen~\cite{BM94}, \cite{BM95}.  For the Eilenberg--MacLane spectrum
$BP\<0\> = H\bZ_{(p)}$ at $p\ge3$ they established an isomorphism
\begin{align*}
V(0)_* TC(\bZ_{(p)}) &\cong P(v_1) \otimes E(\partial, \lambda_1) \\
	&\qquad \oplus P(v_1) \otimes \bF_p\{ \Xi_{1,d} \mid 0<d<p \}
\end{align*}
of free $P(v_1)$-modules of rank~$p+3$, where $\Xi_{1,d}$ is detected
by $t^d \lambda_1$.  The (then unproven) Lichtenbaum--Quillen conjecture
for $K(\bQ_p)$ could be deduced from this, showing that the natural
homomorphism
$$
V(0)_* K(\bQ_p) \longto V(0)_* K(\bar\bQ_p)^{hG_{\bQ_p}}
$$
is $0$-coconnected, where $G_{\bQ_p} = \Gal(\bar\bQ_p/\bQ_p)$ is the
absolute Galois group.  In particular, the $P(v_1)$-module generators
of $V(0)_* TC(\bZ_{(p)})$ correspond in a precise manner to a basis for
the Galois cohomology groups in the descent spectral sequence
$$
E^2_{-s,t} = H_{\Gal}^s(\bQ_p; \bF_p(t/2))
	\Longrightarrow V(0)_{-s+t} K(\bar\bQ_p)^{hG_{\bQ_p}} \,.
$$
The fact that $V(0)_* TC(\bZ_{(p)})$ is $P(v_1)$-torsion free is thus a
reflection of Suslin's theorem~\cite{Sus84} that $V(0)_* K(\bar\bQ_p)
\cong V(0)_* ku = \bF_p[u]$ is $P(v_1)$-torsion free, and the finite
generation and grading of $V(0)_* TC(\bZ_{(p)})$ corresponds to precise
information about the Galois (or motivic) cohomology of~$\bQ_p$.

For the Adams summand $BP\<1\> = \ell$ of $ku_{(p)}$ at $p\ge5$,
two of the present authors~\cite{AR02} thereafter obtained an
isomorphism
\begin{align*}
V(1)_* TC(\ell) &\cong P(v_2) \otimes E(\partial, \lambda_1, \lambda_2) \\
	&\qquad \oplus P(v_2) \otimes E(\lambda_2) \otimes \bF_p\{ \Xi_{1,d} \mid 0<d<p \} \\
	&\qquad \oplus P(v_2) \otimes E(\lambda_1) \otimes \bF_p\{ \Xi_{2,d} \mid 0<d<p \}
\end{align*}
of free $P(v_2)$-modules of rank~$4p+4$, where $\Xi_{1,d}$ is detected
by~$t^d \lambda_1$ and $\Xi_{2,d}$ is detected by~$t^{dp} \lambda_2$.
Moreover, one of us~\cite{Aus10} proceeded to calculate $V(1)_* TC(ku)$,
and showed~\cite{Aus05} that
$$
V(1)_* K(\ell_p) \longto V(1)_* K(ku_p)^{h\Delta}
$$
is an isomorphism.  Another one of us~\cite{Rog14}*{\S5} viewed this as
computational evidence for the existence of a descent spectral sequence,
converging to $V(1)_* K(\ell_p)$, from a form of motivic cohomology
defined for $E_\infty$ ring spectra such as~$\ell_p$.  The fact that
$V(1)_* TC(\ell)$ is $P(v_2)$-torsion free would then reflect an analogue
of Suslin's theorem, and the finite generation and grading of $V(1)_*
TC(\ell)$ would correspond to specific information about this spectrally
defined motivic cohomology.
\footnote{See also Remark~\ref{rem:recent} regarding the recent discovery
by Hahn, Raksit and Wilson~\cite{HRW} of such a cohomology theory,
in the case of $E_\infty$ ring spectra.}

Our present conclusions about $V(2)_* TC(BP\<2\>)$ and $V(2)_*
K(BP\<2\>_p)$ as $P(v_3)$-modules continue this pattern, and further
suggest the existence of a descent spectral sequence from a motivic
cohomology defined for less commutative ring spectra, such as the $E_3$
ring spectrum $BP\<2\>_p$.  If so, Theorem~\ref{thm:mainTCBP2} provides
information about these (at the time of writing, hypothetical) motivic
cohomology groups.
\end{remark}

\begin{remark}
Our calculations in $V(2)$-homotopy involve the homotopy element $v_3 \in
\pi_{2p^3-2} V(2)$ and its $v_2$-Bockstein image $i_2 j_2(v_3) \in
\pi_{2p^3-2p^2-1} V(2)$, closely related to the first element~$\gamma_1
\in \pi_{2p^3-2p^2-2p-1} S$ in the third Greek letter family.  To make
a similar computation of $V(3)_* TC(BP\<3\>)$ as a $P(v_4)$-module would
require knowing the existence of a homotopy element $v_4 \in \pi_{2p^4-2}
V(3)$, mapping to the class with the same name in $BP_* V(3) = BP_*/(p,
\dots, v_3)$.  The existence of~$v_4$ is presently not known for any
prime~$p$, cf.~\cite{Rav04}*{\S5.6, (5.6.13)}.  Conceivably, a calculation could
be made of $V_* TC(BP\<3\>)$ as a $P(w)$-module for another type~$4$
finite ring spectrum~$V$, with $v_4$ self map $w \: \Sigma^d V \to V$.
Something similar was carried out for the Eilenberg--MacLane spectrum
$BP\<0\> = H\bZ_{(2)}$ at $p=2$ in~\cite{Rog99}, calculating $(S/2)_*
TC(\bZ_{(2)})$ and~$(S/4)_* TC(\bZ_{(2)})$ in tandem.
\end{remark}

\begin{remark} \label{rem:telescopic}
Let $T(3) = v_3^{-1} V(2)$ be the telescopic localization of the
type~$3$ complex~$V(2)$, and let $V(3)$ be the mapping cone of $v_3 \:
\Sigma^{2p^2-2} V(2) \to V(2)$.  The three theorems above imply that
$$
T(3)_* TC(BP\<2\>) \cong T(3)_* K(BP\<2\>_p) \cong T(3)_* K(BP\<2\>)
$$
are all nontrivial $P(v_3^{\pm1})$-modules, so that the Bousfield
$T(3)$-localizations
$$
L_{T(3)} TC(BP\<2\>) \simeq L_{T(3)} K(BP\<2\>_p) \simeq L_{T(3)} K(BP\<2\>)
$$
are all nontrivial spectra.  Moreover, the graded abelian groups
$$
V(3)_* TC(BP\<2\>) \longleftarrow V(3)_* K(BP\<2\>_p)
	\longleftarrow V(3)_* K(BP\<2\>)
$$
are all finite, so
$$
TC(BP\<2\>)_p \longleftarrow K(BP\<2\>_p)_p
	\longleftarrow K(BP\<2\>)_p
$$
are all of fp-type~$3$ in the sense of~\cite{MR99}.  These qualitative
statements confirm a weaker form of the chromatic redshift conjecture
for $BP\<2\>$, roughly as formulated in~\cite{AR08}*{Conj.~1.3},
but do not contain the information that $V(2)_* TC(BP\<2\>)$ is free
over~$P(v_3)$, i.e., that $TC(BP\<2\>)$ is of ``pure fp-type~$3$'' in
the sense of~\cite{Rog00}, nor the quantitative information about its
precise rank and generating basis.

In groundbreaking work, Hahn and Wilson~\cite{HW22}*{Thm.~B} confirmed the
qualitative form of the chromatic redshift conjecture for all~$BP\<n\>$,
at all primes~$p$.
However, as outlined in Remark~\ref{rem:descent}, we take the view that
the precise $P(w)$-module structure of $V_* TC(BP\<n\>)$, where $V$ is
some type~$n+1$ finite complex with $v_{n+1}$ self map~$w \: \Sigma^d
V \to V$, will be an essential ingredient of an understanding of it and
$V_* K(BP\<n\>_p)$ as being obtained by descent from a form of motivic
cohomology for ring spectra.
\end{remark}

\begin{remark}
The authors of~\cite{AR02} had outlined a calculation of $V(n)_*
TC(BP\<n\>)$ as a $P(v_{n+1})$-module, under the strong hypotheses that
$V(n)$ exists as a ring spectrum (with a homotopy element~$v_{n+1}$)
and that $BP\<n\>$ admits an $E_\infty$ ring spectrum structure.
As in the case $n=1$, the sketched argument used a homotopy Cartan
formula for $E_\infty$ power operations, and was carried out in the
range of degrees where the comparison homomorphism $\hat\Gamma_{1*} \:
V(n)_* THH(BP\<n\>) \to V(n)_* THH(BP\<n\>)^{tC_p}$ is an isomorphism.
When $n=2$ and $p\ge7$, this homomorphism is $(2p^2+2p-3)$-coconnected,
as we show in Theorem~\ref{thm:CpTate}, so that the calculation would
determine $V(2)_* TC(BP\<2\>)$ for $* > 2p^2+2p-3$.

There is a $(2p^2-2)$-connected map $BP\<2\> \to BP\<1\>$
inducing a $(2p^2-1)$-connected map $V(2)_* TC(BP\<2\>) \to
V(2)_* TC(BP\<1\>)$, cf.~\cite{BM94}*{Prop.~10.9}, \cite{Dun97}
and Proposition~\ref{prop:TCBP2-to-TCBP1}.  Hence the known calculation of
$V(1)_* TC(BP\<1\>)$ does account for $V(2)_* TC(BP\<2\>)$ in degrees $*
< 2p^2-1$.  This leaves a gap in degrees $2p^2-1 \le * \le 2p^2+2p-3$,
where the traditional arguments do not determine $V(2)_* TC(BP\<2\>)$.
(This is a new phenomenon for $n\ge2$; there is no such gap for $n
\in \{0,1\}$.)

Around the year 2000 it was only known that $BP\<n\>$ could be realized
as an $E_1$ ring spectrum~\cite{BJ02}*{Cor.~3.5}, so the calculations
were hypothetical, even for $n=2$ and $p\ge7$.  With the much more
recent Hahn--Wilson construction of an $E_3$ ring structure on~$BP\<n\>$
it has finally become possible to carry out most of the original program,
as we show in this paper.  The lower order of commutativity has, however,
required us to also develop a homotopy Cartan formula for certain $E_2$
power operations, which we do in Section~\ref{sec:power}.

The original B{\"o}kstedt--Hsiang--Madsen presentation~\cite{BHM93}
of $TC(B)$ was given in terms of fixed point spectra $THH(B)^C$ for finite
subgroups $C \subset \bT$, using the language of genuinely equivariant
stable homotopy theory.  However, almost all calculations were made
using the naively equivariant homotopy fixed points $THH(B)^{hC}$ and
Tate constructions $THH(B)^{tC}$, and were therefore only known to be
valid in the range of degrees where the comparison
map $\hat\Gamma_1$ induces an isomorphism.

The new Nikolaus--Scholze presentation~\cite{NS18} of topological
cyclic homology promoted the ingredients that were previously used for
calculations into definitions.  Hence $TC(B)$ was redefined in terms
of the homotopy fixed points $THH(B)^{h\bT}$ and Tate construction
$THH(B)^{t\bT}$, and the key role of the (naively $\bT$-equivariant)
map~$\hat\Gamma_1$, now called the $p$-cyclotomic structure
map~$\varphi_p$, was greatly clarified.  Moreover, Nikolaus--Scholze
proved that the old and new definitions agree when $THH(B)$ is bounded
below, e.g.~for connective~$B$.  This means that by carrying out the
homotopy fixed point and Tate construction calculations in all degrees,
we can now fully calculate $V(2)_* TC(BP\<2\>)$, eliminating the gap
of degrees discussed above.  We compare the old and new terminologies
in Section~\ref{sec:nomenclature}.
\end{remark}

\begin{remark} \label{rem:recent}
After the present paper was first posted in preprint form,
Hahn, Raksit and Wilson~\cite{HRW} introduced a motivic filtration
on~$TC(R)$, for so-called chromatically quasi-syntomic $E_\infty$
ring spectra~$R$, whose associated graded realizes the form of motivic
cohomology that was predicted to exist in Remark~\ref{rem:descent}.
This new cohomology theory for $E_\infty$ ring spectra generalizes the
syntomic cohomology for quasi-syntomic commutative rings introduced by
Bhatt, Morrow and Scholze~\cite{BMS19}*{\S7.4}.

In the same year,
Burklund, Schlank and~Yuan~\cite{BSY}*{Thm.~E}, building
on~\cite{Yua21}*{Thm.~A}, proved that if $R$ is an $E_\infty$
ring spectrum such that $K(n)_* R \ne 0$ and $K(n+1)_* R =
0$, then $K(n+1)_* K(R) \ne 0$.  Combined with previous
work of Land--Meier--Mathew--Tamme~\cite{LMMT}*{Cor.~B} and
Clausen--Mathew--Naumann--Noel~\cite{CMNN} on the vanishing of $K(m)_*
K(R)$ for $m \ge n+2$, this proves that algebraic $K$-theory of an
$E_\infty$ ring spectrum increments chromatic complexity by precisely
one, thus establishing a very general form of qualitative redshift.
\end{remark}

{\bf Acknowledgments.}
We all thank the referee for good advice.
The second author acknowledges support from the project ANR-16-CE40-0003
ChroK.  The third author was supported by the Max Planck Institute for
Mathematics while this work was being carried out.  He would like to thank
the Institute for their hospitality.  The fourth author thanks the Radboud
Excellence Initiative for funding her postdoc position.  This project
received funding from the European Union's Horizon 2020 research and
innovation programme under the Marie Sk{\l}odowska-Curie grant agreement
No 101034255. \thinspace \includegraphics[scale=0.12]{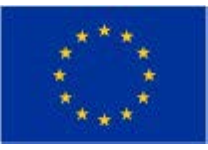}

\section{Smith--Toda and truncated Brown--Peterson spectra}

Let $\cA_*$ be the mod~$p$ dual Steenrod algebra, and write $H_* X =
H_*(X; \bF_p)$ for the mod~$p$ homology of a spectrum~$X$, viewed as
an $\cA_*$-comodule.  Likewise, let $H = H\bF_p$ denote the mod~$p$
Eilenberg--MacLane ($E_\infty$ ring) spectrum.

By a Smith--Toda complex $V(n)$ we mean a finite and $p$-local
spectrum with $H_* V(n) = E(\tau_0, \dots, \tau_n) \subset
\cA_*$.  The spectra $V(0) = S \cup_p e^1$, $V(1) = S \cup_p e^1
\cup_{\alpha_1} e^{2p-1} \cup_p e^{2p}$ and~$V(2)$ exist for $p\ge2$,
$p\ge3$ and~$p\ge5$, respectively, see Smith~\cite{Smi70}*{\S4} and
Toda~\cite{Tod71}*{Thm.~1.1}.  In the stable homotopy category there
are unital multiplications $\mu_0 \: V(0) \wedge V(0) \to V(0)$,
$\mu_1 \: V(1) \wedge V(1) \to V(1)$ and~$\mu_2 \: V(2) \wedge
V(2) \to V(2)$ for $p\ge3$, $p\ge5$ and~$p\ge7$, respectively,
cf.~\cite{YY77}*{\S1.4, \S2.4, \S3.3}.  These are unique, and
therefore commutative.  They are also associative, with the exception
of $\mu_0$ at $p=3$.  Toda~\cite{Tod71}*{Thm.~4.4} showed that $V(3)$
exists for $p\ge7$ and admits a unital multiplication for $p\ge11$.
The spectra $V(n)$ for $n\ge4$ are not known to exist at any prime~$p$,
cf.~\cite{Rav04}*{(5.6.13)}.
We use the following notation for some of the resulting homotopy cofiber
sequences.
\begin{align}
S \overset{p}\longto S
	&\overset{i_0}\longto V(0)
	\overset{j_0}\longto \Sigma S \label{eq:V0} \\
\Sigma^{2p-2} V(0) \overset{v_1}\longto V(0)
	&\overset{i_1}\longto V(1)
	\overset{j_1}\longto \Sigma^{2p-1} V(0) \label{eq:V1} \\
\Sigma^{2p^2-2} V(1) \overset{v_2}\longto V(1)
	&\overset{i_2}\longto V(2)
	\overset{j_2}\longto \Sigma^{2p^2-1} V(1) \label{eq:V2} \\
\Sigma^{2p^3-2} V(2) \overset{v_3}\longto V(2)
	&\overset{i_3}\longto V(3)
	\overset{j_3}\longto \Sigma^{2p^3-1} V(2) \label{eq:V3} \,.
\end{align}
The unital multiplications on $V(0)$, $V(1)$ and $V(2)$ are also regular, in
the sense that the respective Bockstein operators $i_0 j_0 \: V(0) \to
\Sigma V(0)$, $i_1 j_1 \: V(1) \to \Sigma^{2p-1} V(1)$ and $i_2 j_2 \:
V(2) \to \Sigma^{2p^2-1} V(2)$ act as derivations.
See \cite{AT65}*{Thm.~5.9} and~\cite{Yos77}*{Prop.~1.1, Prop.~1.2}.

The complex cobordism spectrum~$MU$ is a prototypical $E_\infty$ ring
spectrum.  Basterra--Mandell~\cite{BM13}*{Thm.~1.1} proved that the
$p$-local Brown--Peterson spectrum $BP$ is a retract up to homotopy of
$MU_{(p)}$ in the category of $E_4$ ring spectra, and that the $E_4$ ring
structure on~$BP$ is unique up to equivalence.  By an $n$-th truncated
Brown--Peterson spectrum~$BP\<n\>$ we mean a complex orientable $p$-local
ring spectrum such that the composite
$$
\bZ_{(p)}[v_1, \dots, v_n] \subset \pi_* BP \longto \pi_* MU_{(p)}
	\longto \pi_* BP\<n\>
$$
is an isomorphism, following~\cite{LN14}*{Def.~4.1}.  It follows, as
in~\cite{LN14}*{Thm.~4.4}, that $H_* BP\<n\> = P(\bar\xi_k \mid k\ge1)
\otimes E(\bar\tau_k \mid k > n)$ as a subalgebra of the dual Steenrod
algebra.  According to recent work by Hahn and Wilson~\cite{HW22}, there
exist towers
$$
\dots \longto BP\<n+1\> \longto BP\<n\> \longto \dots
	\longto BP\<0\> = H\bZ_{(p)}
$$
of $E_3$ $BP$-algebra spectra, for all~$p$, where each $BP\<n\>$ is an
$n$-th truncated Brown--Peterson spectrum.  Hence $THH(BP)$ is an $E_3$
ring spectrum with cyclotomic structure, in the sense to be recalled in
Section~\ref{sec:nomenclature}, and there are towers
$$
\dots \longto THH(BP\<n+1\>) \longto THH(BP\<n\>) \longto \dots
        \longto THH(\bZ_{(p)})
$$
of $E_2$ $THH(BP)$-algebra spectra with cyclotomic structure.  The
availability of these $\bT$-equivariant ring spectrum structures is an
essential prerequisite for the calculations given in the present paper.

Chadwick--Mandell~\cite{CM15}*{Cor.~1.3} showed that the Quillen map
$MU_{(p)} \to BP$ is an $E_2$ ring map, and it follows from~\cite{BM13}
that this map exhibits $BP$ as a retract up to homotopy of $MU_{(p)}$
in the category of $E_2$ ring spectra.  It is not known whether the
Basterra--Mandell and Quillen/Chadwick--Mandell $E_2$ ring spectrum
splittings can be chosen to agree, but the induced splittings of $\pi_*
THH(BP)$ off from $\pi_* THH(MU_{(p)})$, in the category of differential
graded algebras, must agree modulo addition of decomposables and
multiplication by $p$-local units.
Hence the calculations in~\cite{Rog20}*{Thm.~5.6} of the $\sigma$-operator
on $\pi_* THH(BP)$, induced by the $\bT$-action on $THH(BP)$, is valid
also for the Basterra--Mandell splitting, up to decomposables and
$p$-local units.

\section{Topological Hochschild homology}
\label{sec:THH}

Let $p$ be an odd prime.  We use the conjugate pair of presentations
\begin{align*}
\cA_* &= P(\xi_k \mid k \ge 1) \otimes E(\tau_k \mid k \ge 0) \\
  &= P(\bar\xi_k \mid k \ge 1) \otimes E(\bar\tau_k \mid k \ge 0) = H_* H
\end{align*}
of the dual Steenrod algebra~\cite{Mil58}, with $\bar\xi_k = \chi(\xi_k)$
in degree $2(p^k-1)$ and $\bar\tau_k = \chi(\tau_k)$ in degree $2p^k-1$.
The Hopf algebra coproduct is given by
$$
\psi(\bar\xi_k) = \sum_{i+j=k} \bar\xi_i \otimes \bar\xi_j^{p^i}
\qquad\text{and}\qquad
\psi(\bar\tau_k) = 1\otimes\bar\tau_k + \sum_{i+j=k} \bar\tau_i \otimes
\bar\xi_j^{p^i} \,.
$$
The mod~$p$ homology Bockstein satisfies $\beta(\bar\tau_k) = \bar\xi_k$.
The same formulas give the $\cA_*$-coaction~$\nu$ and Bockstein operation
on the subalgebras
\begin{align*}
H_* BP &= P(\bar\xi_k \mid k\ge1) \\
H_* BP\<n\> &= P(\bar\xi_k \mid k\ge1) \otimes E(\bar\tau_k \mid k>n)
\end{align*}
of $\cA_*$.  For each $E_1$ ring spectrum (or $S$-algebra) $B$, the
topological Hochschild homology $THH(B)$ has a natural $\bT$-action,
which induces $\sigma$-operators
\begin{align*}
\sigma \: H_* THH(B) &\longto H_{*+1} THH(B) \\
\sigma \: \pi_* THH(B) &\longto \pi_{*+1} THH(B)
\end{align*}
in homology and homotopy.  Since $BP$ and the $BP\<n\>$ are (at least)
$E_3$ ring spectra, we can make the following homology computations.

\begin{proposition}[\cite{MS93}*{Rem.~4.3}, \cite{AR05}*{Thm.~5.12}]
\label{prop:HTHHBPBPn} There are $\cA_*$-comodule algebra isomorphisms
$$
H_* THH(BP) \cong H_* BP \otimes E(\sigma\bar\xi_k \mid k\ge1)
$$
and
$$
H_* THH(BP\<n\>) \cong H_* BP\<n\>
	\otimes E(\sigma\bar\xi_1, \dots, \sigma\bar\xi_{n+1})
	\otimes P(\sigma\bar\tau_{n+1}) \,.
$$
Each class $\sigma\bar\xi_k$ is $\cA_*$-comodule primitive, while
$\nu(\sigma\bar\tau_{n+1}) = 1 \otimes \sigma\bar\tau_{n+1} + \bar\tau_0
\otimes \sigma\bar\xi_{n+1}$.
\end{proposition}

Passing to homotopy, recall that $\pi_* BP = \bZ_{(p)}[v_n \mid n\ge1]$
with $|v_n| = 2p^n-2$.  To be definite, we take the $v_n$ to be the
Hazewinkel generators.

\begin{proposition}
[\cite{MS93}*{Rem.~4.3}, \cite{Rog20}*{Prop.~4.6, Thm.~5.6}]
\label{prop:piTHHBP}
There is an algebra isomorphism
$$
\pi_* THH(BP) \cong \pi_* BP
	\otimes E(\lambda_n \mid n\ge1) \,,
$$
where $\lambda_n$ has degree~$|\lambda_n| = 2p^n-1$ and (mod~$p$) Hurewicz
image $h(\lambda_n) = \sigma\bar\xi_n$.  Here $\sigma(\lambda_n) = 0$
for each~$n$.  The first few $\sigma(v_n)$ satisfy
\begin{align*}
\sigma(v_1) &= p \lambda_1 \\
\sigma(v_2) &= p \lambda_2 - (p+1) v_1^p \lambda_1 \\
\sigma(v_3) &= p \lambda_3 - (p v_1 v_2^{p-1} + v_1^{p^2}) \lambda_2 \\
        &\qquad - (v_2^p - (p+1) v_1^{p+1} v_2^{p-1}
                + p^2 v_1^{p^2-1} v_2 + p v_1^{p^2+p}) \lambda_1 \,.
\end{align*}
\end{proposition}

The specific choice of $\lambda_n \in \pi_{2p^n-1} THH(BP)$ made
in~\cite{Rog20} is the unique class detected by $t_n \in \pi_{2p^n-2}(BP
\wedge BP)$ in filtration degree~$1$ of the spectral sequence associated
to the skeleton filtration of $THH(BP)$.  The claim that its Hurewicz
image equals $\sigma\bar\xi_n \in H_{2p^n-1} THH(BP)$ follows from the
proof of~\cite{Zah72}*{Lem.~3.7}.

If $V(n)$ exists as a finite spectrum with
$$
H_* V(n) = E(\tau_0, \dots, \tau_n) \,,
$$
then $H_*(V(n) \wedge BP\<n\>) \cong \cA_*$, so that $V(n) \wedge BP\<n\>
\simeq H$.  We write $h_n \: V(n)_* X \to H_* X$ for the (generalized)
Hurewicz homomorphism induced by the map $V(n) \to H$ extending the
unit $S \to H$.

\begin{proposition}[\cite{AR12}*{Lem.~4.1}, \cite{AKCH}*{Prop.~2.9}]
	\label{prop:VnTHHBPn}
Suppose that $V(n)$ exists as a ring spectrum.  Then
\begin{align*}
V(n)_* THH(BP\<n\>) &= \pi_*(V(n) \wedge THH(BP\<n\>)) \\
	&= E(\lambda_1, \dots, \lambda_{n+1}) \otimes P(\mu_{n+1})
\end{align*}
maps isomorphically to the subalgebra of $\cA_*$-comodule primitives
in
\begin{align*}
H_* (V(n) \wedge THH(BP\<n\>)) &\cong H_* V(n) \otimes H_* THH(BP\<n\>) \\
&\cong \cA_* \otimes E(\sigma\bar\xi_1, \dots, \sigma\bar\xi_{n+1})
	\otimes P(\sigma\bar\tau_{n+1}) \,.
\end{align*}
Here each $\lambda_k$ is the image of $\lambda_k \in \pi_{2p^k-1} THH(BP)$
under the natural map induced by $S \to V(n)$ and
$BP \to BP\<n\>$, with Hurewicz images
$$
h(\lambda_k) = 1 \wedge \sigma\bar\xi_k
\quad\text{and}\quad
h_n(\lambda_k) = \sigma\bar\xi_k \,.
$$
Moreover, $\mu_{n+1}$ in degree~$|\mu_{n+1}| = 2p^{n+1}$ is the
class with Hurewicz images
$$
h(\mu_{n+1}) = 1 \wedge \sigma\bar\tau_{n+1}
	+ \tau_0 \wedge \sigma\bar\xi_{n+1}
\quad\text{and}\quad
h_n(\mu_{n+1}) = \sigma\bar\tau_{n+1} \,.
$$
\end{proposition}

Note that the $\cA_*$-coaction sends $h(\mu_{n+1})$ to
$$
1 \otimes (1 \wedge \sigma\bar\tau_{n+1}) + \bar\tau_0 \otimes (1 \wedge
\sigma\bar\xi_{n+1}) + 1 \otimes (\tau_0 \wedge \sigma\bar\xi_{n+1})
+ \tau_0 \otimes (1 \wedge \sigma\bar\xi_{n+1})
	= 1 \otimes h(\mu_{n+1}) \,,
$$
so that this class is $\cA_*$-comodule primitive.  We spell out these
definitions a little more explicitly in the case of main interest in
this paper.

\begin{definition} \label{def:lam1lam2lam3mu}
For $p\ge7$ let
$$
\lambda_1, \lambda_2, \lambda_3, \mu_3 \in V(2)_* THH(BP\<2\>)
$$
denote the classes in degrees $|\lambda_1| = 2p-1$, $|\lambda_2| =
2p^2-1$, $|\lambda_3| = 2p^3-1$ and~$|\mu_3| = 2p^3$ with Hurewicz
images $h(\lambda_1) = 1 \wedge \sigma\bar\xi_1$, $h(\lambda_2) = 1 \wedge
\sigma\bar\xi_2$, $h(\lambda_3) = 1 \wedge \sigma\bar\xi_3$ and~$h(\mu_3)
= 1 \wedge \sigma\bar\tau_3 + \tau_0 \wedge \sigma\bar\xi_3$.  Then
$$
V(2)_* THH(BP\<2\>) = E(\lambda_1, \lambda_2, \lambda_2)
	\otimes P(\mu_3) \,,
$$
which has at most one monomial generator in each degree.
We generally abbreviate $\mu_3$ to $\mu$ when only discussing $BP\<2\>$.
\end{definition}

The $V(n)$-homotopy classes $\mu_{n+1}$ should not be confused
with the ring spectrum multiplications~$\mu_n \: V(n) \wedge V(n)
\to V(n)$, which hereafter appear explicitly only in the proof of
Proposition~\ref{prop:CartanV0V1}.

\section{Cyclotomic nomenclature}
\label{sec:nomenclature}

We review some notations in common use from 1994 to 2017, including
the articles~\cite{HM97}, \cite{Rog99}, \cite{AR02}, \cite{HM03}
and~\cite{AR12}.  For each $\bT$-spectrum $X$ there is a natural map
$$
r \: X^{C_p} \longto \Phi^{C_p} X
$$
of $\bT/C_p$-spectra from the categorical $C_p$-fixed points to the
geometric $C_p$-fixed points.  The latter were introduced, as ``spacewise
$C_p$-fixed points'', in~\cite{LMS86}*{Def.~II.9.7}, essentially as a
left Kan extension.  This definition agrees with what has later been
called the monoidal geometric fixed points~\cite{MM02}.  Recall the
$\bT$-equivariant homotopy cofiber sequence
$$
E\bT_+ \overset{c}\longto S^0 \overset{e}\longto \wET \,.
$$
In the commutative square
$$
\xymatrix{
X^{C_p} \ar[r]^-{r} \ar[d]_-{e} & \Phi^{C_p}(X) \ar[d]^{\simeq} \\
(\wET \wedge X)^{C_p} \ar[r]^-{\simeq} & \Phi^{C_p}(\wET \wedge X)
}
$$
the right hand and lower maps are $\bT/C_p$-equivariant equivalences.
The expression $(\wET \wedge X)^{C_p}$ is therefore sometimes~\cite{HM97}
taken as a definition of the geometric fixed points, but this
construction is not strictly monoidal.  The commutative square
$$
\xymatrix{
X^{C_p} \ar[r]^-{r} \ar[d]_-{c^*} & \Phi^{C_p}(X) \ar[d]^-{c^*} \\
F(E\bT_+, X)^{C_p} \ar[r]^-{r} & \Phi^{C_p}(F(E\bT_+, X))
}
$$
is $\bT/C_p$-equivariantly homotopy Cartesian.  Note that
$\Phi^{C_p}(F(E\bT_+, X)) \simeq [\wET \wedge F(E\bT_+, X)]^{C_p}
= X^{tC_p}$ defines the $C_p$-Tate $\bT/C_p$-spectrum.  These
$\bT/C_p$-spectra are hereafter viewed as $\bT$-spectra via the $p$-th
root isomorphism $\rho \: \bT \cong \bT/C_p$, which we omit to exhibit
in the notation.

The $\bT$-spectra $X = THH(B)$ are cyclotomic, in the sense that there
are $\bT$-equivalences $\Phi^{C_p}(THH(B)) \simeq THH(B)$.
Hence~\cite{BM94}*{(6.1)} there are vertical maps of horizontal homotopy
cofiber sequences
$$
\xymatrix{
THH(B)_{hC_{p^n}} \ar[r]^-{N} \ar@{=}[d]
	& THH(B)^{C_{p^n}} \ar[r]^-{R} \ar[d]_-{\Gamma_n}
	& THH(B)^{C_{p^{n-1}}} \ar[d]^-{\hat\Gamma_n} \\
THH(B)_{hC_{p^n}} \ar[r]^-{N^h}
	& THH(B)^{hC_{p^n}} \ar[r]^-{R^h}
	& THH(B)^{tC_{p^n}}
}
$$
known as the norm--restriction sequences, for all~$n$.
Here the (Witt vector restriction) maps $R$ are given by
$$
r^{C_{p^{n-1}}} \: THH(B)^{C_{p^n}} \longto
	\Phi^{C_p}(THH(B))^{C_{p^{n-1}}} \simeq THH(B)^{C_{p^{n-1}}} \,.
$$
The norm maps~$N$ are given by the Adams transfer equivalence
$THH(B)_{hC_{p^n}} \simeq [E\bT_+ \wedge THH(B)]^{C_{p^n}}$, followed
by the map induced by~$c \: E\bT_+ \to S^0$.  The right hand homotopy
Cartesian squares are compatible with the (Witt vector Frobenius) maps
$F \: X^{C_{p^n}} \to X^{C_{p^{n-1}}}$ that forget some invariance.
The Witt vector terminology is motivated by the effects of these maps
on $\pi_0$ for connective $B$, in view of the isomorphisms $\pi_0
THH(B)^{C_{p^n}} \cong W_{n+1}(\pi_0(B))$ of~\cite{HM97}*{Thm.~3.3}.

The homotopy restriction map $R^h$ is induced by $e \: S \to \wET$,
and induces a map of spectral sequences from the $C_{p^n}$-homotopy
fixed point spectral sequence to the $C_{p^n}$-Tate spectral sequence.
The map $\Gamma_n$ is the comparison map from fixed points to homotopy
fixed points, and $\hat\Gamma_n$ denotes its Tate analogue.  Passing to
homotopy limits over the maps~$F$, and implicitly $p$-completing, one
obtains a map of homotopy cofiber sequences
$$
\xymatrix{
\Sigma THH(B)_{h\bT} \ar[r]^-{N} \ar@{=}[d]
	& TF(B) \ar[r]^-{R} \ar[d]_-{\Gamma}
        & TF(B) \ar[d]^-{\hat\Gamma} \\
\Sigma THH(B)_{h\bT} \ar[r]^-{N^h}
	& THH(B)^{h\bT} \ar[r]^-{R^h}
        & THH(B)^{t\bT} \,.
}
$$
Again, $R^h$ is induced by $e \: S \to \wET$ and induces a map of spectral
sequences from the $\bT$-homotopy fixed point spectral sequence to the
$\bT$-Tate spectral sequence.  The topological cyclic homology
$$
\xymatrix{
TC(B) \ar[r]^-{\pi} & TF(B) \ar@<1ex>[r]^-{1} \ar@<-1ex>[r]_-{R} & TF(B)
}
$$
was originally defined by B{\"o}kstedt, Hsiang and Madsen~\cite{BHM93}
as the homotopy equalizer of the identity $1 \: TF(B) \to TF(B)$ and the
restriction map~$R \: TF(B) \to TF(B)$.  We refer to the preferred lifts
$trc \: K(B) \to TC(B)$ and $tr_{\bT} = \Gamma \circ \pi \circ trc \: K(B)
\to THH(B)^{h\bT}$ of the B{\"o}kstedt trace map $tr \: K(B) \to THH(B)$
as the cyclotomic trace map and the circle trace map, respectively.

Some important recent papers give new emphasis to many of these objects.
Hesselholt~\cite{Hes18} writes
$$
TP(B) = THH(B)^{t\bT}
$$
for the circle Tate construction on $THH(B)$ and calls it the periodic
topological cyclic homology of~$B$.  (One might also say topological
periodic homology.) Nikolaus--Scholze~\cite{NS18} write
$$
TC^-(B) = THH(B)^{h\bT}
$$
for the circle homotopy fixed points of $THH(B)$ and call it the
negative cyclic homology, write
$$
\varphi_p = \hat\Gamma_1 \: THH(B) \longto THH(B)^{tC_p}
$$
for the comparison map and call it the $p$-cyclotomic
structure map, and write
$$
\can \: TC^-(B) \longto TP(B)
$$
for the homotopy restriction map
$$
R^h \: THH(B)^{h\bT} \longto THH(B)^{t\bT}
$$
and refer to it as the canonical map.  The structure map
$$
\epsilon \: X \longto (X^{\wedge p})^{tC_p} = R_+(X)
$$
to the topological Singer construction, from~\cite{BMMS86}*{\S
II.5} and~\cite{LNR12}, is now called the Tate diagonal.

In the definition of $TC(B)$ as a homotopy equalizer, Nikolaus--Scholze
replace $TF(B)$ in the source by $THH(B)^{h\bT}$ via~$\Gamma$, and
replace $TF(B)$ in the target by $THH(B)^{t\bT}$ via $\hat\Gamma$.
In view of the commutative square
$$
\xymatrix{
TF(B) \ar[r]^-{\hat\Gamma} \ar[d]_-{\Gamma}
	& THH(B)^{t\bT} \ar[d]_-{G}^-{\simeq} \\
THH(B)^{h\bT} \ar[r]^-{(\hat\Gamma_1)^{h\bT}}
	& (THH(B)^{tC_p})^{h\bT}
}
$$
from~\cite{HM97}*{p.~68}, \cite{AR02}*{p.~27}, the identity map~$1 \:
TF(B) \to TF(B)$ is then replaced with the circle homotopy fixed points
$(\hat\Gamma_1)^{h\bT} = \varphi_p^{h\bT}$ of the $p$-cyclotomic structure
map, suppressing the (still implicitly $p$-complete) equivalence
$$
G \: THH(B)^{t\bT} = (THH(B)^{tC_p})^{\bT}
	\longto (THH(B)^{tC_p})^{h\bT}
$$
from the notation.  The fact that~$G$ is an equivalence for connective~$B$
was shown by computation in the first instances considered, and then
proved in~\cite{BBLNR14}*{Prop.~3.8} under the assumption that $H_* B$
is of finite type.  It reappears in the new terminology as the Tate orbit
lemma~\cite{NS18}*{Lem.~I.2.1}, since $(THH(B)_{hC_p})^{t\bT} \simeq *$
is equivalent to $\Sigma THH(B)_{h\bT} \to (THH(B)_{hC_p})^{h\bT}$ being
an equivalence, which in turn is equivalent to $G$ being an equivalence.

Likewise, the restriction map $R \: TF(B) \to TF(B)$ is replaced with
the homotopy restriction map $R^h = \can$.  Combining these replacements,
$$
\xymatrix@C+1pc{
TC(B) \ar[r]^-{\pi} & THH(B)^{h\bT}
	\ar@<1ex>[r]^-{G^{-1} (\hat\Gamma_1)^{h\bT}}
	\ar@<-1ex>[r]_-{R^h} & THH(B)^{t\bT}
}
$$
is redefined as the homotopy equalizer of $G^{-1} \circ
(\hat\Gamma_1)^{h\bT}$ and $R^h =\can$, much as in~\cite{AR12}*{p.~1072},
or (in order not to need to invert~$G$) as the homotopy equalizer
$$
\xymatrix{
TC(B) \ar[r]^-{\pi} & THH(B)^{h\bT}
	\ar@<1ex>[r]^-{(\hat\Gamma_1)^{h\bT}}
	\ar@<-1ex>[r]_-{GR^h} & (THH(B)^{tC_p})^{h\bT}
}
$$
of $(\hat\Gamma_1)^{h\bT} = \varphi_p^{h\bT}$ and $G \circ R^h$.
The old and new definitions of~$TC(B)$ agree for connective~$B$,
by~\cite{NS18}*{Thm.~II.3.8}.

\section{Homotopy power operations}
\label{sec:power}

Let $B$ be an $E_{n+1}$ ring spectrum.  Using the Boardman--Vogt tensor
product of operads~\cite{Dun88}, we may view~$B$ as an $E_n$ algebra in
the category of $E_1$ ring spectra (or $S$-algebras).  There are then
natural $E_n$ algebra structures on the algebraic $K$-theory spectrum
$K(B)$ and on the cyclotomic spectrum $THH(B)$, and these are respected
by the trace map $K(B) \to THH(B)$, as well as its cyclotomic refinements.

For each $E_2$ ring spectrum $R$, there is a natural ``top'' homology
operation
$$
\xi_1 \: H_{2k-1} R \longto H_{2pk-1} R
$$
introduced in~\cite{CLM76}*{Thm.~III.1.3}.  If $R$ is an $E_3$ ring
spectrum, then $\xi_1 = Q^k$ is the Araki--Kudo/Dyer--Lashof/Cohen
operator, as defined in~\cite{CLM76}*{Thm.~III.1.1}, and we will also
use this notation in the $E_2$ ring spectrum case, to emphasize the
dependence on~$k$ (and to avoid confusion with the element~$\xi_1$
in the dual Steenrod algebra).
Let $\beta$ denote the mod~$p$ homology Bockstein operator.
In~\cite{AR02}*{\S1.5} two of us discussed a homotopy operation
$$
P^k \: \pi_{2k-1} R \longto V(0)_{2pk-1} R
$$
lifting~$Q^k$ (see Lemma~\ref{lem:h0PkeqQkh}),
in the context of $E_\infty$ ring spectra.
In this paper we will extend its definition to $E_2$ ring spectra,
and construct a homotopy operation
$$
P^k \: V(0)_{2k-1} R \longto V(1)_{2pk-1} R
$$
also lifting~$Q^k$ (see Lemma~\ref{lem:h1PkeqQkh}).

To define these operations for $E_2$ ring spectra~$R$, we make use of
the little $2$-cubes operad~$\cC_2$ encoding $E_2$ algebra structures.
For a spectrum~$X$ let
$$
Br_p X = D_{2,p} X = \cC_2(p) \ltimes_{\Sigma_p} X^{\wedge p}
$$
denote the $p$-th braided-extended power of~$X$.  Note that $Br_p \Sigma^2
X \cong \Sigma^{2p} Br_p X$ by \cite{CMM78}*{Thm.~1}.  In the case $X =
S^{2k-1}$, with $H_* X = \bF_p\{x_{2k-1}\}$,
\begin{equation} \label{eq:HBrpSodd}
H_* Br_p S^{2k-1} = \bF_p \{ \beta Q^k(x_{2k-1}), Q^k(x_{2k-1}) \}
\end{equation}
follows from~\cite{CLM76}*{Thm.~III.5.3}, cf.~\cite{Coh81}*{Prop.~II.1.2}.
Hence there is an (implicitly $p$-complete) equivalence $\bar\eta_0 \:
\Sigma^{2pk-1} DV(0) \simeq Br_p S^{2k-1}$, with right adjoint
$$
\eta_0 \: S^{2pk-1} \longto V(0) \wedge Br_p S^{2k-1} \,.
$$
Here $DV(0) \simeq \Sigma^{-1} V(0)$ denotes the Spanier--Whitehead
dual of $V(0)$, and $h_0(\eta_0) = Q^k(x_{2k-1})$.

For typographical reasons we will often simply write $g$ for the
maps $1 \wedge g \: A \wedge B \to A \wedge C$ and $g \wedge 1 \:
B \wedge D \to C \wedge D$, for suitable $A$, $g \: B \to C$ and~$D$.

\begin{definition} \label{def:PkSV0}
Let $R$ be an $E_2$ ring spectrum.  The homotopy power operation
$$
P^k \: \pi_{2k-1} R \longto V(0)_{2pk-1} R
$$
sends each map $f \: S^{2k-1} \to R$ to the composite
$$
P^k(f) \: S^{2pk-1} \overset{\eta_0}\longto V(0) \wedge Br_p S^{2k-1}
\overset{Br_p f}\longto V(0) \wedge Br_p R
\overset{\theta}\longto V(0) \wedge R \,,
$$
where $\theta \: Br_p R \to R$ is part of the $E_2$ ring structure.
\end{definition}

In the case $X = \Sigma^{2k-1} DV(0)$, where $H_* X = \bF_p\{x_{2k-2},
x_{2k-1}\}$ with $\beta(x_{2k-1}) = x_{2k-2}$, there is an inclusion
$$
\bF_p\{x_{2k-2}^p, \beta Q^k(x_{2k-1}), Q^k(x_{2k-1})\}
	\subset H_* Br_p \Sigma^{2k-1} DV(0)
$$
of left $\cA_*$-comodules, or of right $\cA$-modules.  Of the dual
Steenrod operations, only $\beta$ and $\cP^1_*$ act nontrivially on the
left hand side, with
$$
\cP^1_* Q^k(x_{2k-1}) = 0
\qquad\text{and}\qquad
\cP^1_* \beta Q^k(x_{2k-1}) = - x_{2k-2}^p \,,
$$
according to the spectrum-level Nishida relations,
see~\cite{CLM76}*{Thm.~III.1.1(6), Thm.~III.1.3(3)}
and~\cite{BMMS86}*{Thm.~III.1.1(8)}.
As in~\cite{Tod71}, let
$$
V(1/2) = S \cup_p e^1 \cup_{\alpha_1} e^{2p-1} \,,
$$
so that $V(0) \subset V(1/2) \subset V(1)$ and $DV(1/2) \simeq
\Sigma^{1-2p}(S \cup_{\alpha_1} e^{2p-2} \cup_p e^{2p-1})$.  The following
construction refines a map discussed by Toda in~\cite{Tod68}*{Lem.~3}.

\begin{lemma} \label{lem:3celltoBrpDV0}
There exists a ($p$-complete) map
$$
\bar\eta_{1/2} \:
\Sigma^{2pk-1} DV(1/2) \longto Br_p \Sigma^{2k-1} DV(0)
$$
realizing the inclusion of $\bF_p\{x_{2k-2}^p, \beta Q^k(x_{2k-1}),
Q^k(x_{2k-1})\}$ in homology.
\end{lemma}

\begin{proof}
We can choose a minimal cell structure on $Br_p \Sigma^{2k-1} DV(0)$
with a $(2pk-1)$-cell representing $Q^k(x_{2k-1})$ that is attached by a
degree~$p$ map to a $(2pk-2)$-cell representing $\beta Q^k(x_{2k-1})$.
The $(2pk-1)$-cell is not attached to the $(2pk-2p+1)$-skeleton, since
$\cP^1_* Q^k(x_{2k-1}) = 0$.  We can orient the $(2pk-2p)$-cell so
that the $(2pk-2)$-cell is attached to it by $\alpha_1$, since $\cP^1_*
\beta Q^k(x_{2k-1}) = - x_{2k-2}^p$.
\end{proof}

We fix a choice of $\bar\eta_{1/2}$ for each integer~$k$, but see
Remark~\ref{rem:eta1indet} below.  This specifies a composite map
\begin{equation} \label{eq:bareta1}
\bar\eta_1 \: \Sigma^{2pk-1} DV(1)
	\longto \Sigma^{2pk-1} DV(1/2)
	\overset{\bar\eta_{1/2}}\longto Br_p \Sigma^{2k-1} DV(0) \,,
\end{equation}
with homology image $\bF_p\{x_{2k-2}^p, \beta Q^k(x_{2k-1}),
Q^k(x_{2k-1})\}$, and we write
$$
\eta_1 \: S^{2pk-1} \overset{\eta_{1/2}}\longto
	V(1/2) \wedge Br_p \Sigma^{2k-1} DV(0)
	\longto V(1) \wedge Br_p \Sigma^{2k-1} DV(0)
$$
for its right adjoint.

\begin{definition} \label{def:PkV0V1}
Let $R$ be an $E_2$ ring spectrum.  The homotopy power operation
$$
P^k \: V(0)_{2k-1} R \longto V(1/2)_{2pk-1} R \longto V(1)_{2pk-1} R
$$
sends each map $f \: S^{2k-1} \to V(0) \wedge R$, with left adjoint
$\bar f \: \Sigma^{2k-1} DV(0) \to R$, to the composite
$$
P^k(f) \: S^{2pk-1}
	\overset{\eta_1}\longto V(1) \wedge Br_p \Sigma^{2k-1} DV(0)
	\overset{Br_p \bar f}\longto V(1) \wedge Br_p R
	\overset{\theta}\longto V(1) \wedge R \,.
$$
\end{definition}

\begin{remark} \label{rem:eta1indet}
We discuss the non-uniqueness of $\bar\eta_{1/2}$ and the resulting
ambiguity in the operation~$P^k$ just defined.  For brevity, let $U =
Br_p \Sigma^{2k-1} DV(0)$.  By~\cite{CLM76}*{Thm.~III.3.1} we have
$$
H_* U \cong \bF_p\{x_{2k-2}^p,
x_{2k-2}^{p-1} x_{2k-1}, x_{2k-2}^{p-2} y_{4k-3},
x_{2k-2}^{p-2} y_{4k-2}, x_{2k-2}^{p-3} x_{2k-1} y_{4k-3}\}
$$
in degrees $2pk-2p \le * \le 2pk-2p+2$, plus classes in higher degrees,
where
\begin{align*}
y_{4k-3} &= [x_{2k-2}, x_{2k-2}]_1 \\
y_{4k-2} &= [x_{2k-2}, x_{2k-1}]_1
\end{align*}
are $E_2$ ring spectrum Browder brackets.  (We write $[x,y]_1$ in place
of the traditional $\lambda_1(x,y)$ in order to avoid confusion with
the homotopy class~$\lambda_1$.)  The (additive) indeterminacies
in~$\bar\eta_{1/2}$ and~$\eta_1$ are maps $\Sigma^{2pk-1} DV(1/2) \to U$
and $m \: S^{2pk-1} \to V(1) \wedge U$, respectively, that induce zero
in homology.  The Atiyah--Hirzebruch spectral sequence for $V(1)_* U$
shows that $m = \alpha_1 \cdot n$ for a class $n \in V(1)_{2pk-2p+2}
U \cong H_{2pk-2p+2} U$, generated by $x_{2k-2}^{p-2} y_{4k-2}$
and~$x_{2k-2}^{p-3} x_{2k-1} y_{4k-3}$.  These generators map to zero in
$V(1)_* R$ if the $E_2$ ring structure on~$R$ extends to an~$E_3$ ring
structure.  Hence any two different choices of maps $\bar\eta_{1/2}$ will
give operations $P^k$ that differ at most by a multiple of~$\alpha_1$,
and which strictly agree if $R$ is an $E_3$ ring spectrum.
This means that for all of the assertions we will make about these
homotopy power operations the choice of $\bar\eta_{1/2}$ makes
no difference: In Lemma~\ref{lem:h1PkeqQkh} the Hurewicz homomorphism $h_1$
annihilates $\alpha_1$-multiples, and in Proposition~\ref{prop:CartanV0V1}
we assume that $R$ is an $E_\infty$ ring spectrum.
\end{remark}

\begin{lemma} \label{lem:h0PkeqQkh}
Let $R$ be an $E_2$ ring spectrum.  The square
$$
\xymatrix{
\pi_{2k-1} R \ar[r]^-{P^k} \ar[d]_-{h}
	& V(0)_{2pk-1} R \ar[d]^-{h_0} \\
H_{2k-1} R \ar[r]^-{Q^k}
	& H_{2pk-1} R
}
$$
commutes.
\end{lemma}

\begin{proof}
Let $t_0 \: H \to H \wedge DV(0)$ and $b_0 \: H \wedge V(0) \to H$ be
$H$-module maps that split off the top and bottom cells, respectively.
Then $b_0 h = h_0 \: V(0) \to H$ and the following diagram commutes.
$$
\scalebox{0.7}{
\xymatrix@C-0.5pc{
S^{2pk-1} \ar[r] \ar@(u,u)[rr]^-{\eta_0} \ar[d]^-{h}  
	& V(0) \wedge DV(0) \wedge S^{2pk-1} \ar[r]^-{\bar\eta_0} \ar[d]^-{h}
	& V(0) \wedge Br_p S^{2k-1} \ar[r]^-{Br_p f} \ar[d]^-{h}
	& V(0) \wedge Br_p R \ar[r]^-{\theta} \ar[d]^-{h}
	& V(0) \wedge R \ar[d]^-{h} \\
H \wedge S^{2pk-1} \ar[r] \ar[dr]_-{t_0}
	& H \wedge V(0) \wedge DV(0) \wedge S^{2pk-1} \ar[r]^-{\bar\eta_0} \ar[d]^-{b_0}
	& H \wedge V(0) \wedge Br_p S^{2k-1} \ar[r]^-{Br_p f} \ar[d]^-{b_0}
	& H \wedge V(0) \wedge Br_p R \ar[r]^-{\theta} \ar[d]^-{b_0}
	& H \wedge V(0) \wedge R \ar[d]^-{b_0} \\
& H \wedge DV(0) \wedge S^{2pk-1} \ar[r]^-{\bar\eta_0}
	& H \wedge Br_p S^{2k-1} \ar[r]^-{Br_p f} \ar@{<->}[d]^-{\cong}
	& H \wedge Br_p R \ar[r]^-{\theta} \ar@{<->}[d]^-{\cong}
	& H \wedge R \\
& & Br^H_p(H \wedge S^{2k-1}) \ar[r]^-{Br^H_p \bar g}
	& Br^H_p(H \wedge R)
}
}
$$
Here $\bar g = 1 \wedge f \: H \wedge S^{2k-1} \to H \wedge R$ denotes
the $H$-module map that is left adjoint to the Hurewicz image $g = hf \:
S^{2k-1} \to H \wedge R$, and $Br^H_p$ denotes the $p$-th braided-extended
power construction in the category of $H$-modules.  The upper composite
$S^{2pk-1} \to H \wedge R$ then represents $h_0 P^k(f)$, while the lower
composite represents
$$
Q_{p-1}(g) = \theta_*(e_{p-1} \otimes g^{\otimes p})
$$
up to a known unit in~$\bF_p$, with notation as
in~\cite{May70}*{Def.~2.2}, \cite{CLM76}*{\S I.1}.  This equals $Q^k(hf)$.
\end{proof}

\begin{lemma} \label{lem:h1PkeqQkh}
Let $R$ be an $E_2$ ring spectrum.  The square
$$
\xymatrix{
V(0)_{2k-1} R \ar[r]^-{P^k} \ar[d]_-{h_0}
	& V(1)_{2pk-1} R \ar[d]^-{h_1} \\
H_{2k-1} R \ar[r]^-{Q^k}
	& H_{2pk-1} R
}
$$
commutes.
\end{lemma}

\begin{proof}
Let $t_1 \: H \to H \wedge DV(1)$ and $b_1 \: H \wedge V(1) \to H$ be
$H$-module maps that split off the top and bottom cells, respectively.
Then $b_1 h = h_1 \: V(1) \to H$ and the following diagram commutes,
up to units in $\bF_p$.
$$
\scalebox{0.68}{
\xymatrix@C-1pc{
S^{2pk-1} \ar[r] \ar@(u,u)[rr]^-{\eta_1} \ar[d]^-{h}
	& V(1) \wedge DV(1) \wedge S^{2pk-1} \ar[r]^-{\bar\eta_1} \ar[d]^-{h}
	& V(1) \wedge Br_p \Sigma^{2k-1} DV(0) \ar[r]^-{Br_p \bar f} \ar[d]^-{h}
	& V(1) \wedge Br_p R \ar[r]^-{\theta} \ar[d]^-{h}
	& V(1) \wedge R \ar[d]^-{h} \\
H \wedge S^{2pk-1} \ar[r] \ar[dr]_-{t_1} \ar@(d,l)[ddr]_-{e_{p-1}}
	& H \wedge V(1) \wedge DV(1) \wedge S^{2pk-1} \ar[r]^-{\bar\eta_1} \ar[d]^-{b_1}
	& H \wedge V(1) \wedge Br_p \Sigma^{2k-1} DV(0) \ar[r]^-{Br_p \bar f} \ar[d]^-{b_1}
	& H \wedge V(1) \wedge Br_p R \ar[r]^-{\theta} \ar[d]^-{b_1}
	& H \wedge V(1) \wedge R \ar[d]^-{b_1} \\
& H \wedge DV(1) \wedge S^{2pk-1} \ar[r]^-{\bar\eta_1}
& H \wedge Br_p \Sigma^{2k-1} DV(0) \ar[r]^-{Br_p \bar f} \ar@{<->}[d]^-{\cong}
	& H \wedge Br_p R \ar[r]^-{\theta} \ar@{<->}[d]^-{\cong}
	& H \wedge R \\
& H \wedge Br_p S^{2k-1} \ar[ur]^(0.4){Br_p t_0} \ar@{<->}[d]^-{\cong}
	& Br^H_p(H \wedge \Sigma^{2k-1} DV(0)) \ar[r]^-{Br^H_p \bar f}
	& Br^H_p(H \wedge R) \\
& Br^H_p(H \wedge S^{2k-1}) \ar[ur]^-{Br^H_p t_0} \ar@(r,d)[urr]^-{Br^H_p \bar g}
}
}
$$
Here $\bar g \: H \wedge S^{2k-1} \to H \wedge R$ denotes the $H$-module
map extending the $V(0)$-Hurewicz image $g = h_0 f \: S^{2k-1} \to H
\wedge R$.  The two maps from $H \wedge S^{2pk-1}$ to $H \wedge Br_p
\Sigma^{2k-1} DV(0)$ agree, up to a known unit in $\bF_p$, because the
map~$\bar\eta_{1/2}$ in Lemma~\ref{lem:3celltoBrpDV0} sends the top cell
to $Q^k(x_{2k-1})$.  The two maps from $Br^H_p(H \wedge S^{2k-1})$
to $Br^H_p(H \wedge R)$ agree because $\bar g$ is homotopic through
$H$-module maps to $\bar f t_0$.  The upper composite $S^{2pk-1} \to
H \wedge R$ in the diagram represents $h_1 P^k(f)$, while the lower
composite represents a known unit times $Q_{p-1}(g)$, which equals
$Q^k(h_0 f)$.
\end{proof}

The following homotopy Cartan formula generalizes the one proved for
$E_\infty$ ring spectra in~\cite{AR02}*{Lem.~1.6}.

\begin{proposition} \label{prop:CartanSV0}
Let $R$ be an $E_3$ ring spectrum.  For $x \in \pi_{2i} R$ and $y
\in \pi_{2j-1} R$ the relation
$$
P^k(xy) = x^p P^j(y)
$$
holds in $V(0)_{2pk-1} R$, where $k = i+j$.
\end{proposition}

\begin{proof}
We use the following nearly-commutative diagram, where $\delta_p$ is
the operadic diagonal from~\cite{BMMS86}*{\S I.2},
$$
D_{n,p} X = \cC_n(p) \ltimes_{\Sigma_p} X^{\wedge p}
$$
denotes the $p$-th $E_n$-extended power, and $\sigma_1 \: X^{\wedge p}
\simeq D_{1,p} X \to D_{2,p} X  = Br_p X$ and $\sigma_2 \: Br_p X \to D_{3,p}
X$ are stabilization maps.

$$
\scalebox{0.85}{
\xymatrix@C+1.5pc{
\Sigma^{2pk-1} DV(0) \ar[r]^-{\bar\eta_0}
	\ar[d]^-{\cong}
& Br_p S^{2k-1} \ar[r]^-{Br_p(f \cdot g)} \ar[d]^-{\cong}
	& Br_p R \ar[r]^-{\theta}
	& R \\
S^{2pi} \wedge \Sigma^{2pj-1} DV(0) \ar[d]^-{\simeq}
& Br_p(S^{2i} \wedge S^{2j-1}) \ar[r]^-{Br_p(f \wedge g)}
	\ar[d]^-{\delta_p}
	& Br_p(R \wedge R) \ar[u]_-{Br_p \phi} \ar[d]^-{\delta_p} \\
D_{1,p} S^{2i} \wedge Br_p S^{2j-1} \ar[r]^-{\sigma_1 \wedge 1}
& Br_p S^{2i} \wedge Br_p S^{2j-1} \ar[r]^-{Br_p f \wedge Br_p g}
	\ar[d]^-{\sigma_2 \wedge \sigma_2}
	& Br_p R \wedge Br_p R \ar[r]^-{\theta \wedge \theta}
	\ar[d]^-{\sigma_2 \wedge \sigma_2}
	& R \wedge R \ar[uu]^-{\phi} \\
& D_{3,p} S^{2i} \wedge D_{3,p} S^{2j-1} \ar[r]^-{D_{3,p} f \wedge D_{3,p} g}
	& D_{3,p} R \wedge D_{3,p} R \ar[ur]_-{\theta \wedge \theta}
}
}
$$

We may view the $E_3$ ring spectrum~$R$ as an $E_2$ algebra in the
category of $E_1$ ring spectra.  The ring spectrum pairing $\phi \: R
\wedge R \to R$ is then an $E_2$ ring spectrum map, so that the right
hand rectangle commutes.  Moreover, the right hand triangle commutes,
because the $E_3$ operad action extends the $E_2$ action.

Let $f \: S^{2i} \to R$ and $g \: S^{2j-1} \to R$ be maps representing~$x$
and~$y$.  The composite
$$
f \cdot g \: S^{2k-1} \cong S^{2i} \wedge S^{2j-1}
	\overset{f \wedge g}\longto R \wedge R
	\overset{\phi}\longto R
$$
then represents~$xy$, and the upper square commutes by functoriality of
the braided-extended power.  The central and lower squares commute by
naturality of~$\delta_p$ and~$\sigma_2$.

We do not know whether the left hand rectangle commutes.  However, we
do claim that the two composites $\Sigma^{2pk-1} DV(0) \to Br_p S^{2i}
\wedge Br_p S^{2j-1}$ become homotopic after composition with $\sigma_2
\wedge \sigma_2$.  This implies that the composite along the upper edge,
which is adjoint to the map representing $P^k(xy)$, is homotopic to
the composite along the left hand, lower and right hand edges, which in
turn is homotopic to the central composite via~$Br_p f \wedge Br_p g$,
and the latter is adjoint to the map representing $x^p P^j(y)$.

To justify the claim, we compute in homology.  Recall the
expression~\eqref{eq:HBrpSodd} for $H_* Br_p S^{2k-1}$, which has an
evident analogue for $H_* Br_p S^{2j-1}$.  In the case $X = S^{2i}$,
with $H_* X = \bF_p\{x_{2i}\}$,
\begin{equation} \label{eq:HBrpSeven}
H_* Br_p S^{2i} = \bF_p\{ x_{2i}^p, \alpha_{2pi+1}\}
\end{equation}
follows from~\cite{CLM76}*{Thm.~III.5.2}.  Here $\alpha_{2pi+1} = -
x_{2i}^{p-2} [x_{2i}, x_{2i}]_1$ is a class given in terms of
the $E_2$ Browder bracket, and $\beta \alpha_{2pi+1} = 0$ according
to~\cite{CLM76}*{Thm.~III.1.2(7)}.  Note that $\alpha_{2pi+1}$ maps to
zero under~$\sigma_2$.

Along one route, the right $\cA$-module generator $x_{2pk-1}$ in $H_{2pk-1}
\Sigma^{2pk-1} DV(0)$ maps to $x_{2pi} \otimes x_{2pj-1}$ in the
homology of $S^{2pi} \wedge \Sigma^{2pj-1} DV(0)$, and thereafter to
$x_{2i}^p \otimes Q^j(x_{2j-1})$ in the homologies of $D_{1,p} S^{2i}
\wedge Br_p S^{2j-1}$, $Br_p S^{2i} \wedge Br_p S^{2j-1}$ and $D_{3,p}
S^{2i} \wedge D_{3,p} S^{2j-1}$.

Along the other route, $x_{2pk-1}$ maps to $Q^k(x_{2k-1})$ in the homology of
$Br_p S^{2k-1}$, and to $Q^k(x_{2i} \otimes x_{2j-1})$ in the homologies
of $Br_p(S^{2i} \wedge S^{2j-1})$ and~$D_{3,p}(S^{2i} \wedge S^{2j-1})$.
By the $E_3$ ring spectrum Cartan formula~\cite{CLM76}*{Thm.~III.1.1(4)}
it maps to
$$
Q^i(x_{2i}) \otimes Q^j(x_{2j-1}) = x_{2i}^p \otimes Q^j(x_{2j-1})
$$
in the homology of $D_{3,p} S^{2i} \wedge D_{3,p} S^{2j-1}$.

It follows that the two composites $\bar\ell_1, \bar\ell_2 \:
\Sigma^{2pk-1} DV(0) \to D_{3,p} S^{2i} \wedge D_{3,p} S^{2j-1}$ induce
the same homomorphism in homology.  Hence their adjoints $\ell_1,
\ell_2 \: S^{2pk-1} \to V(0) \wedge D_{3,p} S^{2i} \wedge D_{3,p}
S^{2j-1}$ also agree in homology.  Since $D_{3,p} S^{2i} \wedge
D_{3,p} S^{2j-1}$ is $(2pk-3)$-connected and $h_0 \: V(0) \to H$ is
$(2p-3)$-connected, it follows that $\ell_1$ and $\ell_2$ are homotopic.
Therefore $\bar\ell_1$ and $\bar\ell_2$ are also homotopic.
\end{proof}

\begin{remark}
This proof also shows that
$$
\delta_{p*} Q^k(x_{2i} \otimes x_{2j-1})
	= x_{2i}^p \otimes Q^j(x_{2j-1})
	+ c \cdot \alpha_{2pi+1} \otimes \beta Q^j(x_{2j-1})
$$
in the homology of $Br_p S^{2i} \wedge Br_p S^{2j-1}$, for some unknown
coefficient~$c \in \bF_p$.  If $c \ne 0$ then the two maps $\Sigma^{2pk-1}
DV(0) \to Br_p S^{2i} \wedge Br_p S^{2j-1}$ induce different homomorphisms
in homology, and the left hand rectangle does not commute.
\end{remark}

\begin{corollary}
Let
$$
R \overset{s}\longto T \overset{r}\longto R
$$
be spectrum maps with $rs$ homotopic to the identity.  Assume that
$R$ is an $E_2$ ring spectrum, that $T$ is an $E_3$ ring spectrum,
and that $s$ or $r$ is an $E_2$ ring map.  Then $P^k(xy) = x^p P^j(y)$ in
$V(0)_{2pk-1} R$ for $x \in \pi_{2i} R$, $y \in \pi_{2j-1} R$ and $k=i+j$.
\end{corollary}

\begin{proof}
Replace Proposition~\ref{prop:CartanV0V1} with
Proposition~\ref{prop:CartanSV0} in the proof of
Corollary~\ref{cor:CartanV0V1} below.
\end{proof}

We will also need a homotopy Cartan formula for the power operations
from Definition~\ref{def:PkV0V1}.

\begin{proposition} \label{prop:CartanV0V1}
Let $R$ be an $E_\infty$ ring spectrum.  For $x \in V(0)_{2i} R$ and $y
\in V(0)_{2j-1} R$ the relation
$$
P^k(xy) = x^p P^j(y)
$$
holds in $V(1)_{2pk-1} R$, where $k = i+j$.
\end{proposition}

\begin{proof}
We use the following nearly-commutative diagram, where
$$
D_p X = \cC_\infty(p) \ltimes_{\Sigma_p} X^{\wedge p}
$$
denotes the $p$-th (unqualified) extended power, and $\sigma'_2 \: Br_p
X \to D_p X$ is the infinite stabilization map.  We write $\mu_0^p \:
V(0)^{\wedge p} \to V(0)$ for the $(p-1)$-fold iterate of the ring
spectrum multiplication, and let $m = \mu_1 (i_1 \wedge 1) \: V(0)
\wedge V(1) \to V(1)$ denote the left $V(0)$-module action on~$V(1)$.

$$
\scalebox{0.67}{
\xymatrix@C+1.5pc{
\Sigma^{2pk-1} DV(1) \ar[r]^-{\bar\eta_1}
	\ar[d]^-{Dm}
& Br_p \Sigma^{2k-1} DV(0) \ar[r]^-{Br_p \overline{f \cdot g}} \ar[d]^-{Br_p D\mu_0}
	& Br_p R \ar[r]^-{\theta}
	& R \\
\Sigma^{2pi} DV(0) \wedge \Sigma^{2pj-1} DV(1) \ar[d]^-{D\mu_0^p \wedge \bar\eta_1}
& Br_p(\Sigma^{2i} DV(0) \wedge \Sigma^{2j-1} DV(0)) \ar[r]^-{Br_p(\bar f \wedge \bar g)}
	\ar[d]^-{\delta_p}
	& Br_p(R \wedge R) \ar[u]_-{Br_p \phi} \ar[d]^-{\delta_p} \\
D_{1,p} \Sigma^{2i} DV(0) \wedge Br_p \Sigma^{2j-1} DV(0) \ar[r]^-{\sigma_1 \wedge 1}
& Br_p \Sigma^{2i} DV(0) \wedge Br_p \Sigma^{2j-1} DV(0) \ar[r]^-{Br_p \bar f \wedge Br_p \bar g}
	\ar[d]^-{\sigma'_2 \wedge \sigma'_2}
	& Br_p R \wedge Br_p R \ar[r]^-{\theta \wedge \theta}
	\ar[d]^(0.45){\sigma'_2 \wedge \sigma'_2}
	& R \wedge R \ar[uu]^-{\phi} \\
& D_p \Sigma^{2i} DV(0) \wedge D_p \Sigma^{2j-1} DV(0) \ar[r]^-{D_p \bar f \wedge D_p \bar g}
	& D_p R \wedge D_p R \ar[ur]_-{\theta \wedge \theta}
}
}
$$

The right hand rectangle and triangle commute as before, replacing $E_3$
by~$E_\infty$.

Let $f \: S^{2i} \to V(0) \wedge R$ and $g \: S^{2j-1} \to V(0) \wedge R$ be maps representing~$x$
and~$y$, with adjoints $\bar f \: \Sigma^{2i} DV(0) \to R$
and $\bar g \: \Sigma^{2j-1} DV(0) \to R$.  The composite
$$
\overline{f \cdot g} \: \Sigma^{2k-1} DV(0)
	\overset{D\mu_0}\longto \Sigma^{2i} DV(0) \wedge \Sigma^{2j-1} DV(0)
	\overset{\bar f \wedge \bar g}\longto R \wedge R
	\overset{\phi}\longto R
$$
is then adjoint to the map $f \cdot g \: S^{2k-1} \to V(0) \wedge R$
that represents~$xy$, and the upper square commutes by functoriality of
the braided-extended power.  The central and lower squares commute by
naturality of~$\delta_p$ and~$\sigma'_2$.

As before, we do not know whether the left hand rectangle commutes.
However, we claim that the two composites
$$
\Sigma^{2pk-1} DV(1) \longto Br_p \Sigma^{2i} DV(0)
	\wedge Br_p \Sigma^{2j-1} DV(0)
$$
become homotopic after composition with $\sigma'_2 \wedge \sigma'_2$
to
$$
W = D_p \Sigma^{2i} DV(0) \wedge D_p \Sigma^{2j-1} DV(0) \,.
$$
This implies that the composite along the upper edge, which is adjoint
to the map representing $P^k(xy)$, is homotopic to the composite along
the left hand, lower and right hand edges, which in turn is homotopic
to the central composite via~$Br_p \bar f \wedge Br_p \bar g$, and the
latter is adjoint to the map representing $x^p P^j(y)$.

To justify the claim we first compute in homology,
using~\cite{CLM76}*{Thm.~I.4.1}.  Writing $H_* \Sigma^{2i} DV(0) =
\bF_p\{x_{2i-1}, x_{2i}\}$ and $H_* \Sigma^{2j-1} DV(0) = \bF_p\{x_{2j-2},
x_{2j-1}\}$, with $\beta x_{2i} = x_{2i-1}$ and $\beta x_{2j-1} =
x_{2j-2}$, we have
\begin{multline*}
H_* D_p \Sigma^{2i} DV(0) = \bF_p\{\beta Q^i(x_{2i-1}), Q^i(x_{2i-1}),
	x_{2i-1} x_{2i}^{p-1}, x_{2i}^p, \\
	\beta Q^{i+1}(x_{2i-1}), Q^{i+1}(x_{2i-1}), \dots\}
\end{multline*}
in degrees $* \ge 2pi-2$,
and
\begin{multline*}
H_* D_p \Sigma^{2j-1} DV(0) = \bF_p\{x_{2j-2}^p, x_{2j-2}^{p-1} x_{2j-1}, \\
	\beta Q^j(x_{2j-2}), Q^j(x_{2j-2}),
	\beta Q^j(x_{2j-1}), Q^j(x_{2j-1}), \dots\}
\end{multline*}
in degrees $* \ge 2pj-2p$.  Their tensor product is $H_* W$,
which is concentrated in degrees $2pk-2p-2 \le * \le 2pk-2p+1$
and $* \ge 2pk-4$.

On one hand, the right $\cA$-module generator $x_{2pk-1}$ in $H_{2pk-1}
\Sigma^{2pk-1} DV(1)$ maps to $x_{2pi} \otimes x_{2pj-1}$ in the homology
of $\Sigma^{2pi} DV(0) \wedge \Sigma^{2pj-1} DV(1)$, and thereafter
to $x_{2i}^p \otimes Q^j(x_{2j-1})$ in the homologies of $D_{1,p}
\Sigma^{2i} DV(0) \wedge Br_p \Sigma^{2j-1} DV(0)$, $Br_p \Sigma^{2i}
DV(0) \wedge Br_p \Sigma^{2j-1} DV(0)$ and~$W$.  On the other hand,
$x_{2pk-1}$ maps to $Q^k(x_{2k-1})$ in the homology of $Br_p \Sigma^{2k-1}
DV(0)$, and to $Q^k(x_{2i} \otimes x_{2j-1})$ in the homologies of $Br_p
(\Sigma^{2i} DV(0) \wedge \Sigma^{2j-1} DV(0))$ and $D_p (\Sigma^{2i}
DV(0) \wedge \Sigma^{2j-1} DV(0))$.  By the $E_\infty$ ring spectrum
Cartan formula~\cite{CLM76}*{Thm.~I.1.1(6)} it maps to $Q^i(x_{2i})
\otimes Q^j(x_{2j-1}) = x_{2i}^p \otimes Q^j(x_{2j-1})$ in $H_* W$.

It follows that the two composites $\bar m_1, \bar m_2 \: \Sigma^{2pk-1}
DV(1) \to W$ induce the same homomorphism in homology.  Let $\bar m =
\bar m_2 - \bar m_1$ be their difference, inducing zero in homology.  The
homological Atiyah--Hirzebruch spectral sequence for $V(1)_* W = [DV(1),
W]_*$ shows that $\bar m$ is null-homotopic, since $H_{2pk-2p+2}(W;
\pi_{2p-3} V(1)) = H_{2pk-2p+2} W = 0$.  Hence $\bar m_1$ and $\bar m_2$
are homotopic, as claimed.
\end{proof}

\begin{remark}
A similar proof goes through if $R$ is an $E_n$ ring spectrum
with $n\ge6$, replacing $W$ with $W_n = D_{n,p} \Sigma^{2i} DV(0)
\wedge D_{n,p} \Sigma^{2j-1} DV(0)$.  For $3 \le n \le 5$ the group
$H_{2pk-2p+2} W_n$ will be nonzero, due to the presence of $E_n$ Browder
bracket terms in this degree, so that $\bar m$ might map the top cell of
$\Sigma^{2pk-1} DV(1)$ via $\alpha_1$ to a $(2pk-2p+2)$-cell of~$W_n$,
and hence be essential.  For simplicity we assume $n=\infty$, since this
will suffice for our application.
\end{remark}

\begin{corollary} \label{cor:CartanV0V1}
Let
$$
R \overset{s}\longto T \overset{r}\longto R
$$
be spectrum maps with $rs$ homotopic to the identity.  Assume that $R$
is an $E_2$ ring spectrum, that $T$ is an $E_\infty$ ring spectrum,
and that~$r$ or~$s$ is an $E_2$ ring map.  Then $P^k(xy) = x^p P^j(y)$
in $V(1)_{2pk-1} R$ for $x \in V(0)_{2i} R$, $y \in V(0)_{2j-1} R$
and $k=i+j$.
\end{corollary}

\begin{proof}
Apply Proposition~\ref{prop:CartanV0V1} for~$T$ to see that
$$
r_*(P^k(s_* x \cdot s_* y)) = r_*((s_* x)^p \cdot P^j(s_* y))
$$
in $V(1)_{2k-1}(R)$.  If $r$ is an $E_2$ ring map, then naturality of
the products and homotopy power operations with respect to~$r$ implies
$P^k(r_* s_* x \cdot r_* s_* y) = (r_* s_* x)^p \cdot P^j(r_* s_* y)$.
If $s$ is an $E_2$ ring map, then naturality of the products and homotopy
power operations with respect to~$s$ implies $r_* s_*(P^k(x \cdot y)) =
r_* s_*(x^p \cdot P^j(y))$.  In either case the conclusion follows from
$r_* s_* = 1$.
\end{proof}

\section{Some $V(0)$- and $V(1)$-homotopy classes}
\label{sec:V0V1classes}

The homotopy power operations introduced in Definitions~\ref{def:PkSV0}
and~\ref{def:PkV0V1} apply for $R = S$ with its $E_\infty$ ring structure.
The $E_2$-term of its mod~$p$ Adams spectral sequence
$$
E_2^{s,t}(S) = \Ext_{\cA_*}^{s,t}(\bF_p, \bF_p)
	\Longrightarrow \pi_{t-s}(S)^\wedge_p
$$
contains classes traditionally denoted
$$
a_0 = [\tau_0]
\qquad\text{and}\qquad
h_i = [\xi_1^{p^i}]
$$
for $i\ge0$, in bidegrees $(s,t) = (1,1)$ and $(1,2p^i(p-1))$,
respectively.  Here $\tau_0$ is dual to $\beta$ and $a_0$ detects $p \in
\pi_0(S)^\wedge_p \cong \bZ_p$, while $\xi_1^{p^i}$ is dual to $\cP^{p^i}$
and $h_0$ detects the generator $\alpha_1 \in \pi_{2p-3}(S)^\wedge_p
\cong \bZ/p$.  The classes $h_i$ for $i\ge1$ support nonzero
$d_2$-differentials~\cite{Liu62} in the Adams spectral sequence for~$S$,
but some of these map to permanent cycles in the corresponding spectral
sequences for $V(0)$ and~$V(1)$, detecting interesting homotopy classes.

\begin{definition}
Let
$$
\beta^\circ_1 = P^{p-1}(\alpha_1) \in \pi_{2p^2-2p-1} V(0)
$$
and
$$
\gamma^\circ_1 = P^{p^2-p}(\beta^\circ_1) \in \pi_{2p^3-2p^2-1} V(1) \,.
$$
\end{definition}

The ring/circle superscripts indicate that these classes are constructed
using the $E_2$ ring spectrum structure.

\begin{lemma} \label{lem:betagammadetection}
The classes $\beta^\circ_1$ and $\gamma^\circ_1$ are detected by $i_0(h_1) =
[\xi_1^p]$ and $i_1 i_0(h_2) = [\xi_1^{p^2}]$ in the Adams spectral
sequences for $V(0)$ and~$V(1)$, respectively.
\end{lemma}

\begin{proof}
The case of $\beta^\circ_1$ is due to Toda~\cite{Tod68}*{Lem.~4}.
It suffices to prove that the dual Steenrod operation~$\cP^p_*$ acts
nontrivially in the homology of the mapping cone~$C\bar\beta$, where
$$
\bar\beta \: \Sigma^{2p^2-2p-1} DV(0) \simeq Br_p S^{2p-3}
	\overset{Br_p \alpha_1}\longto Br_p S
	\overset{\theta}\longto S
$$
is left adjoint to $\beta^\circ_1$.
There are natural maps
$$
C\bar\beta \overset{\tilde\theta}\longfrom C(Br_p \alpha_1)
	\overset{D_{\alpha_1}}\longto Br_p(C \alpha_1)
$$
that are induced by $\theta$ and the canonical null-homotopy in a cone,
respectively.  By an analog of~\cite{Tod68}*{Thm.~2} for braided-extended
powers we have
$$
D_{\alpha_1*} ((e_{p-1} \otimes x^{\otimes p})^\wedge)
	= e_0 \otimes (x^\wedge)^{\otimes p} \,,
$$
up to a unit in~$\bF_p$, where $x^\wedge \in H_{2p-2} C\alpha_1$
lifts the generator $x \in H_{2p-3} S^{2p-3}$ and $(e_{p-1}
\otimes x^{\otimes p})^\wedge \in H_{2p^2-2p} C(Br_p \alpha_1)$
lifts $e_{p-1} \otimes x^{\otimes p} \in H_{2p^2-2p-1} Br_p S^{2p-3}$.
Since $\cP^1_*(x^\wedge)$ generates $H_0 C\alpha_1$, it follows from the
homology Cartan formula that $\cP^p_*(e_0 \otimes (x^\wedge)^{\otimes p}) = e_0
\otimes \cP^1_*(x^\wedge)^{\otimes p}$ generates $H_0 Br_p(C\alpha_1)$.
By naturality with respect to Toda's map $D_{\alpha_1}$ it follows that
$\cP^p_*((e_{p-1} \otimes x^{\otimes p})^\wedge)$ generates $H_0 C(Br_p
\alpha_1)$, and by naturality with respect to $\tilde\theta$ it follows that
$$
\cP^p_* \: H_{2p^2-2p} C\bar\beta \longto H_0 C\bar\beta
$$
is nonzero.

The proof for $\gamma^\circ_1$ is similar.  It suffices to prove that the
dual Steenrod operation~$\cP^{p^2}_*$ acts nontrivially in the homology
of the mapping cone~$C\bar\gamma$, where
$$
\bar\gamma \: \Sigma^{2p^3-2p^2-1} DV(1)
	\overset{\bar\eta_1}\longto Br_p(\Sigma^{2p^2-2p-1} DV(0))
	\overset{Br_p \bar\beta}\longto Br_p S
	\overset{\theta}\longto S
$$
is left adjoint to $\gamma^\circ_1$.
Here $\bar\eta_1$ was defined in~\eqref{eq:bareta1}.
There are natural maps
$$
C\bar\gamma \overset{\tilde\eta_1}\longto C(\theta \circ Br_p \bar\beta)
	\overset{\tilde\theta}\longfrom C(Br_p \bar\beta)
	\overset{D_{\bar\beta}}\longto Br_p(C\bar\beta)
$$
induced by $\bar\eta_1$, $\theta$ and the canonical null-homotopy,
respectively.  By~\cite{Tod68}*{Thm.~2} again, we have
$$
D_{\bar\beta*}((e_{p-1} \otimes y^{\otimes p})^\wedge)
	 = e_0 \otimes (y^\wedge)^{\otimes p} \,,
$$
up to a unit in $\bF_p$, where $y^\wedge \in H_{2p^2-2p} C\bar\beta$
lifts the generator
$$
y \in H_{2p^2-2p-1}(\Sigma^{2p^2-2p-1} DV(0))
$$
and $(e_{p-1} \otimes y^{\otimes p})^\wedge \in H_{2p^3-2p^2} C(Br_p
\bar\beta)$ lifts
$$
e_{p-1} \otimes y^{\otimes p}
	\in H_{2p^3-2p^2-1} Br_p(\Sigma^{2p^2-2p-1} DV(0)) \,.
$$
Since $\cP^p_*(y^\wedge)$ generates $H_0 C\bar\beta$ it follows
that $\cP^{p^2}_*(e_0 \otimes (y^\wedge)^{\otimes p})
= e_0 \otimes \cP^p_*(y^\wedge)^{\otimes p}$ generates
$H_0 Br_p(C\bar\beta)$.  Naturality with respect to $D_{\bar\beta}$
implies that
$$
\cP^{p^2}_*((e_{p-1} \otimes y^{\otimes p})^\wedge)
$$
generates $H_0 C(Br_p \bar\beta)$, and naturality with respect
to $\tilde\theta$ and~$\tilde\eta_1$ implies that
$$
\cP^{p^2}_* \: H_{2p^3-2p^2} C\bar\gamma \longto H_0 C\bar\gamma
$$
is nonzero.
\end{proof}

The first Greek letter element $\alpha_1 \in \pi_{2p-3} S$ is the
image under $j_0 \: V(0) \to S^1$ of a class $v_1 \in \pi_{2p-2} V(0)$
detected by the class of the cobar cocycle $[\tau_1] 1 + [\xi_1] \tau_0$
in bidegree $(s,t) = (1, 2p-1)$ of the Adams spectral sequence
$$
E_2^{s,t}(Y) = \Ext_{\cA_*}^{s,t}(\bF_p, H_* Y)
	\Longrightarrow \pi_{t-s}(Y^\wedge_p)
$$
for $Y = V(0)$.  Similarly, $\beta_1 \in \pi_{2p^2-2p-2} S$ is the image
under $j_0 j_1 \: V(1) \to S^{2p}$ of a class $v_2 \in \pi_{2p^2-2}
V(1)$, and $\gamma_1 \in \pi_{2p^3-2p^2-2p-1} S$ is the image under $j_0
j_1 j_2 \: V(2) \to S^{2p^2+2p-1}$ of a class $v_3 \in \pi_{2p^3-2} V(2)$.

\begin{lemma}
The groups $\pi_{2p-2} V(0) \cong \bZ/p$ for $p\ge3$,
$\pi_{2p^2-2} V(1) \cong \bZ/p$ for $p\ge3$ and
$\pi_{2p^3-2} V(2) \cong \bZ/p$ for $p\ge5$
are generated by classes $v_1$, $v_2$ and $v_3$, respectively,
each in Adams filtration~$1$.
\end{lemma}

\begin{proof}
The claim for $V(0)$ is well known.  The claim for $V(1)$ is contained
in~\cite{Tod71}*{Thm.~5.2, (5.7)}.  The claim for $V(2)$ can be deduced
from~\cite{Tod71}*{\S3}, as follows.  Let $\cP \subset \cA$ be the sub
Hopf algebra of the mod~$p$ Steenrod algebra generated by the Steenrod
operations~$\cP^i$.  Let $K = \bF_p\{\cQ_3, \beta \cQ_3, \dots\}$ be
the kernel of the surjection $\cA \otimes_{\cP} \bF_p \to H^* V(2) =
E(\beta, \cQ_1, \cQ_2)$, where $\cQ_i$ denotes the Milnor primitive,
and consider the long exact sequence
$$
\dots \to \Ext_{\cA}^{s-1,t}(K, \bF_p)
	\overset{\delta}\longto \Ext_{\cA}^{s,t}(H^* V(2), \bF_p)
	\longto \Ext_{\cP}^{s,t}(\bF_p, \bF_p) \to \dots \,.
$$
Using the May spectral sequence, Toda~\cite{Tod71}*{\S3} calculated
an upper bound for $\Ext_{\cP}^{s,t}(\bF_p, \bF_p)$ in the range $t
< 2(p^2+2p+3)(p-1)+4$, which shows that these groups are trivial in
topological degrees $t-s = 2p^3-3$ and~$2p^3-2$.  Hence $\delta(\cQ_3^*)$
in cohomological degree~$s=1$ is the only generator of $E_2(V(2)) =
\Ext_{\cA}(H^* V(2), \bF_p)$ in topological degree~$2p^3-2$.  Moreover,
there is no possible target for an Adams differential on this class,
which must therefore detect~$v_3$.
\end{proof}

\begin{lemma} \label{lem:j1v1beta-j2v2gamma}
For $p\ge3$, the classes $\beta^\circ_1$ and $j_1(v_2) = \beta'_1$ in
$\pi_{2p^2-2p-1} V(0)$ agree modulo (a nonzero multiple of)
$\alpha_1 v_1^{p-1}$.  Hence $i_1(\beta^\circ_1) = i_1(\beta'_1)$ in
$\pi_{2p^2-2p-1} V(1)$, and $j_0(\beta^\circ_1) = \beta_1 = j_0(\beta'_1)$
in $\pi_{2p^2-2p-2} S$ is the first element in the $\beta$-family.

For $p\ge5$, the classes $\gamma^\circ_1$ and $j_2(v_3) = \gamma''_1$ in
$\pi_{2p^3-2p^2-1} V(1)$ agree modulo~$\alpha_1 v_2^{p-1}$.
Hence $i_2(\gamma^\circ_1) = i_2(\gamma''_1)$ in $\pi_{2p^3-2p^2-1}
V(2)$.
\end{lemma}

\begin{proof}
The cobar cocycle $[\tau_2]1 + [\xi_2]\tau_0 + [\xi_1^p]\tau_1$ detects
$v_2 \in \pi_{2p^2-2} V(1)$.  The $\cA_*$-comodule homomorphism $j_{1*}
\: H_* V(1) \to H_{*-2p+1} V(0)$ sends $1$ and $\tau_0$ to zero, and maps
$\tau_1$ to $1$.  Hence $j_1 \: E_2^{1,*}(V(1)) \to E_2^{1,*-2p+1}(V(0))$
sends $[\tau_2]1 + [\xi_2]\tau_0 + [\xi_1^p]\tau_1$ to $[\xi_1^p]1
= i_0(h_1)$.  This is also the class detecting $\beta^\circ_1$, by
Lemma~\ref{lem:betagammadetection}.  Therefore $j_1(v_2) = \beta'_1$
and $\beta^\circ_1$ agree modulo Adams filtration $\ge 2$, i.e., modulo
$\alpha_1 v_1^{p-1}$.  (We will see in Remark~\ref{rem:v1betacirc1} that
$v_1 \beta^\circ_1 \ne 0$, while $v_1 \beta'_1 = 0$, so $\beta^\circ_1 -
\beta'_1$ is a nonzero multiple of $\alpha_1 v_1^{p-1}$.)  Nonetheless,
$j_0(\beta^\circ_1) = j_0(\beta'_1)$, since $j_0(\alpha_1 v_1^{p-1})
\doteq \alpha_1 \alpha_{p-1} = 0$.

The cobar cocycle $[\tau_3]1 + [\xi_3]\tau_0 +
[\xi_2^p]\tau_1 + [\xi_1^{p^2}]\tau_2$ detects $v_3 \in
\pi_{2p^3-2} V(2)$.  The $\cA_*$-comodule homomorphism $j_{2*} \:
H_* V(2) \to H_{*-2p^2+1} V(1)$ sends $1$, $\tau_0$ and~$\tau_1$
to zero, and maps $\tau_2$ to $1$.  Hence $j_2 \: E_2^{1,*}(V(2))
\to E_2^{1,*-2p^2+1}(V(1))$ sends $[\tau_3]1 + [\xi_3]\tau_0 +
[\xi_2^p]\tau_1 + [\xi_1^{p^2}]\tau_2$ to $[\xi_1^{p^2}]1 = i_1
i_0(h_2)$.  This is also the class detecting $\gamma^\circ_1$, by
Lemma~\ref{lem:betagammadetection}.  Therefore $j_2(v_3) = \gamma''_1$
and $\gamma^\circ_1$ agree modulo Adams filtration $\ge 2$, i.e., modulo
$\alpha_1 v_2^{p-1}$.
\end{proof}

\begin{remark}
One way to see that $\alpha_1 v_1^{p-1}$ and $\alpha_1 v_2^{p-1}$ generate
Adams filtration $\ge 2$ in $\pi_{2p^2-2p-1} V(0)$ and $\pi_{2p^3-2p^2-1}
V(1)$, respectively, is to compare with the corresponding Adams--Novikov
spectral sequences.  By the beginning calculations in~\cite{Rav04}*{\S4.4}
the classes $h_{11}$ and $h_{10} v_1^{p-1}$ generate the Adams--Novikov
$E_2$-term for $V(0)$ in topological degree~$2p^2-2p-1$, while the
classes $h_{12}$ and $h_{10} v_2^{p-1}$ generate the Adams--Novikov
$E_2$-term for $V(1)$ in topological degree~$2p^3-2p^2-1$.  The
formula~$\eta_R(v_{n+1}) = v_{n+1} + v_n t_1^{p^n} - v_n^p t_1$ in $BP_*
BP/I_n$ from~\cite{Rav04}*{Cor.~4.3.21} shows that $j_n(v_{n+1})$ in
$\pi_* V(n-1)$ is detected by $h_{1n} - h_{10} v_n^{p-1}$, when $v_{n+1}
\in \pi_* V(n)$ exists, while $\alpha_1 v_n^{p-1}$ is detected by $h_{10}
v_n^{p-1}$.
\end{remark}

The homotopy power operations also apply to $R = K(BP)$ and $R = THH(BP)$,
with their $E_3$ ring structures derived from the $E_4$ ring structure
on $BP$, and to $R = K(BP\<n\>)$ and $R = THH(BP\<n\>)$,
with their $E_2$ ring structures derived from the $E_3$ ring structure
on~$BP\<n\>$.  (For $n \le 1$ these are $E_\infty$ ring structures.)

$$
\xymatrix{
\pi_* K(BP) \ar[r] \ar[d]_-{tr}
	& \pi_* K(BP\<n\>) \ar[r] \ar[d]_-{tr}
	& \pi_* K(\bZ_{(p)}) \ar[d]_-{tr} \\
\pi_* THH(BP) \ar[r]
	& \pi_* THH(BP\<n\>) \ar[r]
	& \pi_* THH(\bZ_{(p)})
}
$$

According to~\cite{BM94}*{Thm.~10.14} and~\cite{Rog98}*{Thm.~1.1} we can
find a class $\lambda_1^K \in \pi_{2p-1} K(\bZ)$ with $tr(\lambda_1^K) =
\lambda_1 \in \pi_{2p-1} THH(\bZ)$, having Hurewicz image $h(\lambda_1)
= \sigma\bar\xi_1 \in H_{2p-1} THH(\bZ)$.  The same statements apply with
$\bZ$ replaced by $BP\<0\> = H\bZ_{(p)}$.  The $E_4$ ring spectrum map $BP
\to H\bZ_{(p)}$ is $(2p-2)$-connected, and induces a $(2p-1)$-connected
map $K(BP) \to K(\bZ_{(p)})$ by~\cite{BM94}*{Prop.~10.9}.  Hence we
can lift $\lambda_1^K$ to $\pi_{2p-1} K(BP)$.  Its trace image
$tr(\lambda_1^K) \in \pi_{2p-1} THH(BP) = \bZ_{(p)}\{\lambda_1\}$
then maps to the generator $\lambda_1 \in \pi_{2p-1} THH(\bZ_{(p)})
\cong \bZ/p$.  It follows that we can scale the choice of $\lambda_1^K
\in \pi_{2p-1} K(BP)$ by a $p$-local unit so as to ensure that
$tr(\lambda_1^K) = \lambda_1 $ in $\pi_{2p-1} THH(BP)$.

\begin{definition} \label{def:lambda1K}
We fix a choice of a class $\lambda_1^K \in \pi_{2p-1} K(BP)$ with
$tr(\lambda_1^K) = \lambda_1$ in $\pi_{2p-1} THH(BP)$.  These map to
classes with the same names in $\pi_{2p-1} K(BP\<n\>)$ and $\pi_{2p-1}
THH(BP\<n\>)$, respectively, for each $n\ge0$.
\end{definition}

The choice of $\lambda_1^K \in \pi_{2p-1} K(BP)$ made here
is equivalent to the selection of $\lambda_1^K \in \pi_{2p-1}
K(BP\<1\>)$ discussed in~\cite{AR02}*{\S1.2}, since $BP \to BP\<1\>
= \ell$ is $(2p^2-2)$-connected, so that $K(BP) \to K(BP\<1\>)$ is
$(2p^2-1)$-connected.

\begin{definition} \label{def:lambda2K}
Let $\lambda_2^K = P^p(\lambda_1^K) \in V(0)_{2p^2-1} K(BP)$, mapping
to classes with the same name in $V(0)_{2p^2-1} K(BP\<n\>)$ for each
$n\ge1$.
\end{definition}

By naturality of $P^p$ for $E_2$ ring spectrum maps, this definition
agrees with the case $n=1$ discussed in~\cite{AR02}*{\S1.7}.

\begin{lemma} \label{lem:trlambda2K}
The classes $tr(\lambda_2^K)$ and $i_0(\lambda_2)$ in $V(0)_{2p^2-1}
THH(BP)$ both have Hurewicz image $\sigma\bar\xi_2$ in $H_{2p^2-1}
THH(BP)$.  Hence they agree modulo~$v_1^p \lambda_1$, and have the
same image in $V(1)_{2p^2-1} THH(BP)$.
\end{lemma}

\begin{proof}
We have $tr(\lambda_2^K) = tr(P^p(\lambda_1^K)) = P^p(tr(\lambda_1^K))
= P^p(\lambda_1)$ by naturality of~$P^p$ with respect to~$tr$,
and $h_0 P^p(\lambda_1) = Q^p h(\lambda_1) = Q^p(\sigma\bar\xi_1)$
by Lemma~\ref{lem:h0PkeqQkh}.  Moreover, $Q^p(\sigma\bar\xi_1) =
\sigma Q^p(\bar\xi_1) = \sigma\bar\xi_2$ by~\cite{AR05}*{Prop.~5.9}
and~\cite{BMMS86}*{Thm.~III.2.3}.
\end{proof}

\begin{definition} \label{def:lambda3K}
Let $\lambda_3^K = P^{p^2}(\lambda_2^K) \in V(1)_{2p^3-1} K(BP)$, mapping
to classes with the same name in $V(1)_{2p^3-1} K(BP\<n\>)$ for each
$n\ge2$.
\end{definition}

\begin{lemma} \label{lem:trlambda3K}
The classes
$$
tr(\lambda_3^K) \ ,\ i_1 i_0(\lambda_3) \ ,\ P^{p^2}(i_0(\lambda_2))
$$
in $V(1)_{2p^3-1} THH(BP)$ all have
Hurewicz image $\sigma\bar\xi_3$ in $H_{2p^3-1} THH(BP)$.  Hence they
agree modulo~$v_2^p \lambda_1$, and have the same image in $V(2)_{2p^3-1}
THH(BP)$.
\end{lemma}

\begin{proof}
We have $tr(\lambda_3^K) = tr(P^{p^2}(\lambda_2^K)) =
P^{p^2}(tr(\lambda_2^K))$ by naturality of~$P^{p^2}$ with respect to~$tr$,
and $h_1 P^{p^2}(tr(\lambda_2^K)) = Q^{p^2} h_0(tr(\lambda_2^K))
= Q^{p^2}(\sigma\bar\xi_2)$ by Lemmas~\ref{lem:h1PkeqQkh}
and~\ref{lem:trlambda2K}.
Likewise, $h_1 P^{p^2}(i_0(\lambda_2)) = Q^{p^2} h_0(i_0(\lambda_2))
= Q^{p^2}(\sigma\bar\xi_2)$.
Finally, $Q^{p^2}(\sigma\bar\xi_2) =
\sigma Q^{p^2}(\bar\xi_2) = \sigma\bar\xi_3$ by the same two
references as in the previous lemma.
\end{proof}


Let us summarize these results, for later reference.

\begin{proposition} \label{prop:trace}
Let $p\ge7$.  The trace map $tr \: K(B) \to THH(B)$ induces ring
homomorphisms
\begin{align*}
V(2)_* K(BP) &\longto V(2)_* THH(BP) \\
V(2)_* K(BP\<2\>) &\longto V(2)_* THH(BP\<2\>) \,,
\end{align*}
each mapping $i_2 i_1 i_0(\lambda^K_1)$, $i_2 i_1(\lambda^K_2)$ and
$i_2(\lambda^K_3)$ to $\lambda_1$, $\lambda_2$ and $\lambda_3$,
respectively.
\end{proposition}

\begin{proof}
The claims for $BP$ follow from Definition~\ref{def:lambda1K} and
Lemmas~\ref{lem:trlambda2K} and~\ref{lem:trlambda3K}.  The image
classes in~$V(2)_* THH(BP\<2\>)$ coincide with the classes from
Definition~\ref{def:lam1lam2lam3mu} since their Hurewicz images in $H_*
THH(BP\<2\>)$ agree.
\end{proof}

\section{Approximate homotopy fixed points}

For $C = C_{p^n}$ or $\bT$ we have multiplicative homotopy fixed point
spectral sequences
\begin{align*}
E^2(C) &= H^{-*}(C; V(2)_* THH(B)) \\
	&\Longrightarrow V(2)_* THH(B)^{hC}
\end{align*}
(cf.~\cite{HR}*{\S5})
and multiplicative Tate spectral sequences
\begin{align*}
\hat E^2(C) &= \hat H^{-*}(C; V(2)_* THH(B)) \\
	&\Longrightarrow V(2)_* THH(B)^{tC}
\end{align*}
(cf.~\cite{HR}*{\S6}).  Here $H^*(\bT) = P(t)$ and $\hat H^*(\bT) =
P(t^{\pm1})$, with $t \in H^2 \cong \hat H^2$, while $H^*(C_{p^n}) =
E(u_n) \otimes P(t)$ and $\hat H^*(C_{p^n}) = E(u_n) \otimes P(t^{\pm1})$
with $u_n \in H^1 \cong \hat H^1$.
Note that for $B = BP\<2\>$, each bidegree of~$E^2(C)$ and~$\hat E^2(C)$
is either~$0$ or~$\bF_p$.
This section is devoted to the proof of the following collection of
detection results.

\begin{proposition} \label{prop:unit}
The unit map $S \to K(B)$ and the circle trace map $tr_{\bT} \: K(B)
\to THH(B)^{h\bT}$ induce ring homomorphisms
$$
V(2)_* \longto V(2)_* K(BP) \longto V(2)_* THH(BP)^{h\bT}
	\longto V(2)_* THH(BP\<2\>)^{h\bT}
$$
mapping $i_2 i_1 i_0(\alpha_1)$, $i_2 i_1(\beta^\circ_1)$,
$i_2(\gamma^\circ_1)$ and~$v_3$ to classes detected by $t \lambda_1$, $t^p
\lambda_2$, $t^{p^2} \lambda_3$ and~$t\mu$, respectively.
\end{proposition}

\begin{proof}
By Proposition~\ref{prop:beta-tplambda2} the circle trace image of
$\beta^\circ_1$ is detected by $t^p \lambda_2$ in the $\bT$-homotopy fixed
point spectral sequence for $V(0) \wedge THH(BP)$, hence also for $V(2)
\wedge THH(BP\<2\>)$.

By Proposition~\ref{prop:gamma-tp2lambda3} the image of $\gamma^\circ_1$
is detected by $t^{p^2} \lambda_3$ in the spectral sequence for $V(1)
\wedge THH(BP)$, hence also for $V(2) \wedge THH(BP\<2\>)$.

By Proposition~\ref{prop:v3-tmu} the image of $v_3$ is detected by $t \mu$
in the spectral sequence for $V(2) \wedge THH(BP\<2\>)$.

A simpler case of the latter argument shows that the image of $\alpha_1$
is detected by $t \lambda_1$ in the spectral sequence for $THH(BP)$,
hence also for $V(2) \wedge THH(BP\<2\>)$, but this is also readily
deduced from the previously known case of $THH(\bZ)$.
\end{proof}

\begin{notation}
For any spectral sequence $E^2_{*,*} \Longrightarrow G_*$ and nonzero
element $x \in E^\infty_{*,*}$ we write $\{x\}$ for the coset of elements
$\xi \in G_*$ that are detected by~$x$.  Sometimes we will write $[\![x]\!]$
for a specific choice of such an element~$\xi$, so that $[\![x]\!] \in
\{x\}$.  Similar conventions appear in~\cite{BMT70}*{Prop.~3.1.5}
and~\cite{BR21}*{Thm.~11.61}.
\end{notation}

For each $\bT$-spectrum~$X$ and integer~$m\ge0$ we have an $m$-th
order approximate $\bT$-homotopy fixed point spectral sequence
$$
E^2_{*,*} = \bZ[t]/(t^{m+1}) \otimes \pi_*(X)
	\Longrightarrow \pi_* F(S^{2m+1}_+, X)^{\bT} \,,
$$
obtained by truncating the $\bT$-homotopy fixed point spectral sequence
to (horizontal) filtration degrees $-2m \le * \le 0$.

\begin{proposition} \label{prop:beta-tplambda2}
Consider the $p$-th order spectral sequence
$$
E^2_{*,*} = \bZ[t]/(t^{p+1}) \otimes \pi_* THH(BP)
	\Longrightarrow \pi_* F(S^{2p+1}_+, THH(BP))^{\bT}
$$
for $THH(BP)$, and its analogue for $V(0) \wedge THH(BP)$.  The circle
trace image of $\alpha_1 \in \pi_{2p-3}(S)$ in
$$
\pi_* F(S^{2p+1}_+, THH(BP))^{\bT}
$$
factors as a product $[\![t]\!] \cdot [\![\lambda_1]\!]$, with
$[\![t]\!] \in \{t\}$ and $[\![\lambda_1]\!] \in \{\lambda_1\}$
detected by $t$ and~$\lambda_1$, respectively.  Moreover, the image of
$\beta^\circ_1 \in \pi_{2p^2-2p-1} V(0)$ in
$$
V(0)_* F(S^{2p+1}_+, THH(BP))^{\bT}
$$
is the unique class detected by $t^p \lambda_2$.
\end{proposition}

\begin{proof}
The $p$-th order approximate $\bT$-homotopy fixed point spectral sequence
is multiplicative, and has $E^2$-term
$$
\bZ[t]/(t^{p+1})
	\otimes \bZ_{(p)}\{1, v_1, \lambda_1, v_1^2, v_1 \lambda_1, \dots\}
\,,
$$
with generators as listed in vertical degrees $* < 6p-6$.
Here $d^2(v_1) = t \cdot \sigma(v_1) = t \cdot p \lambda_1$, as
in Proposition~\ref{prop:piTHHBP}, and $E^3 = E^\infty$ in this
range of degrees.  Hence $t$, $\lambda_1$ and $t \lambda_1$ are
all infinite cycles, detecting homotopy classes with indeterminacy
$\bZ_{(p)}\{t^p v_1\}$, $\bZ/p\{t^{p-1} v_1 \lambda_1\}$ and $\bZ/p\{t^p
v_1 \lambda_1\}$, respectively.  The unit map $S \to F(S^{2p+1}_+,
THH(BP))^{\bT}$ takes $\alpha_1$ to a class detected by $t \lambda_1$,
cf.~\cite{Rog98}*{Thm.~1.4}.  Since each element in the indeterminacy
of~$\{t \lambda_1\}$ factors as an element in the indeterminacy of~$\{t\}$
times $\lambda_1$ (and also factors as $t$ times an element in the
indeterminacy of~$\{\lambda_1\}$), it follows that the image of $\alpha_1$
can be factored as a product~$[\![t]\!] \cdot [\![\lambda_1]\!]$ in $\{t\}
\cdot \{\lambda_1\}$.

Let $\lambda^\circ_2 = tr(\lambda_2^K) = P^p(\lambda_1)$
in $V(0)_* THH(BP)$.
By the homotopy Cartan formula from Proposition~\ref{prop:CartanSV0},
applied for the $E_3$ ring spectrum~$F(S^{2p+1}_+, THH(BP))^{\bT}$,
the circle trace image of $\beta^\circ_1 = P^{p-1}(\alpha_1)$ is
$$
P^{p-1}([\![t]\!] \cdot [\![\lambda_1]\!])
	= [\![t]\!]^p \cdot P^p([\![\lambda_1]\!]) \,.
$$
Here $P^p([\![\lambda_1]\!]) \in \{\lambda^\circ_2\}$ is a class detected
by $\lambda^\circ_2$, by naturality of $P^p$ with respect to the edge
homomorphism induced by $F(S^{2p+1}_+, THH(BP))^{\bT} \to THH(BP)$.
It follows that $[\![t]\!]^p \cdot P^p([\![\lambda_1]\!])$ is detected
by $t^p \lambda^\circ_2$, with zero indeterminacy since this class lives
in the lowest filtration degree.

To complete the proof, note that $t^p \lambda^\circ_2 = t^p
\lambda_2$ at the $V(0)$-homotopy $E^3$-term, since these classes
differ by a multiple of $d^2(t^{p-1} v_2) = - t^p v_1^p \lambda_1$
by Proposition~\ref{prop:piTHHBP} and Lemma~\ref{lem:trlambda2K}.
\end{proof}

\begin{proposition} \label{prop:gamma-tp2lambda3}
Consider the $p^2$-th order spectral sequence
$$
E^2_{*,*} = \bZ[t]/(t^{p^2+1}) \otimes V(0)_* THH(BP)
	\Longrightarrow V(0)_* F(S^{2p^2+1}_+, THH(BP))^{\bT}
$$
for $V(0) \wedge THH(BP)$, and its analogue for $V(1) \wedge THH(BP)$.
The circle trace image of $\beta^\circ_1 \in \pi_{2p^2-2p-1} V(0)$ in
$$
V(0)_* F(S^{2p^2+1}_+, THH(BP))^{\bT}
$$
factors as a product $[\![t^p]\!] \cdot [\![\lambda_2]\!]$, with
$[\![t^p]\!] \in \{t^p\}$ and $[\![\lambda_2]\!] \in \{\lambda_2\}$
detected by $t^p$ and~$\lambda_2$, respectively.  Moreover, the image
of $\gamma^\circ_1 \in \pi_{2p^3-2p^2-1} V(1)$ in
$$
V(1)_* F(S^{2p^2+1}_+, THH(BP))^{\bT}
$$
is the unique class detected by $t^{p^2} \lambda_3$.
\end{proposition}

\begin{proof}
Our first goal will be to show that $t^p$ times the indeterminacy in
$\{\lambda_2\}$ and $\lambda_2$ times the indeterminacy in $\{t^p\}$,
in combination, span the indeterminacy in $\{t^p \lambda_2\}$ in the
$p^2$-th order spectral sequence for $V(0) \wedge THH(BP)$.  To do this,
we compare the $m$-th order approximate $\bT$-homotopy fixed point
spectral sequences for the three $\bT$-spectra
$$
V(1) \wedge THH(BP)
\,,\quad
V(0) \wedge THH(BP)
\quad\text{and}\quad
THH(BP) \,,
$$
via the morphisms induced by $i_0 \: S \to V(0)$, $i_1 \:
V(0) \to V(1)$ and $j_1 \: V(1) \to \Sigma^{2p-1} V(0)$.

We begin with the $V(1)$-homotopy spectral sequence, which is easiest
to understand.  The $m$-th order spectral sequence for $V(1) \wedge
THH(BP)$ has $E^2$-term
$$
P_{m+1}(t)
	\otimes P(v_2, \dots) \otimes E(\lambda_1, \lambda_2, \dots) \,,
$$
where the omitted generators have vertical degree $* \ge 2p^3-2$.
Here $v_2$, $\lambda_1$ and $\lambda_2$ are infinite cycles, since
multiplication by $v_2$ is realized by a self-map of $V(1)$ and
since $\lambda_1$ and $\lambda_2$ detect the circle trace images of
$\lambda_1^K$ and $\lambda_2^K$, respectively.  For $m=p$ it follows that
this spectral sequence collapses at the $E^2$-term, in vertical degrees $*
< 2p^3-2$.

For $m > p$ there are nonzero $d^{2p}$-differentials
generated by
$$
d^{2p}(t) \doteq t^{p+1} \lambda_1 \,,
$$
where $x \doteq y$ means that $x$ is a unit (in $\bF_p$) times~$y$.  This
differential is present already in the $\bT$-homotopy fixed point spectral
sequence for $THH(BP)$, and lifts that of~\cite{BM94}*{Thm.~5.8(i)} for
$THH(\bZ_{(p)})$ over the morphism of spectral sequences induced by $BP
\to H\bZ_{(p)}$.  It follows that the $m$-th order $E^{2p+1}$-term equals
$$
\bF_p\{t^i \mid 0 \le i \le m, p \mid i\} \otimes P(v_2)
	\otimes E(\lambda_1, \lambda_2)
$$
in vertical degrees $* < 2p^3-2$, plus some extra classes in even
filtrations $-2m \le * < -2m + 2p$ and $-2p < * \le 0$ that survive
due to being close to the truncation limits.  Moreover, for $m < p^2+p$
the spectral sequence must collapse at this stage, for these vertical
degrees, since there is no room for a differential on $t^p$.

For later use, note that when $m = 3p-2$ no classes survive in total
degree $* = 2p^2 - 2p - 2i$ for $2 \le i < p$, since the classes
$t^{i+p-1} v_2$ support differentials and the classes $t^{i+2p-1}
\lambda_1 \lambda_2$ are hit by differentials.
Hence $V(1)_* F(S^{2m+1}_+, THH(BP))^{\bT}$ is zero in these degrees.
Moreover, for $i=1$ only the classes $t^{2p} \lambda_1 \lambda_2$
and $t^p v_2$ survive in total degree~$* = 2p^2 - 2p - 2$, and here
$i_1 j_1(t^p v_2)$ is detected by $t^{2p} \lambda_2 \ne 0$, so only
$t^{2p} \lambda_1 \lambda_2$ can be (and is) in the image of $i_1$,
since $j_1 i_1 = 0$.  Hence the image of $i_1$ is isomorphic to
$\bZ/p$ in this degree.

We now turn to the $V(0)$-homotopy spectral sequence.  The $p^2$-th order
approximate $\bT$-homotopy fixed point spectral sequence for $V(0)
\wedge THH(BP)$ has $E^2$-term
$$
P_{p^2+1}(t)
	\otimes P(v_1, v_2, \dots) \otimes E(\lambda_1, \lambda_2, \dots) \,,
$$
where the omitted generators have vertical degree $* \ge 2p^3-2$.
Here $t$, $v_1$, $\lambda_1$ and~$\lambda_2$ are $d^2$-cycles, while
$d^2(v_2) = - t v_1^p \lambda_1$ by Proposition~\ref{prop:piTHHBP}.
Hence the $E^3$-term equals
$$
P_{p^2+1}(t) \otimes \bigl( P(v_1) \{1\} \oplus
    P_p(v_1) \{\lambda_1, v_2 \lambda_1, \dots, v_2^{p-1} \lambda_1\} \bigr)
	\otimes E(\lambda_2)
$$
in vertical degrees~$* < 2p^3-2p$, except that there are some
additional classes in filtration degrees $0$ and $-2p^2$.  See
Figure~\ref{fig:approximate}, which is drawn for $p=3$, hence is not
quite to scale for the primes $p\ge7$ under consideration.
As above, we know that the classes $v_1$, $\lambda_1$ and $\lambda_2$
are infinite cycles.  The next nonzero differentials are
\begin{align*}
d^{2p}(t) &\doteq t^{p+1} \lambda_1 \\
d^{2p}(v_2 \lambda_1) &\doteq t^p v_1 \lambda_1 \lambda_2 \,.
\end{align*}
The $d^{2p}$-differential on $t$ for $V(0) \wedge THH(BP)$ follows,
as above, from the one for $THH(BP)$.
The earlier differential $d^{2p-2}(v_2 \lambda_1) \in \bF_p\{t^{p-1}
v_1^{p+3}\}$ must vanish by $t v_1^p$-linearity, since $t v_1^p
\cdot v_2 \lambda_1 = 0$ and $t v_1^p \cdot t^{p-1} v_1^{p+3} \ne 0$.
If $d^{2p}(v_2 \lambda_1)$ were zero, then $v_2 \lambda_1$ would detect
a class in $V(0)_* F(S^{2p+1}_+, THH(BP))^{\bT}$ that maps under $i_1
\: V(0) \to V(1)$ to the class in $V(1)_* F(S^{2p+1}_+, THH(BP))^{\bT}$
detected by $v_2 \lambda_1$.  However, the latter class maps under $i_1
j_1 \: V(1) \to \Sigma^{2p-1} V(1)$ to the nonzero class $i_1 j_1(v_2
\lambda_1) = i_1(\beta'_1) \lambda_1 = i_1(\beta^\circ_1) \lambda_1$
detected by $t^p \lambda_2 \cdot \lambda_1 = - t^p \cdot \lambda_1
\lambda_2$, as follows from Lemma~\ref{lem:j1v1beta-j2v2gamma} and
Proposition~\ref{prop:beta-tplambda2}.  This contradicts $j_1 i_1 = 0$,
and proves that $d^{2p}(v_2 \lambda_1)$ is nonzero in $\bF_p\{t^p v_1
\lambda_1 \lambda_2\}$.

\begin{figure}
\includegraphics[scale=0.7]{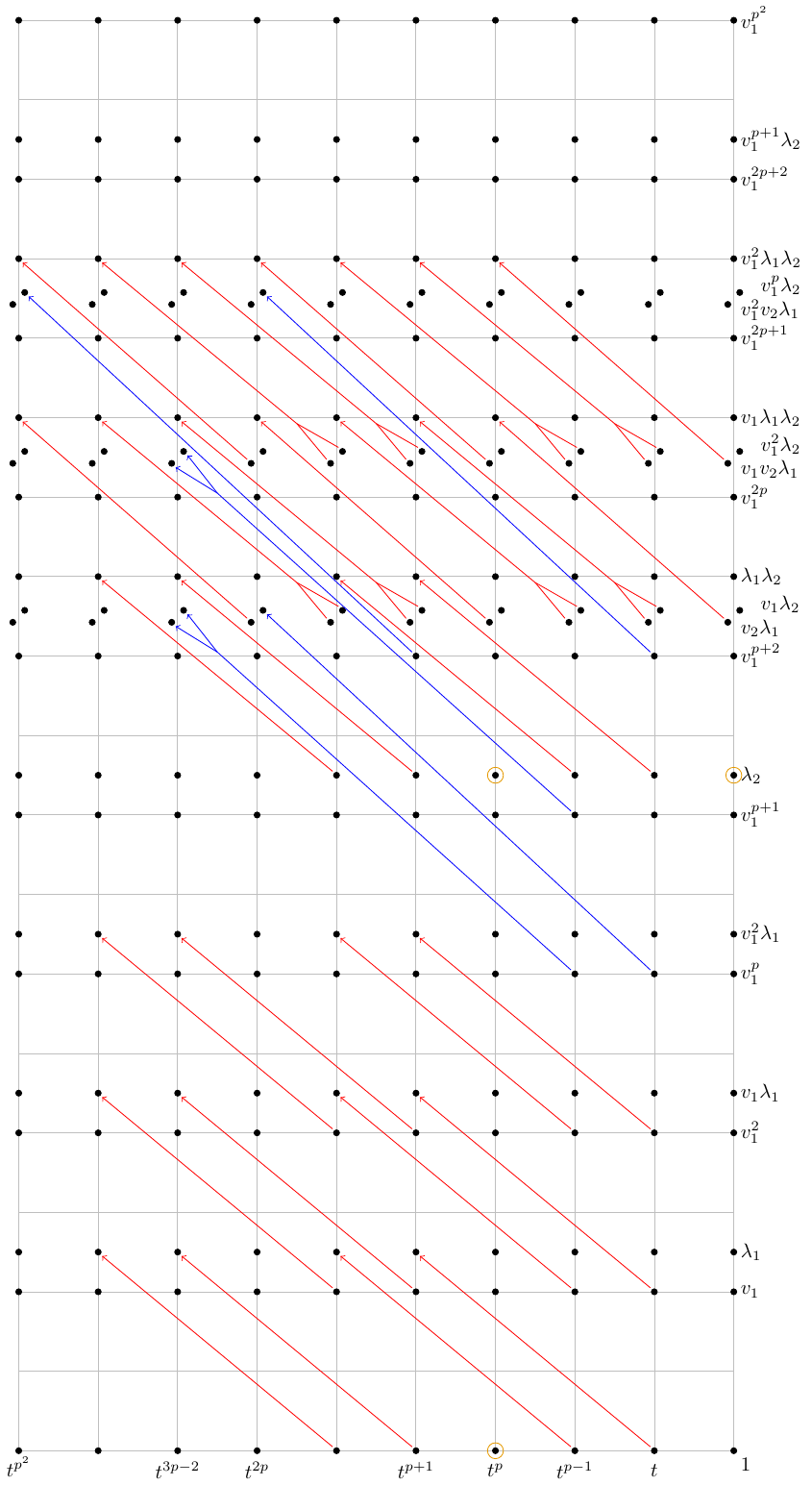}
\caption{$E^3 \Longrightarrow V(0)_* F(S^{2p^2+1}, THH(BP))^{\bT}$ in
vertical degrees $* < 4p^2+2p-5$, with all $d^{2p}$-differentials (red)
and selected $d^{4p-2}$-differentials (blue). \label{fig:approximate}}
\end{figure}

It follows that the $E^{2p+1}$-term equals
\begin{gather*}
P_{p+1}(t^p) \otimes \bigl(
    P(v_1)\{1, \lambda_2\} \oplus P_p(v_1)\{\lambda_1\}
	\oplus \bF_p\{\lambda_1 \lambda_2, v_1^{p-1} v_2 \lambda_1\} \bigr) \\
{} \oplus \bF_p\{t^i \mid 0 < i < p^2, p \nmid i\} \otimes \bigl(
    P(v_1)\{v_1^p, v_1 \lambda_2 + c v_2 \lambda_1\}
	\oplus \bF_p\{v_1^{p-1} v_2 \lambda_1\}
	\bigr)
\end{gather*}
in vertical degrees~$* < 4p^2+2p-5$, plus some extra classes in even
filtrations $-2p^2 \le * < -2p^2 + 2p$ and $-2p < * \le 0$.
In the expression $v_1 \lambda_2 + c v_2 \lambda_1$ the coefficient~$c$
(which will vary with the $t$-exponent $i$) is some unit in~$\bF_p$.

The next differentials include
\begin{align*}
d^{4p-2}(t v_1^p) &\doteq t^{2p} v_1 \lambda_2 \\
d^{4p-2}(t^i v_1^p) &\doteq t^{i+2p-1} (v_1 \lambda_2 + c v_2 \lambda_1)
\end{align*}
for $2 \le i < p$.  To see that these are nonzero, we compare the
$m$-th order spectral sequences for $V(0) \wedge THH(BP)$ and $V(1)
\wedge THH(BP)$, in the particular case $m = 3p-2$.  If $d^{4p-2}(t^i
v_1^p)$ were zero in the former, then $t^i v_1^p$ would survive to detect
a class in degree~$2p^2-2p-2i$ of $V(0)_* F(S^{2m+1}_+, THH(BP))^{\bT}$
that cannot be a $v_1$-multiple, for filtration reasons, and which must
therefore have nonzero image in $V(1)_* F(S^{2m+1}_+, THH(BP))^{\bT}$.
However, for $2 \le i < p$ we checked above that this graded abelian group
is zero in these degrees.  This contradiction shows that $d^{4p-2}(t^i
v_1^p)$ is nonzero in $\bF_p\{ t^{i+2p-1} (v_1 \lambda_2 + c v_2
\lambda_1) \}$, as claimed.

Furthermore, for $i=1$ it is not possible that both $t^{2p} \lambda_1
\lambda_2$ and $t v_1^p$ survive to $E^\infty$, since then the image
of $i_1$ in degree $2p^2-2p-2$ would have order~$p^2$, rather than the
order~$p$ that we established above.  Hence $d^{4p-2}(t v_1^p)$ must
be nonzero in $\bF_p\{t^{2p} v_1 \lambda_2, t^{2p} v_2 \lambda_1\}$.
Extending to the case $m = 3p$ shows that $d^{4p-2}(t v_1^p)$ must be
nonzero in $\bF_p\{t^{2p} v_1 \lambda_2\}$, as claimed.

We can now conclude that $t^p$ is an infinite cycle in the spectral
sequence converging to
$$
V(0)_* F(S^{2p^2+1}_+, THH(BP))^{\bT} \,,
$$
since there are no possible targets for later differentials, and
the indeterminacy in $\{t^p\}$ is generated by (classes detected by)
$$
t^{p^2-p+1} v_1^{p-1}
\quad\text{and}\quad
t^{p^2} v_1^p \,.
$$
The class $t^p \lambda_2$ is also an infinite cycle, detecting the
circle trace image of $\beta^\circ_1$ by Proposition~\ref{prop:beta-tplambda2},
and has indeterminacy generated by (a subset of)
$$
t^{2p-1} (v_1 \lambda_2 + c v_2 \lambda_1)
\,,\quad
t^{p^2-p+1} v_1^{p-1} \lambda_2
\quad\text{and}\quad
t^{p^2} v_1^{p-1} v_2 \lambda_1 \,.
$$
Likewise, $\lambda_2$ is an infinite cycle, detecting the circle
trace image of $\lambda^K_2$ plus some multiple of~$v_1^p \lambda_1^K$
according to Lemma~\ref{lem:trlambda2K}, with indeterminacy generated by
(a subset of)
\begin{gather*}
t^{p-1} (v_1 \lambda_2 + c v_2 \lambda_1)
\,,\quad
t^{2p-2} (v_1^2 \lambda_2 + c v_1 v_2 \lambda_1)
\,,\quad \\
t^{p^2-p} v_1^{p-1} v_2 \lambda_1
\quad\text{and}\quad
t^{p^2-1} v_1^{p+1} \lambda_2 \,.
\end{gather*}
Here $t^{p-1} (v_1 \lambda_2 + c v_2 \lambda_1)$ might support a
nonzero $d^r$-differential and not be an infinite cycle.  However,
there are no possible targets in filtrations $-2p^2 \le * < -2p^2+2p$
of such a $d^r$-differential, since $d^{4p-2}(t^{p^2-1} v_1^{2p+2}) =
t^{p^2+2p-2} v_1^{p+3} \lambda_2 \ne 0$ in the full $\bT$-homotopy fixed
point spectral sequence.  Hence, in this case $t^{2p-1} (v_1 \lambda_2 +
c v_2 \lambda_1)$ will also support a nonzero differential, of the same
length, and also not be an infinite cycle.  Similarly, if $t^{p^2-p}
v_1^{p-1} v_2 \lambda_1$ is hit by a $d^r$-differential, then $t^{p^2}
v_1^{p-1} v_2 \lambda_1$ will be hit by a differential of the same length.

It follows that $t^p$ times the indeterminacy in $\{\lambda_2\}$, together
with the class $t^{p^2-p+1} v_1^{p-1} \lambda_2$, span the indeterminacy
in $\{t^p \lambda_2\}$.  That extra class lies in the indeterminacy of
$\{t^p\}$ times $\lambda_2$.  Hence we have achieved our first goal,
as formulated at the outset of the proof.

Now choose classes $x$ and~$y$ in $V(0)_* F(S^{2p^2+1}_+, THH(BP))^{\bT}$,
detected by $t^p$ and~$\lambda_2$, respectively.  Then the difference
between the circle trace image of $\beta^\circ_1$ and the product $x
y$ lies in the indeterminacy of $\{t^p \lambda_2\}$.  By modifying
the choices of $x$ and~$y$, within the indeterminacies of $\{t^p\}$
and~$\{\lambda_2\}$, respectively, we can reduce the filtration of
this difference until it becomes zero.  Let $[\![t^p]\!] = x$ and
$[\![\lambda_2]\!]  = y$ be the final values of $x \in \{t^p\}$ and $y
\in \{\lambda_2\}$, so that the circle trace image of $\beta^\circ_1$
equals the product $[\![t^p]\!] \cdot [\![\lambda_2]\!]$.

Let $\lambda^\circ_3 = P^{p^2}(\lambda_2)$ in $V(1)_* THH(BP)$.  We apply
the Cartan formula from~Corollary~\ref{cor:CartanV0V1} in the case
of the $E_3$ ring spectrum retract $F(S^{2p^2+1}_+, THH(BP))^{\bT}$ of
$F(S^{2p^2+1}_+, THH(MU_{(p)}))^{\bT}$, where the latter is an $E_\infty$
ring spectrum.  It asserts that the circle trace image of $\gamma^\circ_1
= P^{p^2-p}(\beta^\circ_1)$ is
$$
P^{p^2-p}([\![t^p]\!] \cdot [\![\lambda_2]\!])
	= [\![t^p]\!]^p \cdot P^{p^2}([\![\lambda_2]\!]) \,.
$$
Here $P^{p^2}([\![\lambda_2]\!]) \in \{\lambda^\circ_3\}$ is a
class detected by $\lambda^\circ_3$, by naturality of $P^{p^2}$
with respect to the edge homomorphism induced by $F(S^{2p^2+1}_+,
THH(BP))^{\bT} \to THH(BP)$.  It follows that $[\![t^p]\!]^p \cdot
P^{p^2}([\![\lambda_2]\!])$ is detected by $t^{p^2} \lambda^\circ_3$, with
zero indeterminacy since this class lives in the lowest filtration degree.

To complete the proof, note that $t^{p^2} \lambda^\circ_3 = t^{p^2}
\lambda_3$ at the $V(1)$-homotopy $E^3$-term, since these classes
differ by a multiple of $d^2(t^{p^2-1} v_3) = - t^{p^2} v_2^p \lambda_1$
by Proposition~\ref{prop:piTHHBP} and Lemma~\ref{lem:trlambda3K}.
\end{proof}

\begin{remark} \label{rem:v1betacirc1}
In the course of the previous proof, we have seen that the circle trace
image of $\beta^\circ_1 \in V(0)_*$ is detected by $t^p \lambda_2$,
and that $t^p v_1 \lambda_2$ is not a boundary in the (approximate)
$\bT$-homotopy fixed point spectral sequence, which implies that $v_1
\cdot \beta^\circ_1 \ne 0$.  This confirms a claim made in the proof
of Lemma~\ref{lem:j1v1beta-j2v2gamma}.
\end{remark}

\begin{proposition} \label{prop:v3-tmu}
Consider the first order (approximate $\bT$-homotopy fixed point) spectral
sequence
$$
E^2_{*,*} = \bZ[t]/(t^2) \otimes V(2)_* THH(BP\<2\>)
	\Longrightarrow V(2)_* F(S^3_+, THH(BP\<2\>))^{\bT}
$$
for $V(2) \wedge THH(BP\<2\>)$.
The circle trace image of $v_3 \in \pi_{2p^3-2} V(2)$ in
$$
V(2)_* F(S^3_+, THH(BP\<2\>))^{\bT}
$$
is the unique class detected by $t \mu$.
\end{proposition}

\begin{proof}
The line of argument is the same as for the case of $v_2 \in \pi_{2p^2-2}
V(1)$ in~\cite{AR02}*{Prop.~4.8}.  For brevity, let $Y = F(S^3_+,
THH(BP\<2\>))^{\bT}$.  We have a map of mod~$p$ Adams spectral sequences
$$
E_2(V(2)) = \Ext_{\cA_*}(\bF_p, H_* V(2))
	\longto \Ext_{\cA_*}(\bF_p, H_*(V(2) \wedge Y))
	= E_2(V(2) \wedge Y) \,,
$$
where $v_3$ is detected in the source in bidegree~$(s,t) = (1, 2p^3-1)$
by the class of the cobar cocycle
$$
x = [\tau_3] 1 + [\xi_3] \tau_0 + [\xi_2^p] \tau_1 + [\xi_1^{p^2}] \tau_2
$$
in $E_1^{1,*}(V(2)) = \bar \cA_* \otimes H_* V(2)$.
(As usual, $\bar \cA_*$ denotes the cokernel of the unit $\bF_p \to \cA_*$.)
We claim that this cocycle does not become a coboundary when mapped to
$E_1^{1,*}(V(2) \wedge Y) = \bar \cA_* \otimes H_*(V(2) \wedge Y)$.
This implies that the image of $v_3$ is nonzero in $V(2)_*(Y)$, and
in view of Proposition~\ref{prop:VnTHHBPn} the only possible detecting
class in its total degree is~$t \mu$.

To prove the claim we use the first order spectral sequence for $H
\wedge V(2) \wedge THH(BP\<2\>)$, which reduces to a long exact sequence,
leading to an extension
$$
0 \to \cok(\sigma) \longto H_*(V(2) \wedge Y) \longto \ker(\sigma) \to 0
$$
of $\cA_*$-comodules.  Here
$$
\sigma \: H_*(V(2) \wedge THH(BP\<2\>))
	\longto H_{*+1}(V(2) \wedge THH(BP\<2\>))
$$
acts on $H_*(V(2) \wedge THH(BP\<2\>)) \cong \cA_* \otimes
E(\sigma\bar\xi_1, \sigma\bar\xi_2, \sigma\bar\xi_3) \otimes
P(\sigma\bar\tau_3)$, as per Proposition~\ref{prop:HTHHBPBPn}.
The cocycle~$x$ is a cobar coboundary only if there is a class $y \in
E_1^{0,*}(V(2) \wedge Y) = H_*(V(2) \wedge Y)$ with $\cA_*$-comodule
coaction $\nu(y)$ containing the term $\tau_3 \otimes 1$.

There is no such class $y \in \cok(\sigma)$, since this $\cA_*$-subcomodule
does not contain the algebra unit~$1$.  Moreover, since $\sigma(\bar\tau_3) =
\sigma\bar\tau_3 \ne 0$, the class $\bar\tau_3$ is not in $\ker(\sigma)$.
Hence $\ker(\sigma)$ in total degree~$2p^3-1$ is generated by polynomials
in $\bar\tau_0$, $\bar\tau_1$, $\bar\tau_2$, $\bar\xi_1$, $\bar\xi_2$,
$\bar\xi_3$, $\sigma\bar\xi_1$, $\sigma\bar\xi_2$ and $\sigma\bar\xi_3$,
none of which have $\cA_*$-coaction that involves $\tau_3$.  This proves
that no such class~$y$ exists, and $x$ is not a coboundary.
\end{proof}

\section{The $C_p$-Tate spectral sequence}
\label{sec:CpTatespseq}

We now establish an effective version of the $C_p$-equivariant Segal
conjecture (or homotopy limit property) for $V(2) \wedge THH(BP\<2\>)$, by
direct computation.  The corresponding results for the groups $C_{p^n}$
and $\bT$ then follow from a theorem of Tsalidis.  The analogous
results for $BP\<0\> = H\bZ_{(p)}$ and~$BP\<1\> = \ell$ were proved in
\cite{BM94}*{Thm.~5.8(i)} and~\cite{AR02}*{Thm.~5.5}, respectively.

\begin{theorem} \label{thm:CpTate}
The $C_p$-Tate spectral sequence
\begin{align*}
\hat E^2(C_p) &= \hat H^{-*}(C_p; V(2)_* THH(BP\<2\>)) \\
	&\Longrightarrow V(2)_* THH(BP\<2\>)^{tC_p}
\end{align*}
has $E^2$-term
$$
\hat E^2(C_p) = E(u_1) \otimes P(t^{\pm1})
	\otimes P(t\mu) \otimes E(\lambda_1, \lambda_2, \lambda_3) \,.
$$
There are differentials
\begin{align*}
d^{2p}(t^{1-p}) &\doteq t \lambda_1 \\
d^{2p^2}(t^{p-p^2}) &\doteq t^p \lambda_2 \\
d^{2p^3}(t^{p^2-p^3}) &\doteq t^{p^2} \lambda_3 \\
d^{2p^3+1}(u_1 t^{-p^3}) &\doteq t\mu \,,
\end{align*}
and the classes $\lambda_1$, $\lambda_2$, $\lambda_3$ and~$t^{\pm
p^3}$ are permanent cycles.  The $E^\infty$-term
$$
\hat E^\infty = P(t^{\pm p^3}) \otimes E(\lambda_1, \lambda_2, \lambda_3)
$$
is the associated graded of
$$
V(2)_* THH(BP\<2\>)^{tC_p} \cong E(\lambda_1, \lambda_2, \lambda_3)
	\otimes P(\mu^{\pm1}) \,.
$$
The comparison map $\hat\Gamma_1 \: THH(BP\<2\>) \to THH(BP\<2\>)^{tC_p}$
induces the localization homomorphism
$$
V(2)_* \hat\Gamma_1 \: E(\lambda_1, \lambda_2, \lambda_3) \otimes P(\mu)
	\longto E(\lambda_1, \lambda_2, \lambda_3) \otimes P(\mu^{\pm1}) \,,
$$
which is $(2p^2 + 2p - 3)$-coconnected.
\end{theorem}

\begin{proof}
The circle trace map $K(B) \to THH(B)^{h\bT}$ lifts the trace map,
so by Proposition~\ref{prop:trace} the classes $\lambda^K_i$ for $i
\in \{1,2,3\}$ map to classes in $V(2)_* THH(B)^{h\bT}$ detected by the
$\lambda_i$.  Similarly, by Proposition~\ref{prop:unit} the class $v_3$
in $\pi_* V(2)$ maps to a class detected by~$t\mu$.  Hence these detecting
classes are infinite cycles in all of the $C$-homotopy fixed point and
$C$-Tate spectral sequences.  This means that in order to determine
the $d^r$-differentials in one of these spectral sequences, it suffices
to determine $d^r(x)$ for $x$ ranging through a $P(t\mu) \otimes
E(\lambda_1, \lambda_2, \lambda_3)$-module basis for the $E^r$-term.

The unit map $S \to THH(B)$ factors through~$B$, and $V(2)_* BP\<2\> =
\bF_p$, so the images of $\alpha_1$, $\beta^\circ_1$, $\gamma^\circ_1$
and~$v_3$ in $\pi_* V(2)$ map to zero in $V(2)_* THH(BP\<2\>)$ and~$V(2)_*
THH(BP\<2\>)^{tC_p}$.  Hence the four classes $t \lambda_1$, $t^p
\lambda_2$, $t^{p^2} \lambda_3$ and~$t\mu$ must all be boundaries in
the $C_p$-Tate spectral sequence.

The first possible (nonzero) $d^r$-differentials on $u_1$ and $t^{\pm1}$
in $\hat E^2(C_p)$ have $r = 2p$.  We know that $t \lambda_1$ is a
boundary, so
$$
d^{2p}(t^{1-p}) \doteq t \lambda_1 \,.
$$
Also $d^{2p}(u_1) \in \bF_p\{u_1 t^p \lambda_1\}$, so $d^{2p}(u_1 t^{m_1})
= 0$ for some integer~$m_1$ defined mod~$p$.  Hence
$$
\hat E^{2p+1}(C_p) = E(u_1 t^{m_1}) \otimes P(t^{\pm p}) \otimes P(t\mu)
\otimes E(\lambda_1, \lambda_2, \lambda_3) \,.
$$
The next possible $d^r$-differentials on $u_1 t^{m_1}$ and
$t^{\pm p}$ have $r = 2p^2$.
We know that $t^p \lambda_2$ is a boundary, so
$$
d^{2p^2}(t^{p-p^2}) \doteq t^p \lambda_2 \,.
$$
Also $d^{2p^2}(u_1 t^{m_1}) \in \bF_p\{u_1 t^{m_1+p^2} \lambda_2\}$,
so $d^{2p^2}(u_1 t^{m_2}) = 0$ for some integer~$m_2$ defined mod~$p^2$,
with $m_2 \equiv m_1 \mod p$.  Then
$$
\hat E^{2p^2+1}(C_p) = E(u_1 t^{m_2}) \otimes P(t^{\pm p^2})
	\otimes P(t\mu) \otimes E(\lambda_1, \lambda_2, \lambda_3) \,.
$$
If $m_2 \equiv -p \mod p^2$ then the first possible differential on
$u_1 t^{m_2}$ is $d^r(u_1 t^{m_2}) \in \bF_p\{t^{m_2+p^2+p} \lambda_1
\lambda_2\}$ with $r = 2p^2+2p-1$.  Otherwise, the first possible
differential on $u_1 t^{m_2}$ has $r = 2p^3$.

By naturality with respect to the group cohomology transfer
(Verschiebung), with $V(t^i) = 0$ and $V(u_1 t^i) = u_2 t^i$, the first
possible $d^r$-differential on $t^{\pm p^2}$ cannot take a value of the
form~$u_1 x$, hence has $r = 2p^3$, cf.~\cite{AR02}*{Lem.~5.2}.

We know that $t^{p^2} \lambda_3$ is a boundary, and the only possible
sources in $\hat E^2(C_p)$ of a $d^r$-differential with this target are
$t^{-p^3+2p^2+p-1} \lambda_1 \lambda_2$ with $r = 2p^3-2p^2-2p+2$,
$u_1 t^{-p^3+2p^2-1} \lambda_2$ with $r = 2p^3-2p^2+1$,
$u_1 t^{-p^3+p^2+p-1} \lambda_1$ with $r = 2p^3-2p+1$ and
$t^{-p^3+p^2}$ with $r = 2p^3$.
The first source is not present in $\hat E^{2p^2+1}(C_p)$, and the second
and third sources are present there only if $m_2 \equiv -1 \mod p^2$
or $m_2 \equiv p-1 \mod p^2$, respectively.  In both of these cases $m_2
\not\equiv -p \mod p^2$, so $u_1 t^{m_2}$ survives to the $E^{2p^3}$-term.
In the second case
$$
d^{2p^3-2p^2+1}(u_1 t^{-p^3+2p^2-1} \lambda_2)
    = d^{2p^3-2p^2+1}(u_1 t^{-p^3+2p^2-1}) \lambda_2 = 0 \,,
$$
while in the third case
$$
d^{2p^3-2p+1}(u_1 t^{-p^3+p^2+p-1} \lambda_1)
    = d^{2p^3-2p+1}(u_1 t^{-p^3+p^2+p-1}) \lambda_1 = 0 \,.
$$
Hence the fourth option,
$$
d^{2p^3}(t^{-p^3+p^2}) \doteq t^{p^2} \lambda_3 \,,
$$
is the only possibility.

We also know that $t\mu$ is a boundary, and the only possible sources
of a $d^r$-differential with this target are
$u_1 t^{-p^3+p^2+p-1} \lambda_1 \lambda_2$ with $r = 2p^3-2p^2-2p+3$,
$t^{-p^3+p^2} \lambda_2$ with $r = 2p^3-2p^2+2$,
$t^{-p^3+p} \lambda_1$ with $r = 2p^3-2p+2$ and
$u_1 t^{-p^3}$ with $r = 2p^3+1$.
The first source is only present in $\hat E^{2p^2+1}(C_p)$ if $m_2
\equiv p-1 \mod p^2$, in which case $u_1 t^{m_2}$ survives to the
$E^{2p^3}$-term, and
$$
d^{2p^3-2p^2-2p+3}(u_1 t^{-p^3+p^2+p-1} \lambda_1 \lambda_2)
    = d^{2p^3-2p^2-2p+3}(u_1 t^{-p^3+p^2+p-1}) \lambda_1 \lambda_2 = 0 \,.
$$
In the second case
$$
d^{2p^3-2p^2+2}(t^{-p^3+p^2} \lambda_2)
    = d^{2p^3-2p^2+2}((t^{-p^2})^{p-1}) \lambda_2 = 0 \,,
$$
since $t^{\pm p^2}$ survive to the $E^{2p^3}$-term.  The third source is
not present in $\hat E^{2p^2+1}(C_p)$.  This leaves the fourth option,
$$
d^{2p^3+1}(u_1 t^{-p^3}) \doteq t\mu \,,
$$
as the only possibility.  It follows that $d^{2p^3}(u_1 t^{-p^3}) = 0$.
In particular, $u_1 t^{-p^3}$ must be present in $\hat E^{2p^2+1}(C_p)$,
and we may take $m_1 = m_2 = 0$ in the formulas above.  Then
$$
\hat E^{2p^3+1}(C_p) = E(u_1 t^{-p^3}) \otimes P(t^{\pm p^3})
        \otimes P(t\mu) \otimes E(\lambda_1, \lambda_2, \lambda_3) \,.
$$
Here $d^{2p^3+1}(t^{-p^3})$ lies in a trivial group, so
$$
\hat E^{2p^3+2}(C_p) = P(t^{\pm p^3})
        \otimes E(\lambda_1, \lambda_2, \lambda_3) \,.
$$
This equals $\hat E^\infty(C_p)$, since there are no further targets
for differentials on $t^{-p^3}$.

We claim that $\hat\Gamma_1(\mu)$ in $V(2)_* THH(BP\<2\>)^{tC_p}$
is detected by a unit times~$t^{-p^3}$.  To see this, we can use
naturality with respect to the map $BP\<2\> \to BP\<1\>$, as in the
commutative diagram below.
$$
\xymatrix{
H_* THH(BP\<2\>) \ar[d]
	& V(2)_* THH(BP\<2\>) \ar[l]_-{h_2} \ar[r]^-{\hat\Gamma_1} \ar[d]
	& V(2)_* THH(BP\<2\>)^{tC_p} \ar[d] \\
H_* THH(BP\<1\>)
	& V(2)_* THH(BP\<1\>) \ar[l]_-{h_2} \ar[r]^-{\hat\Gamma_1}
	& V(2)_* THH(BP\<1\>)^{tC_p} \\
H_* THH(BP\<1\>) \ar@{=}[u]
	& V(1)_* THH(BP\<1\>) \ar[l]_-{h_1}
		\ar[r]^-{\hat\Gamma_1} \ar[u]^-{i_2}
	& V(1)_* THH(BP\<1\>)^{tC_p} \ar[u]^-{i_2}
}
$$
Recall from Proposition~\ref{prop:VnTHHBPn}
that $V(1)_* THH(BP\<1\>) = E(\lambda_1, \lambda_2) \otimes P(\mu_2)$,
where $h_1(\mu_2) = \sigma\bar\tau_2$ in $H_* THH(BP\<1\>)$, and that
$\hat\Gamma_1(\mu_2)$ in $V(1)_* THH(BP\<1\>)^{tC_p}$ is detected by a
unit times $t^{-p^2}$ by the proof of~\cite{AR02}*{Thm.~5.5}.  It follows
that $\mu$ maps to $i_2(\mu_2^p)$ in $V(2)_* THH(BP\<1\>)$, since
$h_2(\mu) = \sigma\bar\tau_3$ maps to $h_1(\mu_2^p) = (\sigma\bar\tau_2)^p
= \sigma\bar\tau_3$.  By naturality, $\hat\Gamma_1(\mu)$ maps to a
class detected by a unit times $(t^{-p^2})^p = t^{-p^3}$, which proves
the claim.

The highest-degree class in $E(\lambda_1, \lambda_2, \lambda_3)
\otimes P(\mu^{\pm1})$ that is not in the image from
$E(\lambda_1, \lambda_2, \lambda_3) \otimes P(\mu)$ is
$\lambda_1 \lambda_2 \lambda_3 \mu^{-1}$, in degree
$(2p-1) + (2p^2-1) + (2p^3-1) - (2p^3) = 2p^2 + 2p - 3$.
Hence $V(2)_* \hat\Gamma_1$ is injective in this degree, and
an isomorphism in all higher degrees.
\end{proof}


\begin{corollary}[\cite{Tsa98}*{Thm.~2.4}, \cite{BBLNR14}*{Thm.~2.8}]
\label{cor:coconn}
The comparison maps
\begin{align*}
\Gamma_n &\: V(2) \wedge THH(BP\<2\>)^{C_{p^n}}
	\longto V(2) \wedge THH(BP\<2\>)^{hC_{p^n}} \\
\hat\Gamma_n &\: V(2) \wedge THH(BP\<2\>)^{C_{p^{n-1}}}
	\longto V(2) \wedge THH(BP\<2\>)^{tC_{p^n}}
\end{align*}
for $n\ge1$, and their homotopy limits
\begin{align*}
\Gamma &\: V(2) \wedge TF(BP\<2\>)
	\longto V(2) \wedge THH(BP\<2\>)^{h\bT} \\
\hat\Gamma &\: V(2) \wedge TF(BP\<2\>)
	\longto V(2) \wedge THH(BP\<2\>)^{t\bT} \,,
\end{align*}
are all $(2p^2 + 2p - 3)$-coconnected.
\end{corollary}

\section{The $C_{p^2}$-Tate spectral sequence}

Our next goal is to determine the differential structure of
the $C_{p^n}$-Tate spectral sequence converging to $V(2)_*
THH(BP\<2\>)^{tC_{p^n}}$, for each $n\ge2$.  There are some minor
differences between the cases $n=2$ and $n\ge3$, so we spell out the
$C_{p^2}$-case in this section, including some motivation, and leave
the notationally more elaborate cases $n\ge3$ for the next section.

We first determine the structure of the $C_p$-homotopy fixed point
spectral sequence from that of the $C_p$-Tate spectral sequence, using
the homotopy restriction (= canonical) morphism
$$
R^h \: E^r(C_p) \longto \hat E^r(C_p) \,.
$$
It is algebraically simpler to work with the localized spectral sequence
$\mu^{-1} E^r(C_p)$, keeping in mind that
$$
E^r(C_p) \longto \mu^{-1} E^r(C_p)
$$
is $(2p^2+2p-3)$-coconnected.  The $\mu$-localized $C_p$-homotopy fixed
point spectral sequence for $V(2) \wedge THH(BP\<2\>)$ is isomorphic
to the $C_p$-homotopy fixed point spectral sequence for $V(2) \wedge
THH(BP\<2\>)^{tC_p}$, in view of Theorem~\ref{thm:CpTate}.

\begin{proposition} \label{prop:locCpHFP}
The $\mu$-localized $C_p$-homotopy fixed point spectral sequence
\begin{align*}
\mu^{-1} E^2(C_p) &= H^{-*}(C_p; \mu^{-1} V(2)_* THH(BP\<2\>)) \\
	&\Longrightarrow \mu^{-1} V(2)_* THH(BP\<2\>)^{hC_p}
\end{align*}
has $E^2$-term
$$
\mu^{-1} E^2(C_p) = E(u_1) \otimes P(t\mu) \otimes
	E(\lambda_1, \lambda_2, \lambda_3) \otimes P(\mu^{\pm1}) \,.
$$
There are differentials
\begin{align*}
d^{2p}(\mu) &\doteq (t\mu)^p \lambda_1 \mu^{1-p} \\
d^{2p^2}(\mu^p) &\doteq (t\mu)^{p^2} \lambda_2 \mu^{p-p^2} \\
d^{2p^3}(\mu^{p^2}) &\doteq (t\mu)^{p^3} \lambda_3 \mu^{p^2-p^3} \\
d^{2p^3+1}(u_1 \mu^{p^3}) &\doteq (t\mu)^{p^3+1} \,,
\end{align*}
and the classes $t\mu$, $\lambda_1$, $\lambda_2$, $\lambda_3$
and~$\mu^{\pm p^3}$ are permanent cycles.
\end{proposition}

\begin{proof}
The composite relations
\begin{align*}
d^{2p}(\mu) \cdot \mu^p
	&= d^{2p}(t\mu \cdot t^{-1}) \cdot \mu^p
	\doteq t\mu \cdot t^{p-1} \lambda_1 \cdot \mu^p
	= (t\mu)^p \lambda_1 \mu \\
d^{2p^2}(\mu^p) \cdot \mu^{p^2}
	&= d^{2p^2}((t\mu)^p \cdot t^{-p}) \cdot \mu^{p^2}
	\doteq (t\mu)^p \cdot t^{p^2-p} \lambda_2 \cdot \mu^{p^2}
	= (t\mu)^{p^2} \lambda_2 \mu^p \\
d^{2p^3}(\mu^{p^2}) \cdot \mu^{p^3}
	&= d^{2p^3}((t\mu)^{p^2} \cdot t^{-p^2}) \cdot \mu^{p^3}
	\doteq (t\mu)^{p^2} \cdot t^{p^3-p^2} \lambda_3 \cdot \mu^{p^3}
	= (t\mu)^{p^3} \lambda_3 \mu^{p^2} \\
d^{2p^3+1}(u_1 \mu^{p^3})
	&= d^{2p^3+1}((t\mu)^{p^3} \cdot u_1 t^{-p^3})
	\doteq (t\mu)^{p^3} \cdot t\mu = (t\mu)^{p^3+1}
\end{align*}
lift to the $C_p$-homotopy fixed point spectral sequence, and can be
rewritten as claimed after inverting~$\mu$.
\end{proof}

The first differential leaves
\begin{align*}
\mu^{-1} E^{2p+1}(C_p) &= E(u_1) \otimes P(t\mu) \otimes
	E(\lambda_1, \lambda_2, \lambda_3) \otimes P(\mu^{\pm p}) \\
&\quad \oplus E(u_1) \otimes P_p(t\mu) \otimes E(\lambda_2, \lambda_3)
	\otimes \bF_p\{ \lambda_1 \mu^j \mid v_p(j) = 0 \} \,.
\end{align*}
The second leaves
\begin{align*}
\mu^{-1} E^{2p^2+1}(C_p) &= E(u_1) \otimes P(t\mu) \otimes
	E(\lambda_1, \lambda_2, \lambda_3) \otimes P(\mu^{\pm p^2}) \\
&\quad \oplus E(u_1) \otimes P_p(t\mu) \otimes E(\lambda_2, \lambda_3)
	\otimes \bF_p\{ \lambda_1 \mu^j \mid v_p(j) = 0 \} \\
&\quad \oplus E(u_1) \otimes P_{p^2}(t\mu) \otimes E(\lambda_1, \lambda_3)
	\otimes \bF_p\{ \lambda_2 \mu^j \mid v_p(j) = 1 \} \,.
\end{align*}
The third leaves
\begin{align*}
\mu^{-1} E^{2p^3+1}(C_p) &= E(u_1) \otimes P(t\mu) \otimes
	E(\lambda_1, \lambda_2, \lambda_3) \otimes P(\mu^{\pm p^3}) \\
&\quad \oplus E(u_1) \otimes P_p(t\mu) \otimes E(\lambda_2, \lambda_3)
	\otimes \bF_p\{ \lambda_1 \mu^j \mid v_p(j) = 0 \} \\
&\quad \oplus E(u_1) \otimes P_{p^2}(t\mu) \otimes E(\lambda_1, \lambda_3)
	\otimes \bF_p\{ \lambda_2 \mu^j \mid v_p(j) = 1 \} \\
&\quad \oplus E(u_1) \otimes P_{p^3}(t\mu) \otimes E(\lambda_1, \lambda_2)
	\otimes \bF_p\{ \lambda_3 \mu^j \mid v_p(j) = 2 \} \,.
\end{align*}
The final differential leaves
\begin{align*}
\mu^{-1} E^{2p^3+2}(C_p) &= P_{p^3+1}(t\mu) \otimes
	E(\lambda_1, \lambda_2, \lambda_3) \otimes P(\mu^{\pm p^3}) \\
&\quad \oplus E(u_1) \otimes P_p(t\mu) \otimes E(\lambda_2, \lambda_3)
	\otimes \bF_p\{ \lambda_1 \mu^j \mid v_p(j) = 0 \} \\
&\quad \oplus E(u_1) \otimes P_{p^2}(t\mu) \otimes E(\lambda_1, \lambda_3)
	\otimes \bF_p\{ \lambda_2 \mu^j \mid v_p(j) = 1 \} \\
&\quad \oplus E(u_1) \otimes P_{p^3}(t\mu) \otimes E(\lambda_1, \lambda_2)
	\otimes \bF_p\{ \lambda_3 \mu^j \mid v_p(j) = 2 \} \,,
\end{align*}
which equals $\mu^{-1} E^\infty(C_p)$.

Next we use the commutative diagram
$$
\xymatrix{
THH(BP\<2\>)^{hC_p} \ar[d]^-{F}
	& THH(BP\<2\>)^{C_p} \ar[d]^-{F}
		\ar[l]_-{\Gamma_1} \ar[r]^-{\hat\Gamma_2}
	& THH(BP\<2\>)^{tC_{p^2}} \ar[d]^-{F} \\
THH(BP\<2\>)
	& THH(BP\<2\>) \ar@{=}[l] \ar[r]^-{\hat\Gamma_1}
	& THH(BP\<2\>)^{tC_p}
}
$$
and what is known about $V(2)_* THH(BP\<2\>)^{hC_p}$ above
degree~$2p^2+2p-3$ to pin down the differential pattern
of the $C_{p^2}$-Tate spectral sequence leading to $V(2)_*
THH(BP\<2\>)^{tC_{p^2}}$.

\begin{theorem} \label{thm:Cp2Tate}
The $C_{p^2}$-Tate spectral sequence
\begin{align*}
\hat E^2(C_{p^2}) &= \hat H^{-*}(C_{p^2}; V(2)_* THH(BP\<2\>)) \\
        &\Longrightarrow V(2)_* THH(BP\<2\>)^{tC_{p^2}}
\end{align*}
has $E^2$-term
$$
\hat E^2(C_{p^2}) = E(u_2) \otimes P(t^{\pm1}) \otimes P(t\mu)
	\otimes E(\lambda_1, \lambda_2, \lambda_3) \,.
$$
There are differentials
\begin{align*}
d^{2p}(t^{1-p}) &\doteq t \lambda_1 \\
d^{2p^2}(t^{p-p^2}) &\doteq t^p \lambda_2 \\
d^{2p^3}(t^{p^2-p^3}) &\doteq t^{p^2} \lambda_3 \\
d^{2p^4+2p}(t^{p^3-p^4}) &\doteq t^{p^3} (t\mu)^p \lambda_1 \\
d^{2p^5+2p^2}(t^{p^4-p^5}) &\doteq t^{p^4} (t\mu)^{p^2} \lambda_2 \\
d^{2p^6+2p^3}(t^{p^5-p^6}) &\doteq t^{p^5} (t\mu)^{p^3} \lambda_3 \\
d^{2p^6+2p^3+1}(u_2 t^{-p^6}) &\doteq (t\mu)^{p^3+1} \,,
\end{align*}
and the classes $t\mu$, $\lambda_1$, $\lambda_2$, $\lambda_3$ and~$t^{\pm
p^6}$ are permanent cycles.
\end{theorem}

\begin{proof}
According to~\cite{AR02}*{Lem.~5.2}, naturality with respect to Frobenius
and Verschiebung maps forces the first three differentials, showing that
$$
\hat E^{2p^3+1}(C_{p^2}) = E(u_2) \otimes P(t^{\pm p^3})
        \otimes P(t\mu) \otimes E(\lambda_1, \lambda_2, \lambda_3) \,.
$$
To proceed, we shall make use of the summands
\begin{gather*}
P_p(t\mu) \otimes \bF_p\{\lambda_1 \mu\} \\
P_{p^2}(t\mu) \otimes \bF_p\{\lambda_2 \mu^p\} \\
P_{p^3}(t\mu) \otimes \bF_p\{\lambda_3 \mu^{p^2}\} \\
P_{p^3+1}(t\mu) \otimes \bF_p\{\mu^{p^3}\}
\end{gather*}
in $E^\infty(C_p)$, which is equal to $\mu^{-1} E^\infty(C_p)$ in these
degrees.  There are almost no classes in the same total degrees and of
lower filtration than the vanishing products
$$
(t\mu)^p \cdot \lambda_1 \mu \,,\quad
(t\mu)^{p^2} \cdot \lambda_2 \mu^p \,, \quad
(t\mu)^{p^3} \cdot \lambda_3 \mu^{p^2} \quad\text{and}\quad
(t\mu)^{p^3+1} \cdot \mu^{p^3} \,.
$$
The only exception is the class $(t\mu)^{p^2+p-1} \lambda_1
\lambda_2 \lambda_3$ in the same total degree as~$(t\mu)^{p^2}
\cdot \lambda_2 \mu^p$.  However, this class is itself a
$(t\mu)^{p^2}$-multiple, so there is no room for a hidden
$v_3^{p^2}$-extension on $\lambda_2 \mu^p$.  Hence $\lambda_1
\mu$ detects a $v_3^p$-torsion class~$x_1$, $\lambda_2 \mu^p$
detects a $v_3^{p^2}$-torsion class~$x_2$, $\lambda_3 \mu^{p^2}$
detects a $v_3^{p^3}$-torsion class~$x_3$, $\mu^{p^3}$ detects a
$v_3^{p^3+1}$-torsion class $x_4$ in $V(2)_* THH(BP\<2\>)^{hC_p}$,
and these $v_3$-power torsion orders are all exact.

By Corollary~\ref{cor:coconn} the maps $\Gamma_1$ and $\hat\Gamma_2$
are $(2p^2+2p-3)$-coconnected.  Hence the classes $x_i$ lift uniquely
to classes $y_i$ in $V(2)_* THH(BP\<2\>)^{C_p}$ with $\Gamma_1(y_i) =
x_i$, and we let $z_i = \hat\Gamma_2(y_i)$ denote their images in $V(2)_*
THH(BP\<2\>)^{tC_{p^2}}$.  Since $\hat\Gamma_1(\mu) = t^{-p^3}$,
up to a unit in $\bF_p$ that we hereafter often omit to mention, we see
that $F(z_1)$ is detected by $t^{-p^3} \lambda_1$, $F(z_2)$ is detected
by $t^{-p^4} \lambda_2$, $F(z_3)$ is detected by $t^{-p^5} \lambda_3$
and $F(z_4)$ is detected by $t^{-p^6}$ in~$\hat E^\infty(C_p)$.

We claim that there are no classes in $\hat E^\infty(C_{p^2})$
in the same total degrees and of higher filtrations than
$$
t^{-p^3} \lambda_1 \,,\quad
t^{-p^4} \lambda_2 \,,\quad
t^{-p^5} \lambda_3 \quad\text{and}\quad
t^{-p^6} \,.
$$
This will imply that the $z_i$ are detected by precisely these classes.
Already at the (known) $E^{2p^3+1}$-term the only exception to the
claim is $u_2 t^{-p^3-p^5}$ in the same total degree as~$t^{-p^5}
\lambda_3$, and we shall see below that this class supports a nonzero
$d^{2p^4+2p}$-differential, hence does not survive to the $E^\infty$-term.
It then follows that the products
$$
(t\mu)^p \cdot t^{-p^3} \lambda_1 \,,\quad
(t\mu)^{p^2} \cdot t^{-p^4} \lambda_2 \,,\quad
(t\mu)^{p^3} \cdot t^{-p^5} \lambda_3 \quad\text{and}\quad
(t\mu)^{p^3+1} \cdot t^{-p^6}
$$
must detect zero, and therefore be boundaries, in the $C_{p^2}$-Tate
spectral sequence $\hat E^r(C_{p^2})$.  We shall prove that these
boundaries must be
\begin{align*}
d^{2p^4+2p}(t^{-p^3-p^4}) &\doteq (t\mu)^p \cdot t^{-p^3} \lambda_1 \\
d^{2p^5+2p^2}(t^{-p^4-p^5}) &\doteq (t\mu)^{p^2} \cdot t^{-p^4} \lambda_2 \\
d^{2p^6+2p^3}(t^{-p^5-p^6}) &\doteq (t\mu)^{p^3} \cdot t^{-p^5} \lambda_3 \\
d^{2p^6+2p^3+1}(u_2 t^{-2p^6}) &\doteq (t\mu)^{p^3+1} \cdot t^{-p^6} \,,
\end{align*}
and the asserted formulas follow readily.

We shall make use of the following lemma.  Each cyclic $P(t\mu)$-module
is either free or torsion, being isomorphic to a suspension of~$P(t\mu)$
or of its truncation~$P_h(t\mu) = P(t\mu)/((t\mu)^h)$ at some height
$h\ge1$, according to the case.  Here, and below, exponents
$\epsilon$ and $\epsilon_i$ are always assumed to lie in $\{0,1\}$.

\begin{lemma} \label{lem:Cp2Tatedr}
For each $r\ge 2p^3+1$ the $C_{p^2}$-Tate $E^r$-term $\hat E^r(C_{p^2})$
is a direct sum of cyclic $P(t\mu)$-modules, generated by classes of the
form $u_2^\epsilon t^i \cdot \lambda_1^{\epsilon_1} \lambda_2^{\epsilon_2}
\lambda_3^{\epsilon_3}$ with $p^3 \mid i$.  The $d^r$-differential maps
free summands to free summands, and is zero on the torsion summands.
\end{lemma}

\begin{proof}
We proceed by induction on $r \ge 2p^3+1$, assuming that the
$P(t\mu)$-module structure of the $E^r$-term is as stated.

Suppose that there is a $d^r$-differential $d^r(a) = b$ hitting a
nonzero $t\mu$-torsion class.  Then $t\mu \cdot b = b_0 = d^{r_0}(a_0)$
must have been hit by an earlier $d^{r_0}$-differential, where $a_0$ is
a generator of the form $u_2^\epsilon t^i \cdot \lambda_1^{\epsilon_1}
\lambda_2^{\epsilon_2} \lambda_3^{\epsilon_3}$ with $p^3 \mid i$.  Hence
$a$ must lie in the same total degree as the formal product~$(t\mu)^{-1}
\cdot a_0$, but in a higher filtration.  At the $E^2$-term, this could
happen in three cases:
\begin{itemize}
\item
If $a_0 = u_2^\epsilon t^i \cdot \lambda_1 \lambda_2 \lambda_3$, with
$a$ in the bidegree of $u_2^\epsilon t^{i-p} \cdot \lambda_2$,
$u_2^\epsilon t^{i-p^2} \cdot \lambda_1$, $u_2 t^{i-p^2-p} \cdot 1$
or~$t^{i-p^2-p+1} \cdot 1$.
\item
If $a_0 = u_2^\epsilon t^i \cdot \lambda_2 \lambda_3$, with $a$ in
the bidegree of $u_2 t^{i-p^2+p-1} \cdot \lambda_1$, $t^{i-p^2+p}
\cdot \lambda_1$ or $u_2^\epsilon t^{i-p^2} \cdot 1$.
\item
If $a_0 = u_2^\epsilon t^i \cdot \lambda_1 \lambda_3$, with $a$ in
the bidegree of $u_2^\epsilon t^{i-p} \cdot 1$.
\end{itemize}
However, in none of these cases is the prescribed $t$-exponent ($i-p$,
$i-p^2$, etc.) a multiple of~$p^3$.  Hence there are no nonzero classes
in these bidegrees of~$\hat E^{2p^3+1}(C_{p^2})$, and therefore also
not in~$\hat E^r(C_{p^2})$ for $r \ge 2p^3+1$.

It follows that no differentials hit the torsion summands, so each
nonzero differentials maps a free summand to another free summand.  Its
kernel is then zero, while its cokernel creates a torsion summand in the
$E^{r+1}$-term, which is still generated by a class of the required form.
This proves the inductive statement for~$r+1$.
\end{proof}

The remainder of the proof of Theorem~\ref{thm:Cp2Tate}
can be separated into five steps.

(1) We start with~$z_1$, which we know is detected by~$t^{-p^3} \lambda_1$.
Checking bidegrees in~$\hat E^{2p^3+1}(C_{p^2})$, the next possible
differentials on $u_2$ and $t^{-p^3}$ are
\begin{align*}
d^{2p^4+2p-1}(t^{-p^3}) &\in
	\bF_p\{ u_2 t^{-p^3+p^4} (t\mu)^{p-1} \lambda_1 \lambda_3\} \\
d^{2p^4+2p}(t^{-p^3}) &\in
	\bF_p\{ t^{-p^3+p^4} (t\mu)^p \lambda_1 \} \\
d^{2p^4+2p}(u_2) &\in
	\bF_p\{ u_2 t^{p^4} (t\mu)^p \lambda_1 \} \,.
\end{align*}
Since $t\mu$ and the $\lambda_i$ are infinite cycles, we must have $d^r =
0$ for $2p^3+1 \le r < 2p^4+2p-1$.  Moreover, $(t\mu)^p \cdot t^{-p^3}
\lambda_1 \doteq d^{r_1}(a_1)$ in vertical degree~$2p^4+2p-1$ must be a
boundary, and the only possible source of such a $d^{r_1}$-differential
with $r_1 \ge 2p^4+2p-1$ is~$a_1 = t^{-p^3-p^4}$ with $r_1 = 2p^4+2p$.
It follows that $d^{2p^4+2p-1}(t^{-p^3}) = 0$ vanishes and, furthermore,
that $d^{2p^4+2p}(t^{-p^3}) \doteq t^{-p^3+p^4} (t\mu)^p \lambda_1$
is nonzero.

(2) We turn to $z_4$, which we know is detected by~$t^{-p^6}$.
Thus $t^{p^6}$ and its inverse are permanent cycles.
The nonzero product $v_3^{p^3} \cdot z_4$ is detected by $b_4 =
(t\mu)^{p^3} \cdot t^{-p^6}$ or, if this product is a boundary, by
another class in the same total degree as~$b_4$ but of lower filtration.
Let $b_4'$ denote the actual detecting class.  Then $t\mu \cdot b_4' \doteq
d^{r_4}(a_4)$ in total degree~$4p^6-2$ detects $v_3^{p^3+1} \cdot z_4
= 0$, hence is a boundary.  By Lemma~\ref{lem:Cp2Tatedr}, the
source of this differential is of the form $a_4 = u_2^\epsilon t^i \cdot
\lambda_1^{\epsilon_1} \lambda_2^{\epsilon_2} \lambda_3^{\epsilon_3}$,
with $p^3 \mid i$.  (If $a_4$ were a $t\mu$-multiple at the $E^2$-term,
then it would be a $t\mu$-multiple at the $E^{r_4}$-term, by the lemma.
Then $b_4'$ would be a $d^{r_4}$-boundary, which is impossible since it detects
$v_3^{p^3} \cdot z_4 \ne 0$.)  The total degree of~$a_4$ is $4p^6-1$,
so the only possibilities are $t^{p^3-2p^6} \lambda_3$ with $r_4 \ge
2p^6+2$, or $u_2 t^{-2p^6}$ with $r_4 \ge 2p^6+2p^3+1$.  However, we
showed in~(1) that $t^{p^3-2p^6} \lambda_3$ supports a nonzero (shorter)
$d^{2p^4+2p}$-differential.  Hence $a_4 = u_2 t^{-2p^6}$ survives at
least to the $E^{2p^6+2p^3+1}$-term, and $d^{r_4}(a_4) \ne 0$ for some
$r_4 \ge 2p^6+2p^3+1$.	Since $t^{p^6}$ is an infinite cycle it follows
that $u_2$ also survives to the $E^{2p^6+2p^3+1}$-term.  Hence
\begin{align*}
\hat E^{2p^4+2p+1}(C_{p^2}) &= E(u_2) \otimes P(t^{\pm p^4})
	\otimes P(t\mu) \otimes E(\lambda_1, \lambda_2, \lambda_3) \\
    &\qquad\oplus E(u_2) \otimes P_p(t\mu) \otimes E(\lambda_2, \lambda_3)
	\otimes \bF_p\{t^i \lambda_1 \mid v_p(i) = 3\} \,.
\end{align*}
In particular, $u_2 t^{-p^3-p^5}$ is not an infinite cycle,
and cannot detect $z_3$, confirming our earlier claim.

(3) We continue with~$z_2$, which we know is detected by~$t^{-p^4} \lambda_2$.
Checking bidegrees in~$\hat E^{2p^4+2p+1}(C_{p^2})$, the next possible
differentials on $t^{-p^4}$ are
\begin{align*}
d^{2p^5+2p^2-1}(t^{-p^4}) &\in
	\bF_p\{ u_2 t^{-p^4+p^5} (t\mu)^{p^2-1} \lambda_2 \lambda_3\} \\
d^{2p^5+2p^2}(t^{-p^4}) &\in
	\bF_p\{ t^{-p^4+p^5} (t\mu)^{p^2} \lambda_2 \} \,,
\end{align*}
while $u_2$ survives at least to $\hat E^{2p^6+2p^3+1}(C_{p^2})$ by~(2).
The differentials on $t\mu$, the $\lambda_i$, and the
torsion summand are zero.  Hence $d^r = 0$ for $2p^4+2p+1 \le r <
2p^5+2p^2-1$.  Moreover, $(t\mu)^{p^2} \cdot t^{-p^4} \lambda_2 =
d^{r_2}(a_2)$ in vertical degree~$2p^5+2p^2-1$ must be a boundary,
and the only possible source of such a $d^{r_2}$-differential with $r_2
\ge 2p^5+2p^2-1$ is~$a_2 \doteq t^{-p^4-p^5}$ with $r_2 = 2p^5+2p^2$.
It follows that $d^{2p^5+2p^2-1}(t^{-p^4}) = 0$ vanishes and that
$d^{2p^5+2p^2}(t^{-p^4}) \doteq t^{-p^4+p^5} (t\mu)^{p^2} \lambda_2$ is
nonzero.  Hence
\begin{align*}
\hat E^{2p^5+2p^2+1}(C_{p^2}) &= E(u_2) \otimes P(t^{\pm p^5})
	\otimes P(t\mu) \otimes E(\lambda_1, \lambda_2, \lambda_3) \\
	&\qquad\oplus E(u_2) \otimes P_p(t\mu)
		\otimes E(\lambda_2, \lambda_3)
		\otimes \bF_p\{t^i \lambda_1 \mid v_p(i) = 3\} \\
	&\qquad\oplus E(u_2) \otimes P_{p^2}(t\mu)
		\otimes E(\lambda_1, \lambda_3)
		\otimes \bF_p\{t^i \lambda_2 \mid v_p(i) = 4\} \,.
\end{align*}

(4) Next up is~$z_3$, which we know from~(2) is detected by~$t^{-p^5}
\lambda_3$.  Checking bidegrees in~$\hat E^{2p^5+2p^2+1}(C_{p^2})$,
the next possible differential on $t^{-p^5}$ is
$$
d^{2p^6+2p^3}(t^{-p^5}) \in
	\bF_p\{ t^{-p^5+p^6} (t\mu)^{p^3} \lambda_3 \} \,,
$$
while $u_2$ survives to the $E^{2p^6+2p^3+1}$-term by~(2).
The differentials on $t\mu$, the $\lambda_i$, and the torsion summands
are zero.  Hence $d^r = 0$ for $2p^5+2p^2+1 \le r < 2p^6+2p^3$.
Moreover, $(t\mu)^{p^3} \cdot t^{-p^5} \lambda_3 = d^{r_3}(a_3)$ in
vertical degree~$2p^6+2p^3-1$ must be a boundary, and the only possible
source of such a differential is~$a_3 \doteq t^{-p^5-p^6}$ with $r_3 =
2p^6+2p^3$.  It follows that $d^{2p^6+2p^3}(t^{-p^5}) \doteq t^{-p^5+p^6}
(t\mu)^{p^3} \lambda_3$ is nonzero.  Hence
\begin{align*}
\hat E^{2p^6+2p^3+1}(C_{p^2}) &= E(u_2) \otimes P(t^{\pm p^6})
	\otimes P(t\mu) \otimes E(\lambda_1, \lambda_2, \lambda_3) \\
	&\qquad\oplus E(u_2) \otimes P_p(t\mu)
		\otimes E(\lambda_2, \lambda_3)
		\otimes \bF_p\{t^i \lambda_1 \mid v_p(i) = 3\} \\
	&\qquad\oplus E(u_2) \otimes P_{p^2}(t\mu)
		\otimes E(\lambda_1, \lambda_3)
		\otimes \bF_p\{t^i \lambda_2 \mid v_p(i) = 4\} \\
	&\qquad\oplus E(u_2) \otimes P_{p^3}(t\mu)
		\otimes E(\lambda_1, \lambda_2)
		\otimes \bF_p\{t^i \lambda_3 \mid v_p(i) = 5\} \,.
\end{align*}

(5) Finally, we return to~$z_4$.  Since $b_4 = (t\mu)^{p^3} \cdot t^{-p^6}$
is nonzero, in vertical degree~$2p^6$ of the $E^{2p^6+2p^3+1}$-term,
it can no longer become a boundary.  We can therefore strengthen the
conclusions in~(2) to conclude that $v_3^{p^3} \cdot z_4$ is detected
by $b_4' = b_4$, and that $t\mu \cdot b_4 = (t\mu)^{p^3+1} \cdot t^{-p^6}$
is a unit times $d^{r_4}(a_4)$, with $r_4 = 2p^6+2p^3+1$ and $a_4 =
u_2 t^{-2p^6}$. It follows that $d^{2p^6+2p^3+1}(u_2 t^{-p^6}) \doteq
(t\mu)^{p^3+1}$, since $t^{p^6}$ is an infinite cycle.  Hence
\begin{align*}
\hat E^{2p^6+2p^3+2}(C_{p^2}) &= P(t^{\pm p^6})
	\otimes P_{p^3+1}(t\mu) \otimes E(\lambda_1, \lambda_2, \lambda_3) \\
	&\qquad\oplus E(u_2) \otimes P_p(t\mu)
		\otimes E(\lambda_2, \lambda_3)
		\otimes \bF_p\{t^i \lambda_1 \mid v_p(i) = 3\} \\
	&\qquad\oplus E(u_2) \otimes P_{p^2}(t\mu)
		\otimes E(\lambda_1, \lambda_3)
		\otimes \bF_p\{t^i \lambda_2 \mid v_p(i) = 4\} \\
	&\qquad\oplus E(u_2) \otimes P_{p^3}(t\mu)
		\otimes E(\lambda_1, \lambda_2)
		\otimes \bF_p\{t^i \lambda_3 \mid v_p(i) = 5\} \,.
\end{align*}
No free summands remain, so by Lemma~\ref{lem:Cp2Tatedr}
there are no further differentials, and this $E^r$-term equals $\hat
E^\infty(C_{p^2})$.
\end{proof}

\section{The $C_{p^n}$-Tate spectral sequences}

The following notations will be convenient when we now determine the
differential structure of the $C_{p^n}$-Tate spectral sequence.

\begin{definition}[\cite{AR02}*{Def.~2.5}, \cite{AKCH}*{(5.8)}]
Let $r(k) = 0$ for $k \in \{0,-1,-2\}$ and set $r(k) = p^k + r(k-3)$
for $k\ge1$.  Thus $r(3n-2) = p^{3n-2} + \dots + p$, $r(3n-1) = p^{3n-1}
+ \dots + p^2$ and $r(3n) = p^{3n} + \dots + p^3$, with $n$ terms in
each sum.

Let $[k] \in \{1,2,3\}$ be defined by $k \equiv [k] \mod 3$, so that
$\{\lambda_{[k]}, \lambda_{[k+1]}, \lambda_{[k+2]}\} = \{\lambda_1,
\lambda_2, \lambda_3\}$.
\end{definition}

\begin{theorem} \label{thm:CpnTate}
The $C_{p^n}$-Tate spectral sequence
\begin{align*}
\hat E^2(C_{p^n}) &= \hat H^{-*}(C_{p^n}; V(2)_* THH(BP\<2\>)) \\
	&\Longrightarrow V(2)_* THH(BP\<2\>)^{tC_{p^n}}
\end{align*}
has $E^2$-term
$$
\hat E^2(C_{p^n}) = E(u_n) \otimes P(t^{\pm1})
	\otimes P(t\mu) \otimes E(\lambda_1, \lambda_2, \lambda_3) \,.
$$
There are differentials
$$
d^{2r(k)}(t^{p^{k-1} - p^k}) \doteq t^{p^{k-1}} (t\mu)^{r(k-3)} \lambda_{[k]}
$$
for each $1 \le k \le 3n$, and
$$
d^{2r(3n)+1}(u_n t^{-p^{3n}}) \doteq (t\mu)^{r(3n-3)+1} \,.
$$
The classes $t\mu$, $\lambda_1$, $\lambda_2$, $\lambda_3$ and~$t^{\pm p^{3n}}$
are permanent cycles.
\end{theorem}

For $n=1$, this is Theorem~\ref{thm:CpTate}.  We prove the statement for
general~$n$ by induction, assuming the statement to be true for one value
of $n\ge2$, and deducing that it also holds for $n+1$.  The inductive
beginning for $n=2$ is provided by Theorem~\ref{thm:Cp2Tate}.

The distinct terms of the $C_{p^n}$-Tate spectral sequence are
\begin{multline*}
\hat E^{2r(m)+1}(C_{p^n}) = E(u_n) \otimes P(t^{\pm p^m})
	\otimes P(t\mu) \otimes E(\lambda_1, \lambda_2, \lambda_3) \\
\oplus \bigoplus_{k=4}^m E(u_n) \otimes P_{r(k-3)}(t\mu)
	\otimes E(\lambda_{[k+1]}, \lambda_{[k+2]})
        \otimes \bF_p\{ t^i \lambda_{[k]} \mid v_p(i) = k-1 \}
\end{multline*}
for $1 \le m \le 3n$.  To see this, note that the differential $d^{2r(k)}$
only affects the summand $E(u_n) \otimes \bF_p\{ t^i \mid v_p(i) = k-1 \}
\otimes P(t\mu) \otimes E(\lambda_1, \lambda_2, \lambda_3)$, and here its
homology is $E(u_n) \otimes P_{r(k-3)}(t\mu) \otimes E(\lambda_{[k+1]},
\lambda_{[k+2]}) \otimes \bF_p\{ t^i \lambda_{[k]} \mid v_p(i) = k-1 \}$.
Thereafter,
\begin{multline*}
\hat E^{2r(3n)+2}(C_{p^n}) = P(t^{\pm p^{3n}})
        \otimes P_{r(3n-3)+1}(t\mu)
	\otimes E(\lambda_1, \lambda_2, \lambda_3) \\
\oplus \bigoplus_{k=4}^{3n} E(u_n) \otimes P_{r(k-3)}(t\mu)
	\otimes E(\lambda_{[k+1]}, \lambda_{[k+2]})
        \otimes \bF_p\{ t^i \lambda_{[k]} \mid v_p(i) = k-1 \} \,.
\end{multline*}
To see this, note that $d^{2r(3n)+1}$ only affects the summand
$E(u_n) \otimes P(t^{\pm p^{3n}})
        \otimes P(t\mu) \otimes E(\lambda_1, \lambda_2, \lambda_3)$,
and that its homology is
$P(t^{\pm p^{3n}}) \otimes P_{r(3n-3)+1}(t\mu)
        \otimes E(\lambda_1, \lambda_2, \lambda_3)$.
For bidegree reasons the remaining differentials are zero, so $\hat
E^{2r(3n)+2}(C_{p^n}) = \hat E^\infty(C_{p^n})$, and the classes $t^{\pm
p^{3n}}$ are permanent cycles.

The differential structure of the $C_{p^n}$-homotopy fixed point
spectral sequence $E^r(C_{p^n})$ for $V(2) \wedge THH(BP\<2\>)$
is obtained from that of the $C_{p^n}$-Tate spectral sequence $\hat
E^r(C_{p^n})$ by restricting to the second quadrant.  We write $\mu^{-1}
E^r(C_{p^n})$ for its localization given by inverting (a power of) $\mu$.
It follows from Theorem~\ref{thm:CpTate} that $\mu^{-1} E^r(C_{p^n})$
is isomorphic to the $C_{p^n}$-homotopy fixed point spectral sequence
for $V(2) \wedge THH(BP\<2\>)^{tC_p}$

\begin{proposition} \label{prop:locCpnHFP}
The $\mu$-localized $C_{p^n}$-homotopy fixed point spectral sequence
\begin{align*}
\mu^{-1} E^2(C_{p^n}) &= H^{-*}(C_{p^n}; \mu^{-1} V(2)_* THH(BP\<2\>)) \\
	&\Longrightarrow \mu^{-1} V(2)_* THH(BP\<2\>)^{hC_{p^n}}
\end{align*}
has $E^2$-term
$$
\mu^{-1} E^2(C_{p^n}) = E(u_n) \otimes P(t\mu)
	\otimes E(\lambda_1, \lambda_2, \lambda_3) \otimes P(\mu^{\pm1}) \,.
$$
There are differentials
$$
d^{2r(k)}(\mu^{p^{k-1}})
	\doteq (t\mu)^{r(k)} \lambda_{[k]} \mu^{p^{k-1}-p^k}
$$
for each $1 \le k \le 3n$, and
$$
d^{2r(3n)+1}(u_n \mu^{p^{3n}}) \doteq (t\mu)^{r(3n)+1} \,.
$$
The classes $t\mu$, $\lambda_1$, $\lambda_2$, $\lambda_3$ and $\mu^{\pm
p^{3n}}$ are permanent cycles.
\end{proposition}

\begin{proof}
This follows from Theorem~\ref{thm:CpnTate} by comparison along the
morphism
$$
R^h \: E^r(C_{p^n}) \longto \hat E^r(C_{p^n})
$$
of spectral sequences induced by the homotopy restriction (= canonical) map,
and the $(2p^2+2p-3)$-coconnected localization morphism
$$
E^r(C_{p^n}) \longto \mu^{-1} E^r(C_{p^n}) \,.
$$
Algebraically, the translation is achieved through multiplication with
appropriate powers of $t\mu$.
\end{proof}

The distinct terms of the $\mu$-localized $C_{p^n}$-homotopy fixed point
spectral sequence are
\begin{multline*}
\mu^{-1} E^{2r(m)+1}(C_{p^n}) = E(u_n) \otimes P(t\mu)
	\otimes E(\lambda_1, \lambda_2, \lambda_3)
	\otimes P(\mu^{\pm p^m}) \\
\oplus \bigoplus_{k=1}^m E(u_n) \otimes P_{r(k)}(t\mu)
	\otimes E(\lambda_{[k+1]}, \lambda_{[k+2]})
	\otimes \bF_p\{ \lambda_{[k]} \mu^j \mid v_p(j) = k-1\}
\end{multline*}
for $1 \le m \le 3n$.  To see this, note that the differential
$d^{2r(k)}$ only affects the summand
$E(u_n) \otimes P(t\mu) \otimes E(\lambda_1, \lambda_2, \lambda_3)
	\otimes \bF_p\{\mu^j \mid v_p(j) = k-1\}$,
and here its homology is
$E(u_n) \otimes P_{r(k)}(t\mu) \otimes E(\lambda_{[k+1]}, \lambda_{[k+2]})
	\otimes \bF_p\{\lambda_{[k]} \mu^j \mid v_p(j) = k-1\}$.
Thereafter
\begin{multline*}
\mu^{-1} E^{2r(3n)+2}(C_{p^n}) = P_{r(3n)+1}(t\mu)
        \otimes E(\lambda_1, \lambda_2, \lambda_3)
	\otimes P(\mu^{\pm p^{3n}}) \\
\oplus \bigoplus_{k=1}^{3n} E(u_n) \otimes P_{r(k)}(t\mu)
        \otimes E(\lambda_{[k+1]}, \lambda_{[k+2]})
        \otimes \bF_p\{ \lambda_{[k]} \mu^j \mid v_p(j) = k-1\} \,.
\end{multline*}
As before, $d^{2r(3n)+1}$ only affects the summand
$E(u_n) \otimes P(t\mu) \otimes E(\lambda_1, \lambda_2, \lambda_3)
	\otimes P(\mu^{\pm p^{3n}})$,
and its homology is
$P_{r(3n)+1}(t\mu) \otimes E(\lambda_1, \lambda_2, \lambda_3)
	\otimes P(\mu^{\pm p^{3n}})$.
For bidegree reasons the remaining differentials are zero, so $\mu^{-1}
E^{2r(3n)+2}(C_{p^n}) = \mu^{-1} E^\infty(C_{p^n})$, and the classes
$\mu^{\pm p^{3n}}$ are permanent cycles.

To achieve the inductive step we use the commutative diagram
\begin{equation} \label{eq:GammasFn}
\xymatrix{
THH(BP\<2\>)^{hC_{p^n}} \ar[d]^-{F^n}
	& THH(BP\<2\>)^{C_{p^n}} \ar[d]^-{F^n}
		\ar[l]_-{\Gamma_n} \ar[r]^-{\hat\Gamma_{n+1}}
	& THH(BP\<2\>)^{tC_{p^{n+1}}} \ar[d]^-{F^n} \\
THH(BP\<2\>)
	& THH(BP\<2\>) \ar@{=}[l] \ar[r]^-{\hat\Gamma_1}
	& THH(BP\<2\>)^{tC_p}
}
\end{equation}
and what is known about $V(2)_* THH(BP\<2\>)^{hC_{p^n}}$ above
degree~$2p^2+2p-3$ to determine the differential pattern
of the $C_{p^{n+1}}$-Tate spectral sequence converging to $V(2)_*
THH(BP\<2\>)^{tC_{p^{n+1}}}$.

\begin{proof}[Proof of Theorem~\ref{thm:CpnTate}]
We must show that the $C_{p^{n+1}}$-Tate spectral sequence
\begin{align*}
\hat E^2(C_{p^{n+1}}) &= \hat H^{-*}(C_{p^{n+1}}; V(2)_* THH(BP\<2\>)) \\
	&\Longrightarrow V(2)_* THH(BP\<2\>)^{tC_{p^{n+1}}}
\end{align*}
has the asserted differential pattern.  By naturality with respect to
(Tate spectrum) Frobenius and Verschiebung morphisms
$$
\xymatrix{
F \: \hat E^r(C_{p^{n+1}}) \ar@<1ex>[r]
	& \hat E^r(C_{p^n}) \ar@<1ex>[l] \: V
}
$$
it follows as in~\cite{AR02}*{Lem.~5.2} that the left hand spectral
sequence has differentials
$$
d^{2r(k)}(t^{p^{k-1} - p^k})
	\doteq t^{p^{k-1}} (t\mu)^{r(k-3)} \lambda_{[k]}
$$
for all $1 \le k \le 3n$, leading via the
$E^{2r(3)+1} = E^{2p^3+1}$-term
$$
\hat E^{2p^3+1}(C_{p^{n+1}}) = E(u_{n+1}) \otimes P(t^{\pm p^3})
	\otimes P(t\mu) \otimes E(\lambda_1, \lambda_2, \lambda_3)
$$
to the $E^{2r(3n)+1}$-term
\begin{multline*}
\hat E^{2r(3n)+1}(C_{p^{n+1}}) = E(u_{n+1}) \otimes P(t^{\pm p^{3n}})
        \otimes P(t\mu) \otimes E(\lambda_1, \lambda_2, \lambda_3) \\
\oplus \bigoplus_{k=4}^{3n} E(u_{n+1}) \otimes P_{r(k-3)}(t\mu)
        \otimes E(\lambda_{[k+1]}, \lambda_{[k+2]})
        \otimes \bF_p\{ t^i \lambda_{[k]} \mid v_p(i) = k-1 \} \,.
\end{multline*}
We shall prove that this spectral sequence contains three more families
of even length differentials, followed by one family of odd length
differentials, after which it collapses.

Note that the $E^{2p^3+1}$-term is free as a $P(t\mu)$-module.  Replacing
$C_{p^2}$ with $C_{p^{n+1}}$ and $u_2$ with $u_{n+1}$ in the proof of
Lemma~\ref{lem:Cp2Tatedr}, with no other changes, establishes the more
general statement below.

\begin{lemma} \label{lem:Cpn+1Tatedr}
For each $r\ge 2p^3+1$ the $C_{p^{n+1}}$-Tate $E^r$-term $\hat
E^r(C_{p^{n+1}})$ is a direct sum of cyclic $P(t\mu)$-modules, generated
by classes of the form $u_{n+1}^\epsilon t^i \cdot \lambda_1^{\epsilon_1}
\lambda_2^{\epsilon_2} \lambda_3^{\epsilon_3}$ with $p^3 \mid i$.
The $d^r$-differential maps free summands to free summands, and is zero
on the torsion summands.
\qed
\end{lemma}

By our inductive hypothesis, the abutment $E^\infty(C_{p^n})$, which
is isomorphic to $\mu^{-1} E^\infty(C_{p^n})$ above degree~$2p^2+2p-3$,
contains summands
\begin{gather*}
P_{r(3n-2)}(t\mu) \otimes \bF_p\{\lambda_1 \mu^{p^{3n-3}}\} \\
P_{r(3n-1)}(t\mu) \otimes \bF_p\{\lambda_2 \mu^{p^{3n-2}}\} \\
P_{r(3n)}(t\mu) \otimes \bF_p\{\lambda_3 \mu^{p^{3n-1}}\} \\
P_{r(3n)+1}(t\mu) \otimes \bF_p\{\mu^{p^{3n}}\} \,.
\end{gather*}
Moreover, $\mu^{-1} E^\infty(C_{p^n})$ is generated as a
$P(t\mu)$-module by classes in filtration $-1$ and~$0$.  Hence any
class in $E^\infty(C_{p^n})$ in the same total degree as, but of lower
filtration than, one of the vanishing products
\begin{gather*}
(t\mu)^{r(3n-2)} \cdot \lambda_1 \mu^{p^{3n-3}} \,,\quad
(t\mu)^{r(3n-1)} \cdot \lambda_2 \mu^{p^{3n-2}} \,,\quad \\
(t\mu)^{r(3n)} \cdot \lambda_3 \mu^{p^{3n-1}} \quad\text{and}\quad
(t\mu)^{r(3n)+1} \cdot \mu^{p^{3n}} \,,
\end{gather*}
must itself be divisible by (at least) the indicated power of $t\mu$.
It follows that there are no hidden $v_3$-power extensions present,
so that $\lambda_1 \mu^{p^{3n-3}}$ detects a $v_3^{r(3n-2)}$-torsion
class~$x_1 \in V(2)_* THH(BP\<2\>)^{hC_{p^n}}$, $\lambda_2 \mu^{p^{3n-2}}$
detects a $v_3^{r(3n-1)}$-torsion class~$x_2$, $\lambda_3 \mu^{p^{3n-1}}$
detects a $v_3^{r(3n)}$-torsion class~$x_3$, and $\mu^{p^{3n}}$ detects
a $v_3^{r(3n)+1}$-torsion class~$x_4$, and these $v_3$-power torsion
orders are exact.

By Corollary~\ref{cor:coconn} there are unique classes $y_i \in V(2)_*
THH(BP\<2\>)^{C_{p^n}}$ and $z_i \in V(2)_* THH(BP\<2\>)^{tC_{p^{n+1}}}$
with $\Gamma_n(y_i) = x_i$ and $\hat\Gamma_{n+1}(y_i) = z_i$ for
each~$i$.  Moreover, $z_1, \dots, z_4$ are $v_3$-power torsion classes of
order precisely $r(3n-2)$, $r(3n-1)$, $r(3n)$ and $r(3n)+1$, respectively.

Applying Frobenius maps $F^n$ as in diagram~\eqref{eq:GammasFn},
and the fact from Theorem~\ref{thm:CpTate} that $\hat\Gamma_1$ maps
$\mu$ to $t^{-p^3}$ (up to the usual implicit unit) and preserves the
$\lambda_i$, we deduce that $F^n(z_1), \dots, F^n(z_4)$ are detected
by the classes $t^{-p^{3n}} \lambda_1$, $t^{-p^{3n+1}} \lambda_2$,
$t^{-p^{3n+2}} \lambda_3$ and~$t^{-p^{3n+3}}$ in $\hat E^\infty(C_p)$.
Hence $z_1, \dots, z_4$ are detected in $\hat E^\infty(C_{p^{n+1}})$ in
the same total degree as these classes, in equal or higher filtration.
However, since $n\ge2$ there are no possible detecting classes of
strictly higher filtration present in $\hat E^{2r(3n)+1}(C_{p^{n+1}})$.
We can therefore conclude that $z_1, \dots, z_4$ are detected by
$$
t^{-p^{3n}} \lambda_1 \,,\quad
t^{-p^{3n+1}} \lambda_2 \,,\quad
t^{-p^{3n+2}} \lambda_3 \quad\text{and}\quad
t^{-p^{3n+3}} \,,
$$
respectively, in $\hat E^\infty(C_{p^{n+1}})$.  (The only problematic
class at the $E^{2r(3)+1}$-term, $u_{n+1} t^{-p^3-p^{3n+2}}$ in
the same total degree as~$t^{-p^{3n+2}} \lambda_3$, is now known
to support a $d^{2r(4)}$-differential, as in the $C_{p^2}$-case.)

It follows that the products
\begin{gather*}
(t\mu)^{r(3n-2)} \cdot t^{-p^{3n}} \lambda_1 \,,\quad
(t\mu)^{r(3n-1)} \cdot t^{-p^{3n+1}} \lambda_2 \,,\quad \\
(t\mu)^{r(3n)} \cdot t^{-p^{3n+2}} \lambda_3 \quad\text{and}\quad
(t\mu)^{r(3n)+1} \cdot t^{-p^{3n+3}}
\end{gather*}
must detect zero, and therefore be boundaries, in the $C_{p^{n+1}}$-Tate
spectral sequence.  We shall prove that these boundaries must be
\begin{align*}
d^{2r(3n+1)}(t^{-p^{3n}-p^{3n+1}})
	&\doteq (t\mu)^{r(3n-2)} \cdot t^{-p^{3n}} \lambda_1 \\
d^{2r(3n+2)}(t^{-p^{3n+1}-p^{3n+2}})
	&\doteq (t\mu)^{r(3n-1)} \cdot t^{-p^{3n+1}} \lambda_2 \\
d^{2r(3n+3)}(t^{-p^{3n+2}-p^{3n+3}})
	&\doteq (t\mu)^{r(3n)} \cdot t^{-p^{3n+2}} \lambda_3 \\
d^{2r(3n+3)+1}(u_{n+1} t^{-2p^{3n+3}})
	&\doteq (t\mu)^{r(3n)+1} \cdot t^{-p^{3n+3}} \,.
\end{align*}
In view of the Leibniz rule, the first three can be rewritten as
$$
d^{2r(k)}(t^{p^{k-1}-p^k})
	\doteq t^{p^{k-1}} (t\mu)^{r(k-3)} \lambda_{[k]}
$$
for $3n+1 \le k \le 3n+3$, while the fourth is equivalent to
$$
d^{2r(3n+3)+1}(u_{n+1} t^{-p^{3n+3}}) \doteq (t\mu)^{r(3n)+1} \,.
$$

As for Theorem~\ref{thm:Cp2Tate}, the remainder of the proof of
Theorem~\ref{thm:CpnTate} will consist of five steps, but for $n\ge2$
we can start with~$z_4$ in place of~$z_1$, and this simplifies the
discussion of the class $u_{n+1}$.

(1) We know that $z_4$ is detected by~$t^{-p^{3n+3}}$.  Thus
$t^{p^{3n+3}}$ and its inverse are permanent cycles.  The nonzero
product $v_3^{r(3n)} \cdot z_4$ is detected by $b_4 = (t\mu)^{r(3n)}
\cdot t^{-p^{3n+3}}$ or, if this product is a boundary, by another
class in the same total degree as~$b_4$ but of lower filtration.
Let $b_4'$ denote the actual detecting class.  Then $t\mu \cdot b_4'$
in total degree~$4p^{3n+3}-2$ and vertical degree $\ge 2 r(3n+3)$ detects
$v_3^{r(3n)+1} \cdot z_4 = 0$, hence is a boundary.  We write $t\mu \cdot
b_4' \doteq d^{r_4}(a_4)$.  By Lemma~\ref{lem:Cpn+1Tatedr}, the source
of this differential is of the form $a_4 = u_{n+1}^\epsilon t^i \cdot
\lambda_1^{\epsilon_1} \lambda_2^{\epsilon_2} \lambda_3^{\epsilon_3}$,
with $p^3 \mid i$.  The total degree of~$a_4$ is $4p^{3n+3}-1$, so
the only possible sources are $t^{p^3-2p^{3n+3}} \lambda_3$, with
$r_4 \ge 2r(3n+3)-2p^3+2$, or $u_{n+1} t^{-2p^{3n+3}}$.  However,
since $n\ge2$ the first of these possibilities is no longer present in
$\hat E^{2r(3n)+1}(C_{p^{n+1}})$.  Hence $a_4 = u_{n+1} t^{-2p^{3n+3}}$
survives at least to the $E^{2r(3n+3)+1}$-term, and $d^{r_4}(a_4) \ne
0$ for some $r_4 \ge 2r(3n+3)+1$.  Since $t^{p^{3n+3}}$ is an infinite
cycle it also follows that $d^r(u_{n+1}) = 0$ for all $r \le 2r(3n+3)$.

(2) We continue with $z_1$, which is detected by~$t^{-p^{3n}} \lambda_1$.
The nonzero product $v_3^{r(3n-2)-1} \cdot z_1$ is detected by $b_1 =
(t\mu)^{r(3n-2)-1} \cdot t^{-p^{3n}} \lambda_1$ or, if this product
is a boundary, by another class in the same total degree as~$b_1$
but of lower filtration.  Let $b_1'$ denote the detecting class.
Then $t\mu \cdot b_1'$ in total degree $2p^{3n+1} + 2p^{3n} -
1$ and vertical degree $\ge 2r(3n+1)-1$ detects $v_3^{r(3n-2)}
\cdot z_1 = 0$, hence is a boundary.  We write $t\mu \cdot b_1'
\doteq d^{r_1}(a_1)$.  By Lemma~\ref{lem:Cpn+1Tatedr} the source of
this differential is of the form $a_1 = u_{n+1}^\epsilon t^i \cdot
\lambda_1^{\epsilon_1} \lambda_2^{\epsilon_2} \lambda_3^{\epsilon_3}$
with $p^3 \mid i$.  The only such class in the correct total degree
is $a_1 = t^{-p^{3n}-p^{3n+1}}$.  Considering vertical degrees, it
follows that $r_1 \ge 2r(3n+1)$.  Since the torsion summands in $\hat
E^{2r(3n)+1}(C_{p^{n+1}})$ are not affected by later differentials, the
$\lambda_i$ and $t\mu$ are infinite cycles, and $u_{n+1}$ survives to the
$E^{2r(3n+3)+1}$-term by~(1), it follows that $d^r = 0$ for $2r(3n) <
r < 2r(3n+1)$.  After this, $b_1$ is in too low a vertical degree to be
a boundary.  Hence $b_1' = b_1$ and $r_1 = 2r(3n+1)$.  It follows that
$d^{2r(3n+1)}(t^{-p^{3n}}) \doteq t^{-p^{3n}+p^{3n+1}} (t\mu)^{r(3n-2)}
\lambda_1$.  This establishes the first new even length differential,
and leads to the $E^{2r(3n+1)+1}$-term
\begin{multline*}
\hat E^{2r(3n+1)+1}(C_{p^{n+1}}) = E(u_{n+1}) \otimes P(t^{\pm p^{3n+1}})
        \otimes P(t\mu) \otimes E(\lambda_1, \lambda_2, \lambda_3) \\
\oplus \bigoplus_{k=4}^{3n+1} E(u_{n+1}) \otimes P_{r(k-3)}(t\mu)
        \otimes E(\lambda_{[k+1]}, \lambda_{[k+2]})
        \otimes \bF_p\{ t^i \lambda_{[k]} \mid v_p(i) = k-1 \} \,.
\end{multline*}

(3) Next we turn to $z_2$, which is detected by~$t^{-p^{3n+1}} \lambda_2$.
\footnote{Steps~(3) and~(4) are very similar to step~(2), but we believe
the arguments are easier to follow when written out separately.}
The nonzero product $v_3^{r(3n-1)-1} \cdot z_2$ is detected by $b_2 =
(t\mu)^{r(3n-1)-1} \cdot t^{-p^{3n+1}} \lambda_2$ or, if this product is
a boundary, by another class in the same total degree as~$b_2$ but of
lower filtration.  Let $b_2'$ denote the detecting class.  Then $t\mu
\cdot b_2' \doteq d^{r_2}(a_2)$ detects $v_3^{r(3n-1)} \cdot z_2 = 0$,
hence is a boundary.  The source $a_2$ of this differential is of total
degree $2p^{3n+2} + 2p^{3n+1}$, and by Lemma~\ref{lem:Cpn+1Tatedr}
it has the usual form~$u_{n+1}^\epsilon t^i \cdot \lambda_1^{\epsilon_1}
\lambda_2^{\epsilon_2} \lambda_3^{\epsilon_3}$,
so $a_2 = t^{-p^{3n+1}-p^{3n+2}}$ is the only possibility, with $r_2
\ge 2r(3n+2)$.  It follows as in~(2) that $d^r = 0$ for $2r(3n+1) < r
< 2r(3n+2)$.  After this, $b_2$ lies too close to the horizontal axis
to be a boundary, so $b_2' = b_2$ and $r_2 = 2r(3n+2)$.  It then follows
that $d^{2r(3n+2)}(t^{-p^{3n+1}}) \doteq t^{-p^{3n+1}+p^{3n+2}}
(t\mu)^{r(3n-1)} \lambda_2$.  This establishes the second new even length
differential, and gives the $E^{2r(3n+2)+1}$-term
\begin{multline*}
\hat E^{2r(3n+2)+1}(C_{p^{n+1}}) = E(u_{n+1}) \otimes P(t^{\pm p^{3n+2}})
        \otimes P(t\mu) \otimes E(\lambda_1, \lambda_2, \lambda_3) \\
\oplus \bigoplus_{k=4}^{3n+2} E(u_{n+1}) \otimes P_{r(k-3)}(t\mu)
        \otimes E(\lambda_{[k+1]}, \lambda_{[k+2]})
        \otimes \bF_p\{ t^i \lambda_{[k]} \mid v_p(i) = k-1 \} \,.
\end{multline*}

(4) Carrying on we consider~$z_3$, which is detected by $t^{-p^{3n+2}}
\lambda_3$.  The nonzero product $v_3^{r(3n)-1} \cdot z_3$ is detected by
$b_3 = (t\mu)^{r(3n)-1} \cdot t^{-p^{3n+2}} \lambda_3$, unless this class
is a boundary, in which case the product is detected by another class
in the same total degree as~$b_3$, but of lower filtration.  Let $b_3'$
denote the detecting class.  Then $t\mu \cdot b_3' \doteq d^{r_3}(a_3)$
detects $v_3^{r(3n)}
\cdot z_3 = 0$, and must be a boundary.  The source~$a_3$ of this
differential is of total degree~$2p^{3n+3} + 2p^{3n+2}$, and by
Lemma~\ref{lem:Cpn+1Tatedr} it has the usual form~$u_{n+1}^\epsilon
t^i \cdot \lambda_1^{\epsilon_1} \lambda_2^{\epsilon_2}
\lambda_3^{\epsilon_3}$ not involving~$\mu$.
The only possibility is $a_3 = t^{-p^{3n+2}-p^{3n+3}}$, with $r_3
\ge 2r(3n+3)$.  It follows as above that $d^r = 0$ for $2r(3n+2)
< r < 2r(3n+3)$, after which $b_3$ is in too low a vertical degree
to become a boundary, so $b_3' = b_3$ and $r_3 = 2r(3n+3)$.  Hence
$d^{2r(3n+3)}(t^{-p^{3n+2}}) \doteq t^{-p^{3n+2}+p^{3n+3}} (t\mu)^{r(3n)}
\lambda_3$.  This establishes the third new even length differential,
and leaves the $E^{2r(3n+3)+1}$-term
\begin{multline*}
\hat E^{2r(3n+3)+1}(C_{p^{n+1}}) = E(u_{n+1}) \otimes P(t^{\pm p^{3n+3}})
        \otimes P(t\mu) \otimes E(\lambda_1, \lambda_2, \lambda_3) \\
\oplus \bigoplus_{k=4}^{3n+3} E(u_{n+1}) \otimes P_{r(k-3)}(t\mu)
        \otimes E(\lambda_{[k+1]}, \lambda_{[k+2]})
        \otimes \bF_p\{ t^i \lambda_{[k]} \mid v_p(i) = k-1 \} \,.
\end{multline*}

(5)
Finally, we return to~$z_4$.  Since $b_4 = (t\mu)^{r(3n)} \cdot
t^{-p^{3n+3}}$ is nonzero, in vertical degree~$2r(3n+3) - 2p^3$ of the
$E^{2r(3n+3)+1}$-term, it cannot be a boundary, hence is equal to the
class~$b_4'$ from step~(1).  Thus $t\mu \cdot b_4' = (t\mu)^{r(3n)+1} \cdot
t^{-p^{3n+3}} \doteq d^{r_4}(a_4)$ with $a_4 = u_{n+1} t^{-2p^{3n+3}}$
and $r_4 = 2r(3n+3)+1$.  It follows that $d^{2r(3n+3)+1}(u_{n+1}
t^{-p^{3n+3}}) \doteq (t\mu)^{r(3n)+1}$, since $t^{p^{3n+3}}$ is an
infinite cycle.  This establishes the claimed new odd length differential,
and leaves
\begin{multline*}
\hat E^{2r(3n+3)+2}(C_{p^{n+1}}) = P(t^{\pm p^{3n+3}})
        \otimes P_{r(3n)+1}(t\mu)
	\otimes E(\lambda_1, \lambda_2, \lambda_3) \\
\oplus \bigoplus_{k=4}^{3n+3} E(u_{n+1}) \otimes P_{r(k-3)}(t\mu)
        \otimes E(\lambda_{[k+1]}, \lambda_{[k+2]})
        \otimes \bF_p\{ t^i \lambda_{[k]} \mid v_p(i) = k-1 \} \,.
\end{multline*}
No free summands remain, so by Lemma~\ref{lem:Cpn+1Tatedr} the
remaining differentials are all zero, and this $E^r$-term equals $\hat
E^\infty(C_{p^{n+1}})$.  This completes the $n$-th inductive step.
\end{proof}

\section{The $\bT$-Tate spectral sequence}
\label{sec:TTatespseq}

We can now make the differential structure of the spectral sequences
\begin{align*}
E^2(\bT) = H^{-*}(\bT; V(2)_* THH(BP\<2\>))
	&\Longrightarrow V(2)_* THH(BP\<2\>)^{h\bT} \\
\mu^{-1} E^2(\bT) = H^{-*}(\bT; V(2)_* THH(BP\<2\>)^{tC_p})
	&\Longrightarrow V(2)_* (THH(BP\<2\>)^{tC_p})^{h\bT} \\
\hat E^2(\bT) = \hat H^{-*}(\bT; V(2)_* THH(BP\<2\>))
	&\Longrightarrow V(2)_* THH(BP\<2\>)^{t\bT}
\end{align*}
fully explicit.

\begin{theorem} \label{thm:TTate}
The $\bT$-Tate spectral sequence
\begin{align*}
\hat E^2(\bT) &= \hat H^{-*}(\bT; V(2)_* THH(BP\<2\>)) \\
	&\Longrightarrow V(2)_* THH(BP\<2\>)^{t\bT}
\end{align*}
has $E^2$-term
$$
\hat E^2(\bT) = P(t^{\pm1})
	\otimes P(t\mu) \otimes E(\lambda_1, \lambda_2, \lambda_3) \,.
$$
There are differentials
$$
d^{2r(k)}(t^{p^{k-1} - p^k}) \doteq t^{p^{k-1}} (t\mu)^{r(k-3)} \lambda_{[k]}
$$
for each $k\ge1$.  The classes $t\mu$, $\lambda_1$, $\lambda_2$
and~$\lambda_3$ are permanent cycles.  The $E^\infty$-term is
\begin{multline*}
\hat E^\infty(\bT) =
        P(t\mu) \otimes E(\lambda_1, \lambda_2, \lambda_3) \\
\oplus \bigoplus_{k\ge4} P_{r(k-3)}(t\mu)
        \otimes E(\lambda_{[k+1]}, \lambda_{[k+2]})
        \otimes \bF_p\{ t^i \lambda_{[k]} \mid v_p(i) = k-1 \} \,.
\end{multline*}
\end{theorem}

\begin{proof}
This follows by passage to the limit over~$n$ from Theorem~\ref{thm:CpnTate}.
\end{proof}

\begin{remark}
We saw in Propositions~\ref{prop:locCpHFP} and~\ref{prop:locCpnHFP} that
for each $n\ge1$ some positive power of $\mu \in V(2)_* THH(BP\<2\>)$
lifts to $V(2)_* THH(BP\<2\>)^{hC_{p^n}}$, so that the $\mu$-localized
$C_{p^n}$-homotopy fixed point spectral sequence converges to a
localization $\mu^{-1} V(2)_* THH(BP\<2\>)^{hC_{p^n}}$.  However, no such
power of~$\mu$ lifts to $V(2)_* THH(BP\<2\>)^{h\bT}$, and we therefore
instead express the abutment of the $\mu$-localized $\bT$-homotopy
fixed point spectral sequence in terms of $THH(BP\<2\>)^{tC_p}$, with
$\mu^{-1} V(2)_* THH(BP\<2\>) \cong V(2)_* THH(BP\<2\>)^{tC_p}$ as per
Theorem~\ref{thm:CpTate}.
\end{remark}

\begin{proposition} \label{prop:locTHFP}
The $\mu$-localized $\bT$-homotopy fixed point spectral sequence
\begin{align*}
\mu^{-1} E^2(\bT) &= H^{-*}(\bT; \mu^{-1} V(2)_* THH(BP\<2\>)) \\
	&\Longrightarrow V(2)_* (THH(BP\<2\>)^{tC_p})^{h\bT}
\end{align*}
has $E^2$-term
$$
\mu^{-1} E^2(\bT) = P(t\mu)
	\otimes E(\lambda_1, \lambda_2, \lambda_3) \otimes P(\mu^{\pm1}) \,.
$$
There are differentials
$$
d^{2r(k)}(\mu^{p^{k-1}})
	\doteq (t\mu)^{r(k)} \lambda_{[k]} \mu^{p^{k-1}-p^k}
$$
for each $k\ge1$.  The classes $t\mu$, $\lambda_1$, $\lambda_2$
and~$\lambda_3$ are permanent cycles.  The $E^\infty$-term is
\begin{multline*}
\mu^{-1} E^\infty(\bT) = P(t\mu)
        \otimes E(\lambda_1, \lambda_2, \lambda_3) \\
\oplus \bigoplus_{k\ge1} P_{r(k)}(t\mu)
        \otimes E(\lambda_{[k+1]}, \lambda_{[k+2]})
        \otimes \bF_p\{ \lambda_{[k]} \mu^j \mid v_p(j) = k-1\} \,.
\end{multline*}
\end{proposition}

\begin{proof}
This follows by passage to the limit over~$n$ from
Proposition~\ref{prop:locCpnHFP}.
\end{proof}

\begin{proposition} \label{prop:THFP}
The $\bT$-homotopy fixed point spectral sequence
\begin{align*}
E^2(\bT) &= H^{-*}(\bT; V(2)_* THH(BP\<2\>)) \\
        &\Longrightarrow V(2)_* THH(BP\<2\>)^{h\bT}
\end{align*}
has $E^2$-term
$$
E^2(\bT) = P(t)
	\otimes E(\lambda_1, \lambda_2, \lambda_3) \otimes P(\mu) \,.
$$
For each $k\ge1$ there are differentials
$$
\begin{cases}
d^{2r(k)}(\mu^{dp^{k-1}})
        \doteq (t\mu)^{r(k)} \lambda_{[k]} \mu^{(d-p)p^{k-1}}
	&\text{for $d>p$ with $p \nmid d$,} \\
\\
d^{2r(k)}(\mu^{(p-d)p^{k-1}})
        \doteq t^{dp^{k-1}} (t\mu)^{r(k)-dp^{k-1}} \lambda_{[k]}
	&\text{for $0<d<p$, and} \\
\\
d^{2r(k)}(t^{dp^{k-1}})
	\doteq t^{dp^{k-1}+p^k} (t\mu)^{r(k-3)} \lambda_{[k]}
	&\text{for $d>0$ with $p \nmid d$.}
\end{cases}
$$
The classes $t\mu$, $\lambda_1$, $\lambda_2$
and~$\lambda_3$ are permanent cycles.  The $E^\infty$-term is
\begin{align*}
E^\infty(\bT) &= P(t\mu)
        \otimes E(\lambda_1, \lambda_2, \lambda_3) \\
&\qquad\oplus \bigoplus_{k\ge1} P_{r(k)}(t\mu)
        \otimes E(\lambda_{[k+1]}, \lambda_{[k+2]})
        \otimes \bF_p\{\lambda_{[k]} \mu^{dp^{k-1}} \mid p \nmid d > 0\} \\
&\qquad\oplus \bigoplus_{k\ge1} P_{r(k)-dp^{k-1}}(t\mu)
        \otimes E(\lambda_{[k+1]}, \lambda_{[k+2]})
	\otimes \bF_p\{t^{dp^{k-1}} \lambda_{[k]} \mid 0<d<p\} \\
&\qquad\oplus \bigoplus_{k\ge4} P_{r(k-3)}(t\mu)
        \otimes E(\lambda_{[k+1]}, \lambda_{[k+2]})
	\otimes \bF_p\{t^{dp^{k-1}} \lambda_{[k]} \mid p \nmid d > p\} \,.
\end{align*}
\end{proposition}

\begin{proof}
The differentials on $t^{dp^{k-1}}$ for $p \nmid d > 0$ follow as in
Theorem~\ref{thm:TTate}, while those on $\mu^{dp^{k-1}}$ for $p \nmid d >
p$ are as in Proposition~\ref{prop:locTHFP}.  For $0<d<p$ we also
have
$$
d^{2r(k)}(\mu^{dp^{k-1}}) \doteq
(t\mu)^{r(k)} \lambda_{[k]} \mu^{dp^{k-1}-p^k}
= t^{p^k-dp^{k-1}} (t\mu)^{r(k)+dp^{k-1}-p^k} \lambda_{[k]} \,.
$$
Replacing $d$ by $p-d$ we obtain the claimed formula.

For each $k\ge1$ and $p \nmid d$ the $d^{2r(k)}$-differential maps
the summand
$$
E(\lambda_{[k+1]}, \lambda_{[k+2]})
	\otimes \bF_p\{t^i \mu^j \mid i-j=dp^{k-1}-p^k\}
$$
of $E^{2r(k)}(\bT)$ injectively to the summand
$$
E(\lambda_{[k+1]}, \lambda_{[k+2]})
	\otimes \bF_p\{t^i \lambda_{[k]} \mu^j \mid i-j=dp^{k-1}\} \,,
$$
with cokernel one of the displayed summands in $E^\infty(\bT)$.
Here $i\ge0$ and $j\ge0$ in each case.
\end{proof}

Following the referee's good advice, we decompose these $E^\infty$-terms
as in the next three definitions.

\begin{definition} \label{def:ABCD}
Let $A = P(t\mu) \otimes E(\lambda_1, \lambda_2, \lambda_3)$, viewed as
a subalgebra of $E^\infty(\bT)$.  For $k\ge1$ and $0<d<p$ let
$$
C(k,d) = P_{r(k)-dp^{k-1}}(t\mu)
        \otimes E(\lambda_{[k+1]}, \lambda_{[k+2]})
        \otimes \bF_p\{t^{dp^{k-1}} \lambda_{[k]}\}
$$
be the finite $A$-submodule of $E^\infty(\bT)$ generated by
$$
c_{k,d} = t^{dp^{k-1}} \lambda_{[k]} \,.
$$
The class
$$
x_{k,d} = (t\mu)^{\frac{d}{p} r(k-3)} \cdot c_{k,d}
	= t^{\frac{d}{p} r(k)} \lambda_{[k]} \mu^{\frac{d}{p} r(k-3)}
$$
lies in $C(k,d)$, is nonzero since $\frac{d}{p} r(k-3) < r(k) -
dp^{k-1}$, and has total degree $|x_{k,d}| = 2p^{[k]} - 2dp^{[k]-1} - 1$.
In particular,
$$
x_{1,d} = c_{1,d} = t^d \lambda_1 \ ,\ 
x_{2,d} = c_{2,d} = t^{dp} \lambda_2  \ ,\ 
x_{3,d} = c_{3,d} = t^{dp^2} \lambda_3
$$
for all $0<d<p$.  Let $C = \bigoplus_{k\ge1, 0<d<p} C(k,d)$, and
let
\begin{align*}
B &= \bigoplus_{k\ge1} P_{r(k)}(t\mu)
        \otimes E(\lambda_{[k+1]}, \lambda_{[k+2]})
        \otimes \bF_p\{\lambda_{[k]} \mu^{dp^{k-1}} \mid p \nmid d>0\} \\
D &= \bigoplus_{k\ge4} P_{r(k-3)}(t\mu)
        \otimes E(\lambda_{[k+1]}, \lambda_{[k+2]})
        \otimes \bF_p\{t^{dp^{k-1}} \lambda_{[k]} \mid p \nmid d>p\}
\end{align*}
be the indicated $A$-submodules of $E^\infty(\bT)$, concentrated in
positive and negative total degrees, respectively.  Then $E^\infty(\bT)
= A \oplus B \oplus C \oplus D$.
\end{definition}

It should be clear from the context whether~$B$ refers to this summand
in~$E^\infty(\bT)$ or a generic~$S$-algebra.  The classes $x_{k,d}$ are
the ones mentioned in the introduction.  Their role, together with the
classes $z_{k,d}$ defined just below, will only become apparent starting
with Corollaries~\ref{cor:xikd123} and~\ref{cor:xikdge4}.

\begin{definition} \label{def:A'B'C'D'}
Let $A' = P(t\mu) \otimes E(\lambda_1, \lambda_2, \lambda_3)$ as a
subalgebra of~$\mu^{-1} E^\infty(\bT)$.  For $k\ge1$ and
$0<d<p$ let
$$
C'(k,d) = P_{r(k)}(t\mu) \otimes E(\lambda_{[k+1]}, \lambda_{[k+2]})
        \otimes \bF_p\{\lambda_{[k]} \mu^{-dp^{k-1}}\}
$$
be the finite $A'$-submodule of $\mu^{-1} E^\infty(\bT)$ generated by
$$
c'_{k,d} = \lambda_{[k]} \mu^{-dp^{k-1}} \,.
$$
The class
$$
z_{k,d} = (t\mu)^{\frac{d}{p} r(k)} \cdot c'_{k,d}
	= t^{\frac{d}{p} r(k)} \lambda_{[k]} \mu^{\frac{d}{p} r(k-3)}
$$
lies in $C'(k,d)$, is nonzero since
$\frac{d}{p} r(k) < r(k)$, and has total degree
$|z_{k,d}| = 2p^{[k]} - 2dp^{[k]-1} - 1$.
Let $C' = \bigoplus_{k\ge1, 0<d<p} C'(k,d)$, and
let
\begin{align*}
B' &= \bigoplus_{k\ge1} P_{r(k)}(t\mu)
        \otimes E(\lambda_{[k+1]}, \lambda_{[k+2]})
        \otimes \bF_p\{\lambda_{[k]} \mu^{dp^{k-1}} \mid p \nmid d>0\} \\
D' &= \bigoplus_{k\ge1} P_{r(k)}(t\mu)
        \otimes E(\lambda_{[k+1]}, \lambda_{[k+2]})
        \otimes \bF_p\{\lambda_{[k]} \mu^{-dp^{k-1}} \mid p \nmid d>p\}
\end{align*}
be the indicated $A'$-submodules of $\mu^{-1} E^\infty(\bT)$, concentrated
in positive and negative total degrees, respectively.  Then $\mu^{-1}
E^\infty(\bT) = A' \oplus B' \oplus C' \oplus D'$.
\end{definition}

\begin{definition}
Let
$\hat A = P(t\mu) \otimes E(\lambda_1, \lambda_2, \lambda_3)$,
$$
\hat C(k,d) = P_{r(k-3)}(t\mu)
        \otimes E(\lambda_{[k+1]}, \lambda_{[k+2]})
        \otimes \bF_p\{t^{dp^{k-1}} \lambda_{[k]}\}
$$
for $k\ge1$ and $0<d<p$, and
\begin{align*}
\hat B &= \bigoplus_{k\ge4} P_{r(k-3)}(t\mu)
	\otimes E(\lambda_{[k+1]}, \lambda_{[k+2]})
	\otimes \bF_p\{t^{-dp^{k-1}} \lambda_{[k]} \mid p \nmid d>0\} \\
\hat C &= \bigoplus_{k\ge4, 0<d<p} \hat C(k,d) \\
\hat D &= \bigoplus_{k\ge4} P_{r(k-3)}(t\mu)
	\otimes E(\lambda_{[k+1]}, \lambda_{[k+2]})
	\otimes \bF_p\{t^{dp^{k-1}} \lambda_{[k]} \mid p \nmid d>p\} \,.
\end{align*}
Then $\hat E^\infty(\bT) = \hat A \oplus \hat B \oplus \hat C \oplus
\hat D$.  Note that $\hat C(k,d) = 0$ for $k \in \{1,2,3\}$.
\end{definition}




The $\bT$-equivariant comparison map
$$
\hat\Gamma_1 \: THH(BP\<2\>) \longto THH(BP\<2\>)^{tC_2}
$$
(renamed the $p$-cyclotomic structure map~$\varphi_p$ in~\cite{NS18})
induces a morphism of $\bT$-homotopy fixed point spectral sequences,
given at the $E^2$-term by the homomorphism
\begin{multline*}
E^2(\hat\Gamma_1^{h\bT}) \: E^2(\bT) = P(t) \otimes
	E(\lambda_1, \lambda_2, \lambda_3) \otimes P(\mu) \\
\longto P(t) \otimes E(\lambda_1, \lambda_2, \lambda_3) \otimes
        P(\mu^{\pm1}) = \mu^{-1} E^2(\bT)
\end{multline*}
that inverts~$\mu$.  At the $E^\infty$-terms we have the following
formulas.

\begin{lemma} \label{lem:EinftyhatGamma1hT}
The homomorphism
$$
E^\infty(\hat\Gamma_1^{h\bT}) \: E^\infty(\bT) \longto \mu^{-1} E^\infty(\bT)
$$
maps
\begin{enumerate}
\item
$A$ isomorphically to~$A'$,
\item
$B$ isomorphically to $B'$,
\item
$C$ injectively to $C'$, and
\item
$D$ to zero.
\end{enumerate}
Specifically, $E^\infty(\hat\Gamma_1^{h\bT})$ is injective in total
degrees $* \ge -2p^3+2p^2$, and bijective in total degrees $* \ge
2p^2+2p-2$.
\end{lemma}

\begin{proof}
Cases~(1) and~(2) are clear.  In~(3), the injection $C(k,d) \to C'(k,d)$
takes $c_{k,d} = t^{dp^{k-1}} \lambda_{[k]}$ to $(t\mu)^{dp^{k-1}} \cdot
c'_{k,d}$, which is annihilated by the same $t\mu$-power as~$c_{k,d}$,
namely $(t\mu)^{r(k)-dp^{k-1}}$.  In~(4), the image of $D$ in $D'$
is zero since $t^{dp^{k-1}} \lambda_{[k]}$ maps to $(t\mu)^{dp^{k-1}}
\cdot \lambda_{[k]} \mu^{-dp^{k-1}}$, which is zero because $dp^{k-1}
\ge r(k)$ for $d>p$.

The highest degree element in the kernel of
$E^\infty(\hat\Gamma_1^{h\bT})$ is $t^{(p+1)p^3} (t\mu)^{p-1} \lambda_1
\lambda_2 \lambda_3$ in~$D$ in total degree~$-2p^3+2p^2-1$, mapping to
$d^{2r(4)}(t^{p^3-1} \lambda_2 \lambda_3 \mu^{-1})$.  The highest degree
element not in the image of $E^\infty(\hat\Gamma_1^{h\bT})$ is $\lambda_1
\lambda_2 \lambda_3 \mu^{-1}$ in $C'(1,1)$, in total degree~$2p^2+2p-3$.
\end{proof}

Similarly, the homotopy restriction map
$$
R^h \: THH(BP\<2\>)^{h\bT} \longto THH(BP\<2\>)^{t\bT}
$$
(renamed the canonical map in~\cite{NS18}) induces a morphism of spectral
sequences, given at the $E^2$-term by the homomorphism
\begin{multline*}
E^2(R^h) \: E^2(\bT) = P(t) \otimes
	E(\lambda_1, \lambda_2, \lambda_3) \otimes P(\mu) \\
\longto P(t^{\pm1}) \otimes E(\lambda_1, \lambda_2, \lambda_3) \otimes
        P(\mu) = \hat E^2(\bT)
\end{multline*}
that inverts~$t$.  The following lemma is similar
to~\cite{AR02}*{Prop.~7.2}.

\begin{lemma} \label{lem:EinftyRh}
The homomorphism
$$
E^\infty(R^h) \: E^\infty(\bT) \longto \hat E^\infty(\bT)
$$
maps
\begin{enumerate}
\item
$A$ isomorphically to~$\hat A$,
\item
$B$ to zero,
\item
$C$ surjectively to $\hat C$, and
\item
$D$ isomorphically to $\hat D$.
\end{enumerate}
Specifically, $E^\infty(R^h)$ is surjective in total degrees
$* \le 2p^3+2p-2$, and bijective in total degrees $* \le 0$.
\end{lemma}

\begin{proof}
Cases~(1) and~(4) are clear.  In~(2), the image of $B$ in
$\hat B$ is zero, since $\lambda_{[k]} \mu^{dp^{k-1}}$ maps
to $(t\mu)^{dp^{k-1}} \cdot t^{-dp^{k-1}} \lambda_{[k]}$, which is
zero because $dp^{k-1} \ge r(k-3)$ for $d>0$.  In~(3), the surjection
$C(k,d) \to \hat C(k,d)$ takes $c_{k,d} = t^{dp^{k-1}} \lambda_{[k]}$ to
$t^{dp^{k-1}} \lambda_{[k]}$, which is annihilated by a lower $t\mu$-power
than $c_{k,d}$, since $r(k-3) < r(k) - dp^{k-1}$ for $0<d<p$.

The lowest degree element not in the image of $E^\infty(R^h)$ is $t^{-p^3}
\lambda_1$ in $\hat B$, in total degree $2p^3+2p-1$.  The lowest degree
element in the kernel of $E^\infty(R^h)$ is $t^{p-1} \lambda_1$ in
$C(1,p-1)$ in total degree~$1$, mapping to $d^{2p}(t^{-1})$.
\end{proof}

\section{Topological cyclic homology and algebraic $K$-theory}
\label{sec:TC}

We now pursue the calculational strategy employed in~\cite{BM94},
\cite{BM95}, \cite{HM97}, \cite{Rog99}, \cite{AR02}, \cite{Aus10}
and~\cite{AR12} to identify $TC(B)$ with the homotopy equalizer
of the two maps $GR^h$ and $\hat\Gamma_1^{h\bT}$ displayed below.
$$
\xymatrix{
TC(B) \ar[r]^-{\pi} & THH(B)^{h\bT} \ar[r]^-{R^h}
	\ar[dr]_-{\hat\Gamma_1^{h\bT}}
	& THH(B)^{t\bT} \ar[d]^-{G}_-{\simeq} \\
& & (THH(B)^{tC_p})^{h\bT}
}
$$
In these papers, this identification was only known to be valid
in $V$-homotopy in a range of sufficiently high degrees, for
suitable finite spectra~$V$.  However, with the work of Nikolaus
and Scholze~\cite{NS18}*{Rmk.~1.6}, we now know that $TC(B)$
is given by the homotopy equalizer above in all degrees, whenever
$THH(B)$ is bounded below.  (This certainly holds for all connective
$S$-algebras~$B$.)  Let $GR^h_* = V_*(GR^h)$ and $\hat\Gamma_{1*}^{h\bT}
= V_*(\hat\Gamma_1^{h\bT})$.  The associated long exact sequence
$$
\dots \overset{\partial}\longto V_* TC(B)
	\overset{\pi}\longto V_* THH(B)^{h\bT}
	\overset{GR^h_* - \hat\Gamma_{1*}^{h\bT}}\longto
	V_* (THH(B)^{tC_p})^{h\bT}
	\overset{\partial}\longto \dots
$$
leads to the short exact sequence
$$
0 \to \Sigma^{-1} \cok(GR^h_* - \hat\Gamma_{1*}^{h\bT})
	\overset{\partial}\longto
	V_* TC(B)
	\overset{\pi}\longto
	\ker(GR^h_* - \hat\Gamma_{1*}^{h\bT}) \to 0 \,.
$$

In our case, the task is to calculate the kernel and cokernel of
$GR^h_* - \hat\Gamma_{1*}^{h\bT}$ for $B = BP\<2\>$ and $V = V(2)$,
and thereby to determine $V(2)_* TC(BP\<2\>)$.  We studied the
effect of $\hat\Gamma_1^{h\bT}$ and $R^h$ at the level of spectral
sequence $E^\infty$-terms in Lemmas~\ref{lem:EinftyhatGamma1hT}
and~\ref{lem:EinftyRh}.  In Proposition~\ref{prop:V2G} we do something
similar for~$G$.  Thereafter we find lifts $\wA$, $\wB$, $\wC$ and~$\wD$
in $V(2)_* THH(BP\<2\>)^{h\bT}$ of the summands $A$, $B$, $C$ and~$D$
of $E^\infty(\bT)$ from Definition~\ref{def:ABCD}, and compute the effect
of $GR^h_* - \hat\Gamma_{1*}^{h\bT}$ acting upon these lifts.

\begin{proposition} \label{prop:V2G}
The isomorphism
$$
G_* = V(2)_*(G) \: V(2)_* THH(BP\<2\>)^{t\bT}
	\overset{\cong}\longto V(2)_* (THH(BP\<2\>)^{tC_p})^{h\bT}
$$
takes each class
$$
\eta \in \{t^{p^3 i} (t\mu)^m \lambda_1^{\epsilon_1}
	\lambda_2^{\epsilon_2} \lambda_3^{\epsilon_3}\}
$$
detected by $y = t^{p^3 i} (t\mu)^m \lambda_1^{\epsilon_1}
\lambda_2^{\epsilon_2} \lambda_3^{\epsilon_3} \in \hat E^\infty(\bT)$
to a class
$$
G_*(\eta) \in \{(t\mu)^m \lambda_1^{\epsilon_1} \lambda_2^{\epsilon_2}
	\lambda_3^{\epsilon_3} \mu^{-i}\}
$$
detected by $z = (t\mu)^m \lambda_1^{\epsilon_1} \lambda_2^{\epsilon_2}
\lambda_3^{\epsilon_3} \mu^{-i} \in \mu^{-1} E^\infty(\bT)$ (up to a unit
in~$\bF_p$, which we will suppress).  Conversely, its inverse $G_*^{-1}$
takes each class
$$
\zeta \in \{(t\mu)^m \lambda_1^{\epsilon_1} \lambda_2^{\epsilon_2}
	\lambda_3^{\epsilon_3} \mu^j\}
$$
to a class
$$
G_*^{-1}(\zeta) \in \{t^{-p^3 j} (t\mu)^m \lambda_1^{\epsilon_1}
	\lambda_2^{\epsilon_2} \lambda_3^{\epsilon_3}\}
$$
(again, up to a unit in $\bF_p$, which we suppress).
\end{proposition}

\begin{proof}
We first handle the case $m=0$, using the commutative diagram
$$
\xymatrix{
THH(B)^{t\bT} \ar[r]^-{F^t} \ar[d]_G^{\simeq}
	& THH(B)^{tC_{p^{n+1}}} \ar[d]_{\simeq}^{G_n} \\
(THH(B)^{tC_p})^{h\bT} \ar[r]^-{F^h}
	& (THH(B)^{tC_p})^{hC_{p^n}}
}
$$
in the special case of $B = BP\<2\>$ and $n=0$.  It is constructed by
viewing the $\bT/C_p$-equivariant $C_p$-fixed point spectrum
$$
X = [\widetilde{E\bT} \wedge F(E\bT_+, THH(B))]^{C_p} \simeq THH(B)^{tC_p}
$$
as a $\bT$-spectrum via the $p$-th root isomorphism $\rho \: \bT
\cong \bT/C_p$.  The comparison map~$G \: X^{\bT} \to X^{h\bT}$ is then
compatible with the comparison map~$G_n \: X^{C_{p^n}} \to X^{hC_{p^n}}$,
via the group restriction maps along $C_{p^n} \subset \bT$.

In the case $n=0$, the group restriction map~$F^t$ induces a morphism
of spectral sequences given at the $E^2$-terms by the inclusion
\begin{multline*}
E^2(F^t) \: \hat E^2(\bT) = P(t^{\pm1}) \otimes E(\lambda_1, \lambda_2,
	\lambda_3) \otimes P(\mu) \\
\longto E(u_1) \otimes P(t^{\pm1}) \otimes E(\lambda_1, \lambda_2,
        \lambda_3) \otimes P(\mu) = \hat E^2(C_p) \,.
\end{multline*}
Hence each class $\eta \in V(2)_* THH(BP\<2\>)^{t\bT}$ detected
by $t^{p^3i} \lambda_1^{\epsilon_1} \lambda_2^{\epsilon_2}
\lambda_3^{\epsilon_3} \ne 0$ in $\hat E^\infty(\bT)$ maps to a
class $F^t_*(\eta)$ detected by $t^{p^3i} \lambda_1^{\epsilon_1}
\lambda_2^{\epsilon_2} \lambda_3^{\epsilon_3}$ in $\hat E^\infty(C_p)
= P(t^{\pm p^3}) \otimes E(\lambda_1, \lambda_2,\lambda_3)$,
which remains nonzero there.  It follows from Theorem~\ref{thm:CpTate}
that $(G_0 F^t)_*(\eta) = \lambda_1^{\epsilon_1} \lambda_2^{\epsilon_2}
\lambda_3^{\epsilon_3} \mu^{-i}$ in $V(2)_* THH(BP\<2\>)^{tC_p}$ up to a
unit in $\bF_p$, which we suppress.  This equals $(F^h G)_*(\eta)$, where
the group restriction map~$F^h$ for $n=0$ induces the edge homomorphism
$$
E^\infty(F^h) \: \mu^{-1} E^\infty(\bT) \longto
	E(\lambda_1, \lambda_2, \lambda_3) \otimes P(\mu^{\pm1}) \,.
$$
Hence $G_*(\eta)$ must be detected in $\mu^{-1} E^\infty(\bT)$ by a
class~$z$ mapping to $\lambda_1^{\epsilon_1} \lambda_2^{\epsilon_2}
\lambda_3^{\epsilon_3} \mu^{-i}$ under the edge homomorphism, and the only
possibility is that $z = \lambda_1^{\epsilon_1} \lambda_2^{\epsilon_2}
\lambda_3^{\epsilon_3} \mu^{-i}$, in filtration degree~zero.

For $m\ge1$, each class $\eta$ detected by $y = t^{p^3i} (t\mu)^m
\lambda_1^{\epsilon_1} \lambda_2^{\epsilon_2} \lambda_3^{\epsilon_3} \ne
0$ in $\hat E^\infty(\bT)$ is of the form $\eta = v_3^m \cdot \eta_0$,
with $\eta_0$ detected by $y_0 = t^{p^3i} \lambda_1^{\epsilon_1}
\lambda_2^{\epsilon_2} \lambda_3^{\epsilon_3} \in \hat E^\infty(\bT)$.
To see this, note that each element of $\hat E^\infty(\bT)$
in the same total degree as~$y$, but of lower filtration, is a
$(t\mu)^m$-multiple.  This follows from the case enumeration in
the proof of Lemma~\ref{lem:Cp2Tatedr}.  By the first part of the
proof, $G_*(\eta_0)$ is detected by $z_0 = \lambda_1^{\epsilon_1}
\lambda_2^{\epsilon_2} \lambda_3^{\epsilon_3} \mu^{-i}$.  Hence $G_*(\eta)
= v_3^m \cdot G_*(\eta_0)$ is detected by $(t\mu)^m \cdot z_0 = z$,
since this product is nonzero.

For the converse, consider any class $\zeta$ detected by $z = (t\mu)^m
\lambda_1^{\epsilon_1} \lambda_2^{\epsilon_2} \lambda_3^{\epsilon_3}
\mu^j \ne 0$ in $\mu^{-1} E^\infty(\bT)$.  Then $\eta = G_*^{-1}(\zeta)$
must be detected by some monomial~$y$ in $E^\infty(\bT)$, and $G_*(\eta)
= \zeta$ is detected by~$z$.  By the first part of the proposition,
this monomial must be $y = t^{-p^3 j} (t\mu)^m \lambda_1^{\epsilon_1}
\lambda_2^{\epsilon_2} \lambda_3^{\epsilon_3}$.
\end{proof}

Recall $\lambda_1^K$, $\lambda_2^K$ and~$\lambda_3^K$ from
Definitions~\ref{def:lambda1K}, \ref{def:lambda2K} and~\ref{def:lambda3K}.

\begin{definition} \label{def:liftA}
Let
$$
\wA = P(v_3) \otimes E(\lambda_1, \lambda_2, \lambda_3)
	\subset V(2)_* THH(BP\<2\>)^{h\bT}
$$
be the subalgebra generated by the images of $v_3 \in \pi_* V(2)$ and
$i_2 i_1 i_0(\lambda_1^K)$, $i_2 i_1(\lambda_2^K)$, $i_2(\lambda_3^K)
\in V(2)_* K(BP\<2\>)$ under the composites
$$
S \longto K(BP\<2\>) \overset{trc}\longto TC(BP\<2\>)
	\overset{\pi}\longto THH(BP\<2\>)^{h\bT} \,,
$$
where $trc$ denotes the cyclotomic trace map~\cite{BHM93}.
The homomorphisms $GR^h_*$ and $\hat\Gamma_{1*}^{h\bT}$
agree on these classes, and we let
$$
\wA' = P(v_3) \otimes E(\lambda_1, \lambda_2, \lambda_3)
        \subset V(2)_* (THH(BP\<2\>)^{tC_p})^{h\bT}
$$
be the subalgebra generated by the images of $v_3$, $\lambda_1$,
$\lambda_2$ and~$\lambda_3$, under either one of these homomorphisms.
\end{definition}

The subalgebras~$\wA$ and~$\wA'$ are lifts to $V(2)$-homotopy of
the subalgebras $A \subset E^\infty(\bT)$ and~$A' \subset \mu^{-1}
E^\infty(\bT)$, respectively.  To choose good lifts~$\wC(k,d)$
and~$\wC'(k,d)$ in $V(2)$-homotopy of the summands $C(k,d)$ and $C'(k,d)$
we make use of the norm--restriction homotopy cofiber sequence
\begin{multline*}
\Sigma THH(BP\<2\>)_{h\bT} \overset{N^h}\longto
THH(BP\<2\>)^{h\bT} \overset{R^h}\longto
THH(BP\<2\>)^{t\bT} \\ \overset{\partial^h}\longto
\Sigma^2 THH(BP\<2\>)_{h\bT}
\end{multline*}
and the associated long exact sequence.  The $\bT$-Tate spectral sequence
maps to a horizontally shifted $\bT$-homotopy orbit spectral sequence
\begin{equation} \label{eq:horbitspseq}
\begin{aligned}
E^2_{*,*} &= H_{*-2}(\bT; V(2)_* THH(BP\<2\>)) \\
	&= \bF_p\{ t^i \mid i<0 \}
	\otimes E(\lambda_1, \lambda_2, \lambda_3) \otimes P(\mu) \\
	&\Longrightarrow V(2)_* \Sigma^2 THH(BP\<2\>)_{h\bT} \,,
\end{aligned}
\end{equation}
concentrated in filtrations $s\ge2$ of the first quadrant.

The $\bT$-Tate differentials crossing the vertical line $s=1$ are
closely related to the homotopy norm map $N^h_* = V(2)_*(N^h)$,
cf.~\cite{BM94}*{Thm.~2.15}.  Let $R^h_* = V(2)_*(R^h)$, so that
$\im(N^h_*) = \ker(R^h_*)$ by exactness.  The following two lemmas spell
out some upper bounds for $\ker E^\infty(R^h)$.

\begin{lemma} \label{lem:kerEinftyRh}
In the $\bT$-Tate spectral sequence $(\hat E^r(\bT), d^r)$, the
nonzero differentials from total degrees $* < 2p^3$ that cross the line $s=1$
are of the form
\begin{align*}
d^{2p}(t^{d-p} \lambda_2^{\epsilon_2})
  &\doteq t^d \lambda_1 \lambda_2^{\epsilon_2} \\
d^{2p^2}(t^{dp-p^2} \lambda_1^{\epsilon_1})
  &\doteq t^{dp} \lambda_1^{\epsilon_1} \lambda_2 \\
d^{2p^3}(t^{dp^2-p^3} \lambda_1^{\epsilon_1} \lambda_2^{\epsilon_2})
  &\doteq t^{dp^2} \lambda_1^{\epsilon_1} \lambda_2^{\epsilon_2} \lambda_3 \,,
\end{align*}
for suitable $0<d<p$ and $\epsilon_1, \epsilon_2 \in \{0,1\}$.  Hence,
in total degrees $* \le 2p^3-2$ the classes on the right-hand side
generate $\ker E^\infty(R^h)$.  These lie in filtrations $-2(p^3-p^2)
\le s \le -2$, and there is at most one class in each total degree $*
\le 2p^3-2$.
\end{lemma}

\begin{proof}
The restriction to total degrees $* < 2p^3$ means we only have to
consider differentials on the classes $t^i$ for $-p^3 < i < 0$, and their
$\lambda_1$- and $\lambda_2$-multiples.  The $d^{2p}$-differentials
only cross $s=1$ for $-p < i < 0$, The $d^{2p^2}$-differentials
are defined for $p \mid i$ and only cross $s=1$ when $-p^2 < i < 0$.
The $d^{2p^3}$-differentials are defined for $p^2 \mid i$, and cross $s=1$
whenever $-p^3 < i < 0$.  An explicit enumeration shows that each total
degree in the range $1 \le * \le 2p^3-2$ occurs at most once.
\end{proof}

\begin{lemma} \label{lem:kerEinftyRhbis}
In the $\bT$-Tate spectral sequence, the nonzero $d^r$-differentials from
total degrees $* < 4p^3-1$ that cross the line $s=1$ are of the form
\begin{align*}
d^{2p}(t^{d-p} (t\mu)^m \lambda_2^{\epsilon_2} \lambda_3^{\epsilon_3})
  &\doteq t^d (t\mu)^m \lambda_1 \lambda_2^{\epsilon_2} \lambda_3^{\epsilon_3} \\
d^{2p^2}(t^{dp-p^2} (t\mu)^m \lambda_1^{\epsilon_1} \lambda_3^{\epsilon_3})
  &\doteq t^{dp} (t\mu)^m \lambda_1^{\epsilon_1} \lambda_2 \lambda_3^{\epsilon_3} \\
d^{2p^3}(t^{dp^2-p^3} (t\mu)^m \lambda_1^{\epsilon_1} \lambda_2^{\epsilon_2})
  &\doteq t^{dp^2} (t\mu)^m \lambda_1^{\epsilon_1} \lambda_2^{\epsilon_2} \lambda_3 \\
d^{2p^4+2p}(t^{-p^3} \lambda_2^{\epsilon_2})
  &\doteq t^{p^4-p^3} (t\mu)^p \lambda_1 \lambda_2^{\epsilon_2} \,.
\end{align*}
for suitable $m, \epsilon_1, \epsilon_2, \epsilon_3 \in \{0,1\}$, with
$m + \epsilon_3 \le 1$.  In the $d^{2p}$-case with $m=1$ we have $d=-1$
or $0 < d < p-1$, while in the remaining $d^{2p}$-, $d^{2p^2}$- and
$d^{2p^3}$-cases we have $0 < d < p$.  Hence, in total degrees $* \le
4p^3-3$ the classes on the right-hand side generate $\ker E^\infty(R^h)$.
These lie in filtrations $-2(p^3-p^2+1) \le s \le 0$, except for the last
two classes $t^{p^4-p^3} (t\mu)^p \lambda_1 \lambda_2^{\epsilon_2}$,
which lie in filtration $-2(p^4-p^3+p)$ and total degrees $2p^3 - 1 +
\epsilon_2 (2p^2-1)$.
\end{lemma}

\begin{proof}
The restriction to total degrees~$* < 4p^3 - 1$ means that we only have to
consider differentials on the classes $t^i$ for $-p^3 \le i < 0$, and some
of their $t\mu$-, $\lambda_1$-, $\lambda_2$- and~$\lambda_3$-multiples
(without repeated factors).  The resulting right-hand classes have the
form $t^i y$ in Tate filtration $s = -2i$, where $0 \le i \le p^3-p^2+1$
except in the last two cases.
\end{proof}

Recall $c_{k,d}$ and~$c'_{k,d}$ from Definitions~\ref{def:ABCD}
and~\ref{def:A'B'C'D'}.

\begin{proposition} \label{prop:gammakd123}
For each $k \in \{1,2,3\}$ and $0<d<p$ there is a unique element
$$
\gamma_{k,d} \in \{c_{k,d}\} \subset V(2)_* THH(BP\<2\>)^{h\bT}
$$
that satisfies
$$
R^h_*(\gamma_{k,d}) = 0 \,.
$$
Moreover, $\lambda_k \cdot \gamma_{k,d} = 0$ and $v_3^{p^k-dp^{k-1}}
\cdot \gamma_{k,d} = 0$.
\end{proposition}

\begin{proof}
The tower of spectra inducing the $\bT$-homotopy fixed point spectral
sequence is obtained by restricting the tower inducing the $\bT$-Tate
spectral sequence to filtrations $s\le0$.  Hence each nonzero class $x \in
\ker E^\infty(R^h) \subset E^\infty(\bT)$ can be represented by an element
$\xi \in \ker(R^h_*) \subset V(2)_* THH(BP\<2\>)^{h\bT}$, in the sense
that $\xi \in \{x\}$.  (See~\cite{BM94}*{p.~75} and~\cite{AR02}*{Lem.~7.3}.)
Furthermore, for $x$ in total degree $* \le 2p^3-2$ the element $\xi$
is unique.  To see this, suppose that $\xi' \in \{x\}$ is also in
$\ker(R^h_*)$.  Then $\xi'-\xi$ in $\ker(R^h_*)$ must be detected by
a class $x'$ in $\ker E^\infty(R^h)$, in the same total degree as~$x$,
but in lower filtration.  As noted in Lemma~\ref{lem:kerEinftyRh},
there are no nonzero such $x'$, so $\xi' = \xi$.

In particular, for $k \in \{1,2,3\}$ and $0<d<p$ this applies to the
classes $c_{k,d} = t^{dp^{k-1}} \lambda_k$ in total degrees $1 \le
2p^k-2dp^{k-1}-1 \le 2p^3-2p^2-1$, and uniquely defines the homotopy
elements $\gamma_{k,d}$.

By exactness, we can write $\gamma_{k,d} = N^h_*(\theta_{k,d})$ with
$$
\theta_{k,d} \in V(2)_* \Sigma^2 THH(BP\<2\>)_{h\bT}
$$
in degree~$2p^k - 2dp^{k-1}$.  In fact, $\theta_{k,d} \in
\{t^{dp^{k-1}-p^k}\}$, up to a unit multiple, but we only need to
know that $\theta_{k,d}$ must be detected in filtration $s \le
2p^k - 2dp^{k-1}$ in the shifted $\bT$-homotopy orbit spectral
sequence~\eqref{eq:horbitspseq}.  Hence $v_3^{p^k - dp^{k-1}} \cdot
\theta_{k,d} = 0$, for filtration reasons, which implies that $v_3^{p^k
- dp^{k-1}} \cdot \gamma_{k,d} = 0$ since $N^h_*$ is $P(v_3)$-linear.

Finally, $\lambda_k \cdot \theta_{k,d} = 0$, because $t^{dp^{k-1}-p^k}
\cdot \lambda_k \doteq d^{2p^k}(t^{dp^{k-1}-2p^k})$ is a boundary
and by inspection of bidegrees there are no other classes in the
$E^\infty$-term of~\eqref{eq:horbitspseq} in the same total degree and
of lower filtration.  Applying $N^h_*$ we can conclude that $\lambda_k
\cdot \gamma_{k,d} = 0$.
\end{proof}

\begin{proposition} \label{prop:gammakdge4}
For each $k\ge1$ and $0<d<p$ there are elements
\begin{align*}
\gamma'_{k,d} \in \{c'_{3,d}\}
	&\subset V(2)_* (THH(BP\<2\>)^{tC_p})^{h\bT} \\
\gamma_{k+3,d} \in \{c_{k+3,d}\}
	&\subset V(2)_* THH(BP\<2\>)^{h\bT}
\end{align*}
that satisfy
\begin{align*}
v_3^{dp^{k-1}} \cdot \gamma'_{k,d}
	&= \hat\Gamma_{1*}^{h\bT}(\gamma_{k,d}) \\
GR^h_*(\gamma_{k+3,d})
	&= \gamma'_{k,d} \,.
\end{align*}
Moreover, $v_3^{r(k)} \cdot \gamma'_{k,d} = 0$ and $v_3^{r(k+3)-dp^{k+2}}
\cdot \gamma_{k+3,d} = 0$.
\end{proposition}


\begin{proof}
We proceed by induction on~$k\ge1$, starting from
Proposition~\ref{prop:gammakd123}.  By Lemma~\ref{lem:EinftyhatGamma1hT}
the image $\hat\Gamma_{1*}^{h\bT}(\gamma_{k,d}) \in V(2)_*
(THH(BP\<2\>)^{tC_p})^{h\bT}$ of the previously constructed
class~$\gamma_{k,d} \in \{c_{d,k}\}$ is detected by $t^{dp^{k-1}}
\lambda_k = (t\mu)^{dp^{k-1}} \cdot c'_{k,d}$ in $\mu^{-1}
E^\infty(\bT)$, so any initial choice of $\gamma'_{k,d} \in
\{c'_{k,d}\}$ will satisfy $v_3^{dp^{k-1}} \cdot \gamma'_{k,d}
\equiv \hat\Gamma_{1*}^{h\bT}(\gamma_{k,d})$ modulo classes of
lower filtration.  Since $\mu^{-1} E^\infty(\bT)$ is generated as a
$P(t\mu)$-module by classes in filtration~$s=0$, each nonzero class
in lower filtration than~$c_{k,d}$, but of the same total degree, is
$(t\mu)^{dp^{k-1}}$ times a class in the same total degree as~$c'_{k,d}$
and of lower filtration.  Hence the choice of~$\gamma'_{k,d}$ can be
iteratively adjusted so as to make $v_3^{dp^{k-1}} \cdot \gamma'_{k,d}
= \hat\Gamma_{1*}^{h\bT}(\gamma_{k,d})$.

It follows that $v_3^{r(k)} \cdot \gamma'_{k,d} =
v_3^{r(k)-dp^{k-1}} \cdot \hat\Gamma_{1*}^{h\bT}(\gamma_{k,d}) = 0$, since
$\hat\Gamma_{1*}^{h\bT}$ is $P(v_3)$-linear and $v_3^{r(k)-dp^{k-1}} \cdot
\gamma_{k,d} = 0$ by the inductive hypothesis.

The final choice of class~$\gamma'_{k,d}$ is still detected by
$c'_{k,d} = \lambda_{[k]} \mu^{-dp^{k-1}}$ in $C'(k,d) \subset \mu^{-1}
E^\infty(\bT)$, so by Proposition~\ref{prop:V2G}, $G_*^{-1}(\gamma'_{k,d})
\in V(2)_* THH(BP\<2\>)^{t\bT}$ is detected by $t^{dp^{k+2}}
\lambda_{[k]}$ in $\hat C(k+3,d) \subset \hat E^\infty(\bT)$.  This class
lies in negative total degree, where $E^\infty(R^h)$ is bijective by
Lemma~\ref{lem:EinftyRh}.  It follows that $R^h_*(\gamma_{k+3,d}) =
G_*^{-1}(\gamma'_{k,d})$ for a uniquely determined class $\gamma_{k+3,d}
\in V(2)_* THH(BP\<2\>)^{h\bT}$, which is detected by $c_{k+3,d}
= t^{dp^{k+2}} \lambda_{[k]}$ in $C(k+3,d) \subset E^\infty(\bT)$.

From the relation $v_3^{r(k)} \cdot \gamma'_{k,d} = 0$ and
$P(v_3)$-linearity of $G_*$ and $R^h_*$ we deduce that $v_3^{r(k)} \cdot
G_*^{-1}(\gamma'_{k,d}) = 0$ and $R^h_*(v_3^{r(k)} \cdot \gamma_{k+3,d})
= 0$.  Since $\ker(R^h_*) = \im(N^h_*)$, we can write $v_3^{r(k)} \cdot
\gamma_{k+3,d} = N^h_*(\theta_{k+3,d})$ for some $\theta_{k+3,d}$ in
degree $2p^{k+3} - 2dp^{k+2}$.  From the $\bT$-Tate differential
$$
d^{2r(k+3)}(t^{dp^{k+2}-p^{k+3}})
	\doteq t^{dp^{k+2}} (t\mu)^{r(k)} \lambda_{[k]}
	= (t\mu)^{r(k)} \cdot c_{k+3,d}
$$
we could prove that $\theta_{k+3,d} \in \{t^{dp^{k+2}-p^{k+3}}\}$
(up to a unit), but again we only need to know that $\theta_{k+3,d}$
must be detected in filtration $s \le 2p^{k+3} - 2dp^{k+2}$
in~\eqref{eq:horbitspseq}.  Hence $v_3^{p^{k+3} - dp^{k+2}} \cdot
\theta_{k+3,d} = 0$ in $V(2)_* \Sigma^2 THH(BP\<2\>)_{h\bT}$, which
implies that
$$
v_3^{r(k+3) - dp^{k+2}} \cdot \gamma_{k+3,d} = v_3^{p^{k+3} - dp^{k+2}}
	\cdot N^h_*(\theta_{k+3,d}) = 0
$$
in $V(2)_* THH(BP\<2\>)^{h\bT}$, as asserted.
\end{proof}

Recall the classes $x_{k,d}$ and~$z_{k,d}$ from Definitions~\ref{def:ABCD}
and~\ref{def:A'B'C'D'}.

\begin{corollary} \label{cor:xikd123}
For each $k \in \{1,2,3\}$ and $0<d<p$ there is a unique element
$$
\xi_{k,d} \in \{x_{k,d}\} \subset V(2)_* THH(BP\<2\>)^{h\bT}
$$
that satisfies $R^h_*(\xi_{k,d}) = 0$.  Moreover, $\lambda_k \cdot
\xi_{k,d} = 0$ and $v_3^{p^k - dp^{k-1}} \cdot \xi_{k,d} = 0$.
\end{corollary}

\begin{proof}
Let $\xi_{k,d} = \gamma_{k,d}$ as in Proposition~\ref{prop:gammakd123},
noting that $x_{k,d} = c_{k,d}$.
\end{proof}

\begin{corollary} \label{cor:xikdge4}
For each $k\ge1$ and $0<d<p$ there are unique elements
\begin{align*}
\xi_{k+3,d} \in \{x_{k+3,d}\} &\subset V(2)_* THH(BP\<2\>)^{h\bT} \\
\zeta_{k,d} \in \{z_{k,d}\} &\subset V(2)_* (THH(BP\<2\>)^{tC_p})^{h\bT}
\end{align*}
that satisfy
$$
GR^h_*(\xi_{k+3,d}) = \hat\Gamma_{1*}^{h\bT}(\xi_{k,d}) = \zeta_{k,d} \,.
$$
Moreover, $\lambda_{[k]} \cdot \xi_{k+3,d} = 0$ and $v_3^{(1-\frac{d}{p})
r(k+3)} \cdot \xi_{k+3,d} = 0$.
\end{corollary}

\begin{proof}
For $k\ge1$, choose elements $\gamma_{k+3,d}$ and~$\gamma'_{k,d}$
as in Proposition~\ref{prop:gammakdge4}.  Recalling that $x_{k+3,d} =
(t\mu)^{\frac{d}{p} r(k)} \cdot c_{k+3,d}$, we let
$$
\xi_{k+3,d} = v_3^{\frac{d}{p} r(k)} \cdot \gamma_{k+3,d} \,.
$$
Then
\begin{align*}
GR^h_*(\xi_{k+3,d})
	&= v_3^{\frac{d}{p} r(k)} \cdot GR^h_*(\gamma_{k+3,d})
	= v_3^{\frac{d}{p} r(k) - dp^{k-1}}
		\cdot v_3^{dp^{k-1}} \cdot \gamma'_{k,d} \\
	&= v_3^{\frac{d}{p} r(k-3)}
		\cdot \hat\Gamma_{1*}^{h\bT}(\gamma_{k,d})
	= \hat\Gamma_{1*}^{h\bT}(\xi_{k,d}) \,.
\end{align*}

To see that this uniquely determines~$\xi_{k+3,d} \in \{x_{k+3,d}\}$,
note that any other choice of class $\xi \in \{x_{k+3,d}\}$ with
$GR^h_*(\xi) = GR^h_*(\xi_{k+3,d})$ would differ from $\xi_{k+3,d}$
by an element~$\xi'$ in $\ker(R^h_*)$ that is detected by an element
$x'$ in $\ker E^\infty(R^h)$ of lower filtration than~$x_{k+3,d}$,
hence of filtration $s < -2(p^3+1)$.  By Lemma~\ref{lem:kerEinftyRh},
no such element~$x'$ exists in total degree~$|\xi_{k+3,d}| = 2p^{[k]}
- 2dp^{[k]-1} - 1 \le 2p^3-2$.

By induction, we know that $GR^h_*(\lambda_{[k]} \cdot \xi_{k+3,d})
= \hat\Gamma_{1*}^{h\bT}(\lambda_{[k]} \cdot \xi_{k,d}) = 0$.
Hence, if $\xi'' = \lambda_{[k]} \cdot \xi_{k+3,d}$ were nonzero,
it would be a class in $\ker(R^h_*)$, in total degree $4p^{[k]} -
2dp^{[k]-1} - 2 \le 4p^3-3$, that is detected by an element $x''$ in
$\ker E^\infty(R^h)$ of lower filtration than that of~$x_{k+3,d}$.
By Lemma~\ref{lem:kerEinftyRhbis}, treating the cases $k+3=4$ and
$k+3\ge5$ separately, no such element~$x''$ exists.  This contradiction
proves that $\lambda_{[k]} \cdot \xi_{k+3,d} = 0$.  The relation
$$
v_3^{(1-\frac{d}{p})r(k+3)} \cdot \xi_{k+3,d}
	= v_3^{r(k+3) - dp^{k+2}} \cdot \gamma_{k+3,d} = 0
$$
follows from Proposition~\ref{prop:gammakdge4}.  Finally, let
$\zeta_{k,d} = \hat\Gamma_{1*}^{h\bT}(\xi_{k,d})$, which is then detected
by~$E^\infty(\hat\Gamma_1^{h\bT})(x_{k,d}) = z_{k,d}$.
\end{proof}

We now fix compatible choices of classes $\gamma_{k,d}$
and~$\gamma'_{k,d}$, as in Propositions~\ref{prop:gammakd123}
and~\ref{prop:gammakdge4}.

\begin{definition}
For $k\ge1$ and $0<d<p$ let
$$
\wC(k,d) \cong P_{r(k)-dp^{k-1}}(v_3)
	\otimes E(\lambda_{[k+1]}, \lambda_{[k+2]})
	\otimes \bF_p\{\gamma_{k,d}\}
$$
be the $P(v_3) \otimes E(\lambda_{[k+1]}, \lambda_{[k+2]})$-submodule
of~$V(2)_* THH(BP\<2\>)^{h\bT}$ generated by~$\gamma_{k,d}$, and let
$$
\wC'(k,d) \cong P_{r(k)}(v_3)
	\otimes E(\lambda_{[k+1]}, \lambda_{[k+2]})
	\otimes \bF_p\{\gamma'_{k,d}\}
$$
be the $P(v_3) \otimes E(\lambda_{[k+1]}, \lambda_{[k+2]})$-submodule
of~$V(2)_* (THH(BP\<2\>)^{tC_p})^{h\bT}$ generated by~$\gamma'_{k,d}$.
Let
$$
\wC = \prod_{k\ge1, 0<d<p} \wC(k,d)
\qquad\text{and}\qquad
\wC' = \prod_{k\ge1, 0<d<p} \wC'(k,d) \,.
$$
These are detected by the summands $C \subset E^\infty(\bT)$ and $C'
\subset \mu^{-1} E^\infty(\bT)$, respectively.
\end{definition}

\begin{lemma}
The $P(v_3) \otimes E(\lambda_{[k+1]}, \lambda_{[k+2]})$-submodules
$$
\< \xi_{k,d} \> \subset \wC(k,d)
\qquad\text{and}\qquad
\< \zeta_{k,d} \> \subset \wC'(k,d)
$$
generated by $\xi_{k,d} = v_3^{\frac{d}{p} r(k-3)} \cdot \gamma_{k,d}$
and $\zeta_{k,d} = v_3^{\frac{d}{p} r(k)} \cdot \gamma'_{k,d}$,
respectively,
are equal to the (uniquely defined) $\wA$-submodules
generated by $\xi_{k,d}$ and~$\zeta_{k,d}$, with
\begin{align*}
\< \xi_{k,d} \> &\cong P_{(1 - \frac{d}{p}) r(k)}(v_3)
        \otimes E(\lambda_{[k+1]}, \lambda_{[k+2]})
        \otimes \bF_p\{\xi_{k,d}\} \\
\< \zeta_{k,d} \> &\cong P_{(1 - \frac{d}{p}) r(k)}(v_3)
        \otimes E(\lambda_{[k+1]}, \lambda_{[k+2]})
        \otimes \bF_p\{\zeta_{k,d}\} \,.
\end{align*}
\end{lemma}

\begin{proof}
These $P(v_3) \otimes E(\lambda_{[k+1]}, \lambda_{[k+2]})$-submodules are
$\wA$-submodules, since we proved that $\lambda_{[k]} \cdot \xi_{k,d} =
0$ in Corollaries~\ref{cor:xikd123} and~\ref{cor:xikdge4}, which readily
implies that $\lambda_{[k]} \cdot \zeta_{k,d} = 0$.
\end{proof}

\begin{remark}
With these notations, Proposition~\ref{prop:gammakdge4} shows that
$\hat\Gamma_{1*}^{h\bT}$ induces isomorphisms $\<\xi_{k,d}\> \to
\<\zeta_{k,d}\>$, and injections $\wC(k,d) \to \wC'(k,d)$
and $\wC(k,d)/\<\xi_{k,d}\> \to \wC'(k,d)/\<\zeta_{k,d}\>$.  It also
shows that $GR^h_*$ induces isomorphisms $\wC(k+3,d)/\<\xi_{k+3,d}\> \to
\wC'(k,d)/\<\zeta_{k,d}\>$, and surjections $\<\xi_{k+3,d}\>
\to \<\zeta_{k,d}\>$ and $\wC(k+3,d) \to \wC'(k,d)$, for
all $k\ge1$ and $0<d<p$.
\end{remark}

Choosing lifts of the $B$- and $D$-summands requires less precision.

\begin{definition}
For each $k\ge1$ and $p \nmid d > 0$ choose a class
$$
\beta_{k,d} \in V(2)_* THH(BP\<2\>)^{h\bT}
$$
detected by $\lambda_{[k]} \mu^{dp^{k-1}} \in B$, and let
$$
\wB(k,d) \cong P_{r(k)}(v_3)
        \otimes E(\lambda_{[k+1]}, \lambda_{[k+2]})
        \otimes \bF_p\{\beta_{k,d}\}
$$
be the $E(\lambda_{[k+1]}, \lambda_{[k+2]})$-submodule
of~$V(2)_* THH(BP\<2\>)^{h\bT}$ generated by
$v_3^m \cdot \beta_{k,d}$ for $0 \le m < r(k)$.
For each $k\ge4$ and $p \nmid d > p$ choose a class
$$
\delta_{k,d} \in V(2)_* THH(BP\<2\>)^{h\bT}
$$
detected by $t^{dp^{k-1}} \lambda_{[k]} \in D$, and let
$$
\wD(k,d) \cong P_{r(k-3)}(v_3)
        \otimes E(\lambda_{[k+1]}, \lambda_{[k+2]})
        \otimes \bF_p\{\delta_{k,d}\}
$$
be the $E(\lambda_{[k+1]}, \lambda_{[k+2]})$-submodule
of~$V(2)_* THH(BP\<2\>)^{h\bT}$ generated by
$v_3^m \cdot \delta_{k,d}$ for $0 \le m < r(k-3)$.
Let
$$
\wB = \prod_{k\ge1, p \nmid d > 0} \wB(k,d)
\qquad\text{and}\qquad
\wD = \prod_{k\ge4, p \nmid d > p} \wD(k,d) \,.
$$
These are detected by the summands~$B$ and~$D$ of $E^\infty(\bT)$,
respectively.
\end{definition}

\begin{lemma} \label{lem:diff-on-B-D}
For each $k\ge1$ and $p \nmid d > 0$ the difference
$$
(GR^h_* - \hat\Gamma_{1*}^{h\bT})(\beta_{k,d})
	\in V(2)_* (THH(BP\<2\>)^{tC_p})^{h\bT}
$$
is detected by $- \lambda_{[k]} \mu^{dp^{k-1}} \in B'$.
For each $k\ge4$ and $p \nmid d > p$ the difference
$$
(GR^h_* - \hat\Gamma_{1*}^{h\bT})(\delta_{k,d})
	\in V(2)_* (THH(BP\<2\>)^{tC_p})^{h\bT}
$$
is detected by $\lambda_{[k]} \mu^{-dp^{k-4}} \in D'$.
\end{lemma}

\begin{proof}
On one hand,
by Lemma~\ref{lem:EinftyhatGamma1hT} the image
$-\hat\Gamma_{1*}^{h\bT}(\beta_{k,d})$ is detected by $-\lambda_{[k]}
\mu^{dp^{k-1}}$ in homotopy fixed point filtration~$0$, while by
Lemma~\ref{lem:EinftyRh} and Proposition~\ref{prop:V2G} the image
$GR^h_*(\beta_{k,d})$ lies in negative filtration (or is zero).
Hence $(GR^h_* - \hat\Gamma_{1*}^{h\bT})(\beta_{k,d})$ is detected by
the filtration~$0$ class.

On the other hand,
by Lemma~\ref{lem:EinftyRh} and Proposition~\ref{prop:V2G} the image
$GR^h_*(\delta_{k,d})$ is detected by $\lambda_{[k]} \mu^{-dp^{k-4}}$
in filtration~$0$, while by Lemma~\ref{lem:EinftyhatGamma1hT} the image
$-\hat\Gamma_{1*}^{h\bT}(\delta_{k,d})$ lies in negative filtration
(or is zero).  Hence $(GR^h_* - \hat\Gamma_{1*}^{h\bT})(\delta_{k,d})$
is detected by the filtration~$0$ class.
\end{proof}

\begin{definition}
Let
$$
\wB' = (GR^h_* - \hat\Gamma_{1*}^{h\bT})(\wB)
\qquad\text{and}\qquad
\wD' = (GR^h_* - \hat\Gamma_{1*}^{h\bT})(\wD)
$$
as subgroups of~$V(2)_* (THH(BP\<2\>)^{tC_p})^{h\bT}$.  These are detected
by the summands $B'$ and~$D'$ of $\mu^{-1} E^\infty(\bT)$, respectively.
\end{definition}

\begin{proposition} \label{prop:diff-as-direct-sum}
The inclusions induce isomorphisms
\begin{align*}
V(2)_* THH(BP\<2\>)^{h\bT}
	&\cong \wA \oplus \wB \oplus \wC \oplus \wD \\
V(2)_* (THH(BP\<2\>)^{tC_p})^{h\bT}
	&\cong \wA' \oplus \wB' \oplus \wC' \oplus \wD' \,.
\end{align*}
In these terms, $GR^h_* - \hat\Gamma_{1*}^{h\bT}$ is the direct sum of
the zero homomorphism $\wA \overset{0}\to \wA'$, two isomorphisms $\wB
\overset{\cong}\to \wB'$ and~$\wD \overset{\cong}\to \wD'$, and the
difference $\Delta \: \wC \to \wC'$ between the restricted homomorphisms
\begin{align*}
GR^h_* \: \prod_{k\ge1,0<d<p} \wC(k,d)
	&\longto \prod_{k\ge1,0<d<p} \wC'(k,d) \\
(\dots, \gamma_{k,d}, \dots)
	&\longmapsto (\dots, \gamma'_{k-3,d}, \dots)
\end{align*}
and
\begin{align*}
\hat\Gamma_{1*}^{h\bT} \: \prod_{k\ge1,0<d<p} \wC(k,d)
	&\longto \prod_{k\ge1,0<d<p} \wC'(k,d) \\
(\dots, \gamma_{k,d}, \dots)
	&\longmapsto (\dots, v_3^{dp^{k-1}} \cdot \gamma'_{k,d}, \dots) \,.
\end{align*}
Here $\gamma'_{k-3,d}$ is to be interpreted as~$0$ for $k \in \{1,2,3\}$.
\end{proposition}

\begin{proof}
The submodules~$\wA$, $\wB$, $\wC$ and~$\wD$ are detected by the
direct summands $A$, $B$, $C$ and~$D$ spanning~$E^\infty(\bT)$, so
$\wA \oplus \wB \oplus \wC \oplus \wD \to V(2)_* THH(BP\<2\>)^{h\bT}$
is an isomorphism by strong convergence of the $\bT$-homotopy fixed
point spectral sequence.  Likewise, $\wA'$, $\wB'$, $\wC'$ and~$\wD'$
are detected by the direct summands $A'$, $B'$, $C'$ and~$D'$
spanning~$\mu^{-1} E^\infty(\bT)$.

The homomorphisms $GR^h_*$ and~$\hat\Gamma_{1*}^{h\bT}$ agree on $\wA$,
since the classes $v_3$, $\lambda_1$, $\lambda_2$ and~$\lambda_3$
come from algebraic $K$-theory, hence also from topological
cyclic homology.  Their difference is therefore the zero homomorphism.
The restricted homomorphisms $GR^h_* - \hat\Gamma_{1*}^{h\bT} \: \wB
\to \wB'$ and $GR^h_* - \hat\Gamma_{1*}^{h\bT} \: \wD \to \wD'$ are
isomorphisms, by the construction of the target modules, which relies
on Lemma~\ref{lem:diff-on-B-D}.  The restricted homomorphism $GR^h_* -
\hat\Gamma_{1*}^{h\bT} = \Delta \: \wC \to \wC'$ factors as asserted,
by Propositions~\ref{prop:gammakd123} and~\ref{prop:gammakdge4}.
\end{proof}

\begin{proposition}
There are $P(v_3) \otimes E(\lambda_1, \lambda_2,
\lambda_3)$-module isomorphisms
\begin{align*}
\ker(GR^h_* - \hat\Gamma_{1*}^{h\bT})
&\cong P(v_3) \otimes E(\lambda_1, \lambda_2, \lambda_3) \\
	&\qquad\oplus P(v_3) \otimes E(\lambda_2, \lambda_3)
	\otimes \bF_p\{ \Xi_{1,d} \mid 0 < d < p \} \\
	&\qquad\oplus P(v_3) \otimes E(\lambda_1, \lambda_3)
	\otimes \bF_p\{ \Xi_{2,d} \mid 0 < d < p \} \\
	&\qquad\oplus P(v_3) \otimes E(\lambda_1, \lambda_2)
	\otimes \bF_p\{ \Xi_{3,d} \mid 0 < d < p \} \\
\cok(GR^h_* - \hat\Gamma_{1*}^{h\bT})
	&\cong P(v_3) \otimes E(\lambda_1, \lambda_2, \lambda_3) \,.
\end{align*}
Here $\Xi_{i,d}$ in degree~$2p^i - 2dp^{i-1} - 1$ is detected by $x_{i,d}
= t^{dp^{i-1}} \lambda_i \in E^\infty(\bT)$, for each $i \in \{1,2,3\}$
and $0<d<p$.
\end{proposition}

\begin{proof}
Let $\Delta \: \wC \to \wC'$ be as in
Proposition~\ref{prop:diff-as-direct-sum}.  Then
\begin{align*}
\ker(GR^h_* - \hat\Gamma_{1*}^{h\bT})
	&= \wA \oplus \ker(\Delta) \\
\cok(GR^h_* - \hat\Gamma_{1*}^{h\bT})
	&= \wA' \oplus \cok(\Delta) \,.
\end{align*}
Consider the associated map of vertical short exact sequences
$$
\xymatrix{
\ds\prod_{k\ge1,0<d<p} \<\xi_{k,d}\> \ar[r]^-{\Delta'} \ar@{ >->}[d]
	& \ds\prod_{k\ge1,0<d<p} \<\zeta_{k,d}\> \ar@{ >->}[d] \\
\ds\prod_{k\ge1,0<d<p} \wC(k,d) \ar[r]^-{\Delta} \ar@{->>}[d]
	& \ds\prod_{k\ge1,0<d<p} \wC'(k,d) \ar@{->>}[d] \\
\ds\prod_{k\ge1,0<d<p} \frac{\wC(k,d)}{\<\xi_{k,d}\>}
	\ar[r]^-{\Delta''}_-{\cong}
	& \ds\prod_{k\ge1,0<d<p} \frac{\wC'(k,d)}{\<\zeta_{k,d}\>}
	\rlap{\,.}
}
$$
In the upper row, the $\hat\Gamma_{1*}^{h\bT} \: \<\xi_{k,d}\> \to
\<\zeta_{k,d}\>$ for $k\ge1$ and $0<d<p$ are isomorphisms, so we can
identify $\ker(\Delta')$ with the product over
$i \in \{1,2,3\}$ and $0<d<p$ of the limit of the sequence
$$
\dots \longto \<\xi_{k+3,d}\>
	\overset{(\hat\Gamma_{1*}^{h\bT})^{-1} GR^h_*}\longto \<\xi_{k,d}\>
	\longto \dots \longto \<\xi_{i+3,d}\>
	\overset{(\hat\Gamma_{1*}^{h\bT})^{-1} GR^h_*}\longto \<\xi_{i,d}\> \,,
$$
where $k \equiv i \mod 3$.  Since
$$
(\hat\Gamma_{1*}^{h\bT})^{-1} GR^h_* \: \xi_{k+3,d}
	\longmapsto \xi_{k,d} \,,
$$
this limit is isomorphic, as an $\wA$-module, to $P(v_3) \otimes
E(\lambda_{[i+1]}, \lambda_{[i+2]}) \otimes \bF_p\{\Xi_{i,d}\}$, with
$$
\Xi_{i,d} = (\dots, 0, \xi_{k+3,d}, 0, 0, \xi_{k,d}, 0, \dots)
$$
detected by $x_{i,d}$ in $E^\infty(\bT)$.  Similarly, we can identify
$\cok(\Delta')$ with the (right) derived limit of this sequence, which
vanishes because each $GR^h_* \: \<\xi_{k+3,d}\> \to \<\zeta_{k,d}\>$
is surjective.

In the lower row, $\ker(\Delta'') = 0$ and $\cok(\Delta'') = 0$ because
$\wC(i,d) / \<\xi_{i,d}\> = 0$ for $i \in \{1,2,3\}$ and the $GR^h_*
\: \wC(k+3,d) / \<\xi_{k+3,d}\> \to \wC'(k,d) / \<\zeta_{k,d}\>$ are
isomorphisms.  Taken together, this proves that
$$
\ker(\Delta) = \ker(\Delta')
	\cong \prod_{i \in \{1,2,3\}, 0<d<p} P(v_3)
	\otimes E(\lambda_{[i+1]}, \lambda_{[i+2]})
	\otimes \bF_p\{\Xi_{i,d}\}
$$
and $\cok(\Delta) = 0$.
\end{proof}

\begin{theorem} \label{thm:TCBP2}
Let $p\ge7$.  There is a preferred $P(v_3) \otimes E(\lambda_1,
\lambda_2, \lambda_3)$-module isomorphism
\begin{align*}
V(2)_* TC(BP\<2\>) &\cong P(v_3)
	\otimes E(\partial, \lambda_1, \lambda_2, \lambda_3) \\
	&\qquad\oplus P(v_3) \otimes E(\lambda_2, \lambda_3)
	\otimes \bF_p\{ \Xi_{1,d} \mid 0 < d < p \} \\
	&\qquad\oplus P(v_3) \otimes E(\lambda_1, \lambda_3)
	\otimes \bF_p\{ \Xi_{2,d} \mid 0 < d < p \} \\
	&\qquad\oplus P(v_3) \otimes E(\lambda_1, \lambda_2)
	\otimes \bF_p\{ \Xi_{3,d} \mid 0 < d < p \} \,,
\end{align*}
with $\Xi_{i,d}$ detected by $x_{i,d} = t^{dp^{i-1}} \lambda_i$ for $i \in
\{1,2,3\}$ and $0<d<p$.  Here $|v_3| = 2p^3-2$, $|\lambda_i| = 2p^i-1$,
$|\partial| = -1$ and $|t| = -2$.  This is a free $P(v_3)$-module on
the $16 + 12(p-1) = 12p+4$ generators
$$
\partial^{\epsilon} \lambda_1^{\epsilon_1}
	\lambda_2^{\epsilon_2} \lambda_3^{\epsilon_3}
\ ,\ 
\lambda_2^{\epsilon_2} \lambda_3^{\epsilon_3} \Xi_{1,d}
\ ,\ 
\lambda_1^{\epsilon_1} \lambda_3^{\epsilon_3} \Xi_{2,d}
\ ,\ 
\lambda_1^{\epsilon_1} \lambda_2^{\epsilon_2} \Xi_{3,d}
$$
in degrees $-1 \le * \le 2p^3+2p^2+2p-3$, where
$\epsilon, \epsilon_i \in \{0,1\}$ and $0<d<p$.
\end{theorem}

\begin{proof}
The definition of $TC(BP\<2\>)$ as the homotopy equalizer of
$\hat\Gamma_1^{h\bT}$ and~$GR^h$ leads to the short exact sequence
$$
0 \to \Sigma^{-1} \cok(GR^h_* - \hat\Gamma_{1*}^{h\bT})
	\overset{\partial}\longto
	V(2)_* TC(BP\<2\>) \overset{\pi}\longto
	\ker(GR^h_* - \hat\Gamma_{1*}^{h\bT}) \to 0 \,.
$$
It splits as an extension of $P(v_3) \otimes E(\lambda_1, \lambda_2,
\lambda_3)$-modules, since the image of $\partial$ is trivial in the
(even) degrees of the products $\lambda_i \cdot \Xi_{i,d}$ that vanish
on the right-hand side.  The splitting is unique, since the left-hand
side is trivial in the (zero or odd) degrees of the module generators~$1$
and $\Xi_{i,d}$.
\end{proof}

\begin{corollary}
The classes $\alpha_1$, $\beta'_1$ and~$\gamma''_1 \in \pi_* V(2)$
map under the unit map $S \to TC(BP\<2\>)$ to the classes $\Xi_{1,1}$,
$\Xi_{2,1}$ and~$\Xi_{3,1}$, respectively.
\end{corollary}

\begin{proof}
These elements are detected, in pairs, by $t \lambda_1$, $t^p \lambda_2$
and $t^{p^2} \lambda_3$ in $E^\infty(\bT)$, and in these (total)
degrees there are no other classes of lower filtration, nor in the
image of~$\partial$.
\end{proof}

Thanks to the Nikolaus--Scholze model for $TC(B)$ we no longer need
to recover $V(2)_* TC(BP\<2\>)$ from $V(2)_* TC(BP\<1\>)$ in low
degrees, but we nonetheless have the following consistency result.

\begin{proposition} \label{prop:TCBP2-to-TCBP1}
The $E_3$ $BP$-algebra map $BP\<2\> \to BP\<1\>$ induces a
$(2p^2-1)$-connected surjective ring homomorphism
\begin{align*}
V(2)_* TC(BP\<2\>) \longto V(2)_* TC(BP\<1\>)
	&\cong E(\partial, \lambda_1, \lambda_2) \\
	&\qquad\oplus E(\lambda_2) \otimes \bF_p\{\Xi_{1,d} \mid 0<d<p\} \\
	&\qquad\oplus E(\lambda_1) \otimes \bF_p\{\Xi_{2,d} \mid 0<d<p\} \,,
\end{align*}
mapping $\partial$, $\lambda_1$, $\lambda_2$, $\Xi_{1,d}$ and $\Xi_{2,d}$
to the classes with the same names, and mapping $v_3$, $\lambda_3$
and~$\Xi_{3,d}$ to zero.
\end{proposition}

\begin{proof}
This is clear for~$\partial$, $\lambda_1$ and~$\lambda_2$.  Moreover,
$\Xi_{1,d}$ and $\Xi_{2,d}$ in $V(2)_* TC(BP\<2\>)$ map to classes
in $V(2)_* THH(BP\<1\>)^{h\bT}$ that are detected by $t^d \lambda_1$
and $t^{dp} \lambda_2$, respectively, which characterizes their
images in $V(2)_* TC(BP\<1\>)$.  The classes $v_3$, $\lambda_3$
and $\Xi_{3,d}$ for $1 \le d \le p-2$ are mapped to trivial groups.
Finally, $\Xi_{3,p-1}$ in degree~$2p^2-1$ maps to zero in $V(2)_*
THH(BP\<2\>)$, hence cannot be detected by~$\lambda_2$.
\end{proof}

We write $BP\<2\>_p$ for the $p$-completion of the $p$-local $E_3$
ring spectrum~$BP\<2\>$.

\begin{theorem} \label{thm:KBP2p}
Let $p\ge7$.  There is an exact sequence
\begin{multline*}
0 \to \Sigma^{-2} \bF_p\{\bar\tau_1, \bar\tau_2, \bar\tau_1\bar\tau_2\}
	\longto V(2)_* K(BP\<2\>_p) \\
	\overset{trc_*}\longto V(2)_* TC(BP\<2\>)
	\longto \Sigma^{-1} \bF_p\{1\} \to 0 \,.
\end{multline*}
Hence $V(2)_* K(BP\<2\>_p)$ is the direct sum of a free $P(v_3)$-module
on $12p+4$ generators in degrees $0 \le * \le 2p^3+2p^2+2p-3$, plus
an $\bF_p$-module with trivial $v_3$-action spanned by three classes in
degrees~$2p-3$, $2p^2-3$ and $2p^2+2p-4$.  In particular, the localization
homomorphism
$$
V(2)_* K(BP\<2\>_p) \longto v_3^{-1} V(2)_* K(BP\<2\>_p)
$$
is an isomorphism in degrees $* \ge 2p^2+2p$.
\end{theorem}

\begin{proof}
By~\cite{Dun97} and~\cite{HM97}*{Thm.~D} there is a homotopy cofiber
sequence
$$
K(BP\<2\>_p)_p \overset{trc}\longto TC(BP\<2\>)_p
	\overset{\varpi}\longto \Sigma^{-1} H\bZ_p \,,
$$
hence also a long exact sequence
$$
\dots \to V(2)_* K(BP\<2\>_p) \overset{trc_*}\longto V(2)_* TC(BP\<2\>)
	\overset{\varpi_*}\longto V(2)_*(\Sigma^{-1} H\bZ_p) \to \dots \,.
$$
Here $V(2)_*(H\bZ_p) \cong E(\bar\tau_1, \bar\tau_2)$ with
$|\bar\tau_1|=2p-1$ and $|\bar\tau_2|=2p^2-1$.  The only
$P(v_3)$-module generator of~$V(2)_* TC(BP\<2\>)$ that is mapped
nontrivially by~$\varpi_*$ is~$\partial$, with $\varpi_*(\partial)
\doteq \Sigma^{-1} 1$.  The generators $\partial \lambda_1$, $\partial
\lambda_2$ and $\partial \lambda_1 \lambda_2$ come from $V(0)$-homotopy,
hence factor through $V(0)_*(\Sigma^{-1} H\bZ_p)$, and therefore map to
zero.  The generator $\lambda_1 \Xi_{2,1}$ is the product of two classes
in the image of $trc_*$, hence also maps to zero under~$\varpi$.
It follows that $\ker(\varpi_*)$ is freely generated as a $P(v_3)$-module
by the same generators as for $V(2)_* TC(BP\<2\>)$, except that $\partial$
in degree~$-1$ is replaced by $v_3 \partial$ in degree~${2p^3-3}$.
\end{proof}

\begin{theorem} \label{thm:KBP2}
The $p$-completion map $\kappa \: BP\<2\> \to BP\<2\>_p$ induces a
$(2p^2+2p-2)$-coconnected homomorphism
$$
V(2)_* K(BP\<2\>) \overset{\kappa_*}\longto V(2)_* K(BP\<2\>_p) \,.
$$
Hence $V(2)_* K(BP\<2\>)$ is the direct sum of a free $P(v_3)$-module
on $12p+4$ generators in degrees $0 \le * \le 2p^3+2p^2+2p-3$, plus an
$\bF_p$-module with trivial $v_3$-action concentrated in degrees~$1 \le *
\le 2p^2+2p-3$.  In particular,
$$
V(2)_* K(BP\<2\>) \longto v_3^{-1} V(2)_* K(BP\<2\>)
$$
is an isomorphism in degrees $* \ge 2p^2+2p$.
\end{theorem}

\begin{proof}
We know that $V(1)_* K(\bQ)$ and $V(1)_* K(\bQ_p)$ are concentrated in
degrees $0 \le * \le 2p-2$, by the proven Lichtenbaum--Quillen/Bloch--Kato
conjectures~\cite{Voe11}, and the earlier calculation of $V(0)_* TC(\bZ)$
from~\cite{BM94},~\cite{BM95}.  Hence $V(1) \wedge K(\bQ) \to V(1) \wedge
K(\bQ_p)$ is $(2p-1)$-coconnected.  It follows from the localization
sequence in algebraic $K$-theory that $V(1) \wedge K(\bZ_{(p)}) \to
V(1) \wedge K(\bZ_p)$ is also $(2p-1)$-coconnected, so that $V(2) \wedge
K(\bZ_{(p)}) \to V(2) \wedge K(\bZ_p)$ is $(2p^2+2p-2)$-coconnected.
By the commutative cube
$$
\xymatrix@-0.5pc{
K(BP\<2\>)_p \ar[rr]^-{\kappa} \ar[dr] \ar[dd]_-{trc}
	&& K(BP\<2\>_p)_p \ar[dr] \ar[dd]_(0.3){trc}|\hole \\
& K(\bZ_{(p)})_p \ar[rr]^(0.3){\kappa} \ar[dd]_(0.3){trc}
	&& K(\bZ_p)_p \ar[dd]^-{trc} \\
TC(BP\<2\>)_p \ar[rr]^(0.3){\simeq}|\hole \ar[dr]
	&& TC(BP\<2\>_p)_p \ar[dr] \\
& TC(\bZ_{(p)})_p \ar[rr]^-{\simeq} && TC(\bZ_p)_p
}
$$
and~\cite{Dun97} applied to the left hand and right hand faces, it
follows that $V(2) \wedge K(BP\<2\>) \to V(2) \wedge K(BP\<2\>_p)$
is also $(2p^2+2p-2)$-coconnected.  This implies the assertions above.
\end{proof}

\end{document}